\documentclass[11pt,a4paper]{amsart} 
\calclayout

\usepackage{tikz-cd}
\usepackage{graphicx}

\usepackage{pgfplots}
\usepackage{mathrsfs}
\usepackage{microtype}

\usepackage{pgfplots}
\pgfplotsset{compat=1.15}
\usepackage{tikz}
\usetikzlibrary{arrows}

\usepackage{stmaryrd}
\usepackage{slashed}
\usepackage{tikz-cd}
\usetikzlibrary{positioning}
\usetikzlibrary{calc,intersections}
\usepackage{hyperref}
\usepackage{todonotes}

\usepackage[all,cmtip]{xy}

\usepackage[justification=centering]{caption}

\usepackage[english]{babel}

\usepackage{csquotes}
\usepackage{amssymb}
\usepackage{amsmath}
\usepackage{amsthm}
\usepackage{xcolor}
\newtheorem{thm}{Theorem}[section]
\newtheorem{lem}[thm]{Lemma}
\newtheorem{prop}[thm]{Proposition}
\newtheorem{coro}[thm]{Corollary}

\theoremstyle{remark}
\newtheorem{rema}[thm]{Remark}
\newtheorem{exa}[thm]{Example}
\newtheorem{defi}[thm]{Definition}

\DeclareMathOperator{\gr}{Gr}
\DeclareMathOperator{\sym}{Sym}

\DeclareMathOperator{\Lie}{Lie}
\DeclareMathOperator\pr{pr}
\DeclareMathOperator{\der}{Der}

\newcommand{\aaa}{\rotatebox[origin=r]{-90}{\scalebox{0.52}{$\sim$}}}

\DeclareFontFamily{U} {MnSymbolA}{}
\DeclareFontShape{U}{MnSymbolA}{m}{n}{
  <-6> MnSymbolA5
  <6-7> MnSymbolA6
  <7-8> MnSymbolA7
  <8-9> MnSymbolA8
  <9-10> MnSymbolA9
  <10-12> MnSymbolA10
  <12-> MnSymbolA12}{}
\DeclareFontShape{U}{MnSymbolA}{b}{n}{
  <-6> MnSymbolA-Bold5
  <6-7> MnSymbolA-Bold6
  <7-8> MnSymbolA-Bold7
  <8-9> MnSymbolA-Bold8
  <9-10> MnSymbolA-Bold9
  <10-12> MnSymbolA-Bold10
  <12-> MnSymbolA-Bold12}{}

\newcommand{\cA}{\mathcal{A}}
\newcommand{\cB}{\mathcal{B}}

\newcommand{\cD}{\mathcal{D}}

\newcommand{\cG}{\mathcal{G}}
\newcommand{\cH}{\mathcal{H}}

\newcommand{\cK}{\mathcal{K}}
\newcommand{\cL}{\mathcal{L}}
\newcommand{\cM}{\mathcal{M}}
\newcommand{\cN}{\mathcal{N}}

\newcommand{\cV}{\mathcal{V}}

\newcommand{\fg}{\mathfrak{g}}
\newcommand{\fh}{\mathfrak{h}}

\newcommand{\fX}{\mathfrak{X}}

\newcommand{\qq}{q}

\newcommand{\gs}{\mathtt{s}}
\newcommand{\gt}{\mathtt{t}}
\newcommand{\gu}{\mathtt{u}}
\newcommand{\gi}{\mathtt{i}}
\newcommand{\gm}{\mathtt{m}}

\newcommand{\man}[1]   {#1\text{-}\mathcal{M}an}

\newcommand{\bN}{\mathbb{N}}
\newcommand{\bR}{\mathbb{R}}
\newcommand{\bZ}{\mathbb{Z}}

\newcommand{\pair}[2]{\langle #1, #2 \rangle}
\newcommand{\cbrack}[2]{\llbracket #1, #2 \rrbracket}

\DeclareSymbolFont{MnSyA} {U} {MnSymbolA}{m}{n}
\DeclareMathSymbol{\lcirclearrowright}{\mathrel}{MnSyA}{252}
\DeclareMathSymbol{\rcirclearrowleft}{\mathrel}{MnSyA}{250}
\DeclareMathOperator{\id}{Id}
\usepackage{stmaryrd}
\usepackage{slashed}
\usepackage{tikz}
\usetikzlibrary{positioning}
\usetikzlibrary{calc,intersections}
\usetikzlibrary{matrix,arrows}
\usepackage{todonotes}

\let\oldtocsection=\tocsection

\let\oldtocsubsection=\tocsubsection

\let\oldtocsubsubsection=\tocsubsubsection

\renewcommand{\tocsection}[2]{\bf\hspace{0em}\oldtocsection{#1}{#2}}
\renewcommand{\tocsubsection}[2]{\hspace{1em}\oldtocsubsection{#1}{#2}}
\renewcommand{\tocsubsubsection}[2]{\hspace{2em}\oldtocsubsubsection{#1}{#2}}

\title{Homological vector fields over differentiable stacks}
\author{Daniel \'Alvarez}
\address{Departament of Mathematics, University of Toronto, 40 St. George St., Toronto, Ontario, M5S 2E4}
\email{dalv@math.toronto.edu}

\author{Miquel Cueca}
\address{Mathematics Institute\\Georg-August-University of G\"ottingen\\Bunsenstra{\ss}e 3-5\\G\"ottingen 37073\\Germany}
\email{miquel.cuecaten@mathematik.uni-goettingen.de}

\keywords{Differentiable stacks, Homological vector fields, Lie algebroids} 
\subjclass{18F20, 58A50, 53D17, 58H05}

\date{} 

\begin{document}
\begin{abstract} In this work we solve the problem of providing a Morita invariant definition of Lie and Courant algebroids over Lie groupoids. By relying on supergeometry, we view these structures as instances of vector fields on graded groupoids which are homological up to homotopy. We describe such vector fields in general from two complementary viewpoints: firstly, as Maurer-Cartan elements in a differential graded Lie algebra of multivector fields and, secondly, we also view them from a categorical approach, in terms of functors and natural transformations. Thereby, we obtain a unifying conceptual framework for studying LA-groupoids, $L_2$-algebroids (including semistrict Lie 2-algebras and 2-term representations up to homotopy), infinitesimal gerbe prequantizations, higher gauge theory (specifically, 2-connections on 2-bundles), quasi-Poisson groupoids and (twisted) multiplicative Courant algebroids.
\end{abstract}

\maketitle
\tableofcontents
 
\section{Introduction} It is usually desirable in both mathematics and physics to describe geometric structures in a flexible enough manner that they can be transported under certain change of frame of reference operations. More specifically, whenever we want to define a geometric structure over a differentiable stack, understood as a class of {\em weakly equivalent} Lie groupoids \cite{difger}, we have to make sure it is such that it can be transported under Morita equivalences. It has been realized that the appropriate way to do so is by providing definitions that rely on homotopies or natural isomorphisms instead of equalities. In this context, the program of {\em categorifying} several fundamental concepts is particularly relevant: see \cite{grp,bacr,higyanmil,higauthe,bone:1shp,xu:mom} 
for a discussion of the ideas of Lie group, Lie algebra, principal bundle, parallel transport and Poisson and symplectic manifolds in this spirit.

 For Lie algebras, in particular, two generalizations naturally appeared: (1) one can define skew-symmetric brackets that satisfy the Jacobi identity {\em up to chain homotopy} in a complex of vector spaces, this is the idea behind $L_\infty$-algebras; (2) on the other hand, we can think of a categorified Lie algebra as a category itself which is endowed with a Lie bracket that is a functor and that satisfies the Jacobi identity {\em up to a natural isomorphism}. Both approaches were illustrated and contrasted in \cite{bacr}.
 
In this work we further develop this program by examining the idea of {\em Lie algebroid}.  Defining Lie algebroids up to homotopy as $L_2$-algebroids (or split degree 2 $Q$-manifolds) is straightforward, see for instance \cite{bon:on, sev:some}. On the other hand, a partial categorification of Lie algebroids was achieved by the introduction of LA-groupoids \cite{macdou}; however, the properties of such objects are still characterized by equalities instead of natural isomorphisms and so they are not Morita invariant. Notwithstanding this fact, there have been some attempts to use LA-groupoids to define Lie algebroids over differentiable stacks \cite{rog:lie, wal} (see also \cite{ be:alg} for a study of Lie algebroids and stacks in the algebraic category). 

The goal of this work is to unify the study of those two classes of objects in a way that sheds light on many concrete examples that have appeared in the literature. In doing so we aim to give a framework which is concrete enough to perform explicit computations and yet is sufficiently general to explain and put into a suitable context a number of results and constructions about homotopic or categorified notions in classical differential geometry. Our perspective encompasses and clarifies diverse notions of Lie bialgebroids \cite{cc:bia,bur:dou, xu:2bia,kra:bia}, the emergence of Lie 2-algebra structures on the spaces of multiplicative sections of LA-groupoids and quasi Poisson groupoids \cite{l2algqpoi,strhom, ortwal}, the definition of a flat 2-connection on a 2-bundle in higher gauge theory \cite{higyanmil, higauthe, difgeoger}, and Courant algebroids over groupoids \cite{tracou, raj:qgr,muldir} (possibly twisted as in \cite{stro:4form}). 

Our interpretation of $L_2$-algebroids in categorical terms 
sheds light on the integration of 2-term representations up to homotopy via 2-functors \cite{int2ter} (the  version of Lie's second theorem relevant for this interpretation is discussed in \cite{cc:lie2}). 
We plan to examine these global aspects of the theory further in future work.

\subsection{Summary} Due to the length of this work, we include here a brief description of our contributions. Over the last few decades the language of supergeometry has led to the unified study of many different geometric structures such as Poisson manifolds \cite{sch:geoBV}, Lie algebroids \cite{vai:lie}, Lie bialgebroids and Courant algebroids \cite{roy:on}; all these objects can be viewed as {\em $Q$-manifolds}, i.e. $\bN$-graded manifolds endowed with a vector field $Q$ that squares to zero; in this situation, $Q$ is called a {\em $Q$-structure or a homological vector field.} 

It has long been realized that Lie groupoids are very convenient models for differentiable stacks, see  \cite{difger} and compare it with the original algebraic approach of \cite{gro:SG4}. In the light of this, we can define a \emph{graded differentiable stack}, i.e. a smooth stack in the category of $\mathbb{N}$-graded manifolds, as the Morita equivalence class of a graded Lie groupoid as in \cite{raj:tes, raj:qgr}. 
So the first question that we address is the following 
\begin{align}
    \text{What is a homological vector field on a graded differentiable stack?}\tag{Q1}\label{q1}
\end{align}

We give an answer to this question in \S \ref{sec:qqgpd} by recalling the general definition of graded groupoids as in \cite{raj:tes, raj:qgr} and introducing degree 1 multiplicative vector fields $Q$ over them that are homological {\em up to homotopy} in the sense that $Q^2$ is equal to a simplicially exact term. This is in contrast with the $Q$-groupoids introduced in \cite{raj:tes, raj:qgr} where $Q^2$ is required to be equal to 0. The resulting condition immediately characterizes $Q$ as part of a Maurer-Cartan element in the dgla of multiplicative multivector fields on a groupoid (see $\S\ref{sec:MC}$) and hence we get that it is a structure that can be transported along Morita equivalences, providing thus the appropriate concept of $Q$-structure over a differentiable stack. We call these objects {\em quasi $Q$-groupoids} (Definition \ref{def:qqgpd}). 

We also characterize the above Maurer-Cartan elements as natural transformations in \S \ref{subsec:qqnat}. This equivalent viewpoint allows us to view these objects as genuine categorifications of $Q$-manifolds, in the same way that Lie 2-algebras are a categorification of Lie algebras \cite{bacr}. This categorical viewpoint naturally leads to our definition of morphism between quasi $Q$-groupoids in \S \ref{sec:mor}. Let us just point out that this concept extends the idea of gauge transformation between Maurer-Cartan elements and includes morphisms of semistrict Lie 2-algebras \cite{bacr} as particular examples. 

Poisson structures have been frequently studied in the presence of a homological vector field, most famously under the guise of {\em $QP$-manifolds} \cite{aksz}. So another closely related question to \eqref{q1} is the following:
\begin{align}
    \text{What is a $QP$ manifold structure on a graded differentiable stack?}\label{q2}\tag{Q2}
\end{align}
We give an answer to this question in Definition \ref{def:deg2sym}, drawing inspiration from homotopy and shifted Poisson structures \cite{bone:1shp, raj:hop, pridh:hop}. We pay particular emphasis to non degenerate Poisson structures as that situation is what allows us to consider some of the most interesting examples as we shall see below.

In the rest of the paper we specialize our answers to \eqref{q1} and \eqref{q2} in many situations. Since Lie algebroid structures are equivalent to degree 1 $Q$-manifolds \cite{vai:lie}, for graded groupoids of degree 1, \eqref{q1} becomes the following question: \begin{align}
    \text{What is a Lie algebroid over a differentiable stack?}\label{q3}\tag{Q3}
\end{align}
Our answer to this question in \S \ref{sec:1q-groupoids} specializes, on one hand, to LA-groupoids \cite{macdou} viewed as $Q$-groupoids \cite{raj:tes} and, on the other, to $L_2$-algebroids, see Theorem \ref{l2algqla}. We refer to the categorified Lie algebroids thus obtained as {\em quasi LA-groupoids}. Our viewpoint allows us to consider many examples: in fact, we show that quasi LA-groupoids  appear as central extensions determined by closed 1-shifted 2-forms and quasi Poisson groupoids structures, see \S \ref{subsec:cenext}. Furthermore, we explain how 2-term representations up to homotopy \S \ref{subsec:ruth} and flat 2-connections are examples of quasi LA-groupoid morphisms \S \ref{sec:ati}. 

Lie algebroid structures are in duality with respect to linear Poisson structures. In \S \ref{sec:qPVB} we promote that relationship to a duality between our quasi LA-groupoids and linear quasi Poisson structures \cite{quapoi}. Theorem \ref{thm:main} generalizes the main result of \cite{strhom} and allows us to study the categories of multiplicative sections of such objects in a straightforward manner; as an application, we obtain the Morita invariance of such categories in \S \ref{subsec:mulsec} (which was proven for LA-groupoids in \cite{ortwal}). We also use the duality thus established to develop the Lie theory of quasi LA-groupoids, following the differentiation and integration results for LA-groupoids discussed in \cite{burcabhoy}.  Moreover, we also treat $L_2$-algebroids, generalizing the equivalence proven in \cite{bacr} between Lie 2-algebras and 2-term $L_\infty$-algebras, see \S \ref{subsec:l2alg}. 

Question \eqref{q3} naturally leads to a subquestion:
\begin{align}
    \text{What is a Lie bialgebroid over a differentiable stack?}\label{q4}\tag{Q4}
\end{align}
Since a Lie bialgebroid can be described as a homological vector field together with a compatible Poisson bracket, our answer to \eqref{q2} leads naturally to an answer to \eqref{q4} and hence to the study of {\em multiplicative Lie bialgebroids} \S \ref{subsec:catliebia}. As we pointed out before, we obtain then a clear general picture for a number of constructions in \cite{cc:bia,bur:dou, xu:2bia,kra:bia}. In particular, when the relevant Poisson bracket is symplectic, our multiplicative Lie bialgebroids
 become the cotangent groupoids of quasi Poisson groupoids \S \ref{subsec:cotqpoi}. Cotangents of quasi Poisson groupoids are then categorified $QP$-manifolds of degree 1. As a result, we  clarify some of the observations in \cite{l2algqpoi} about the presence of a Lie 2-algebra structure on the category of multiplicative 1-forms of a quasi Poisson groupoid. 
 
 In \S \ref{sec:2q-groupoids} we study quasi $Q$-groupoids of degree two, answering the questions
 \begin{align}
    &\text{What is a Courant algebroid over a differentiable stack?}\label{q5}\tag{Q5}\\
    &\text{What is the double of a multiplicative Lie bialgebroid?}\label{q6}\tag{Q6}
\end{align}
The work in this section then promotes the ideas of \cite{liuweixu,roycou} to the setting of differentiable stacks. Thus our definition of {\em quasi CA-groupoids} can be regarded as a categorification of the idea of $QP$-manifolds of degree 2 and it includes as examples exact (twisted) CA-groupoids as in \cite{tracou, raj:qgr,muldir} and also the doubles of the multiplicative Lie bialgebroids described in \S \ref{subsec:catliebia}.
 
 In Appendix \S \ref{sec:douqpoi} we include a description of double quasi Poisson groupoids which are some of the global counterparts of the multiplicative Lie bialgebroids of \S \ref{subsec:catliebia}. In this regard, we generalize the correspondence between quasi-Manin triples of Lie 2-algebras and quasi Poisson 2-groups \cite{poi2gro} by showing that quasi Poisson LA-groupoids over a unit manifold are indeed integrable to a double quasi Poisson groupoid under natural conditions.

\subsection{Future outlook} 
The work presented here motivates a number of questions that mostly have to do with the global aspects of the theory. Now we shall briefly elaborate on this.

Our concept of quasi-LA groupoids should be well suited to formulate the integrability condition of quasi-Nijenhuis tensors over Lie groupoids, see \cite[Remark  6.4]{cla:dir}. In particular, one should be able to define quasi-complex structures over Lie groupoids using these objects, thus providing a suitable concept of complex structure on a differentiable stack. In the same spirit, one can use the quasi-CA groupoids that we define here as background for generalized complex structures up to homotopy. We hope to analyse these ideas in particular examples in the future.

The notions introduced here can be significantly generalized. In fact, it is natural to consider homological vector fields over graded differential higher (and derived)  stacks. To do so one can work with graded simplicial manifolds and shifted multivector fields as in \cite{pridh:hop}. Notably, this would allow us to include $L_\infty$-algebroids in full generality instead of their truncated versions. Also, $n$-shifted Poisson and symplectic structures should give other natural examples of these higher Lie algebroids via central extensions as in \S \ref{subsec:cenext}. However, working in such generality makes difficult to obtain explicit examples and formulae which was one of our objectives in this project. 

In a different direction, one can try to find the global counterparts of quasi $Q$-groupoids and their morphisms. As conceptual guidance for this integration problem, one can look at three different particular cases. (1) Firstly, we have the integration of representations up to homotopy \cite{intruth,int2ter}. Since 2-term representations up to homotopy can be viewed as an example of quasi LA-groupoid morphism, there should be a general relationship between the integration of quasi LA-groupoids morphisms and two-dimensional holonomy, as suggested by the results in \cite{highol, int2ter, surhol}. We believe that a better understanding of this relationship would allow us to study Wilson surface observables \cite{surhol} in greater generality and also in contexts more closely related to (higher) symplectic structures. (2) Secondly, it is well known that double Lie groupoids are the integrations of LA-groupoids \cite{macdou} but there has been little progress integrating these objects, see \cite{morla}. We hope that combining our more flexible concept of morphism for (quasi) LA-groupoids with Sullivan's spatial realization (as applied in \cite{sev:some}) can shed some light on this problem. (3) Finally, we expect that integrating quasi LA-groupoids over unit groupoids should produce stacky Lie groupoids as in \cite{cc:ngro}.

On a related note, one can try to quantize the categorified $QP$-manifolds described here. In other words, it seems interesting to study the AKSZ construction \cite{aksz} but adapted to the setting of graded differentiable stacks, in the spirit of \cite{cal:aksz}. In particular, we expect that categorifying the usual Poisson sigma model should give us a way of quantizing and integrating quasi Poisson groupoids, in the same way that Poisson manifolds are quantized and integrated in \cite{catfelqua,catfel}. We hope that progress in this direction can also shed light on the classical AKSZ sigma models from a global perspective.

\vspace{3mm}

\noindent{\bf Acknowledgements.} We thank Henrique Bursztyn, Alejandro Cabrera, Rajan Mehta, Cristian Ortiz and Chenchang Zhu for fruitful discussions  and suggestions on previous drafts of this paper. D. A. thanks all the members of the Geometric Structures Laboratory at the University of Toronto for countless inspiring conversations. M.C. also thanks all the members of the Higher Structure Seminar at G\"{o}ttingen. 

\section{Quasi Q-groupoids}\label{sec:qqgpd} 

In this section we introduce quasi $Q$-groupoids, a generalization of the $Q$-groupoids introduced in \cite{raj:qgr}.  We approach these objects from two complementary viewpoints: one in terms of functors and natural transformations, the other one, in terms of Maurer-Cartan elements. This last description allows us to study their compatibility with multiplicative homotopy Poisson structures. We conclude this section by showing that quasi $Q$-groupoids form a strict $2$-category.

As motivation for the general theory of graded manifolds, let us emphasize its remarkable power to express classical differential geometric ideas in a concise and conceptual manner, see \S \ref{sec:fun1} and \S \ref{sec:RS} for the respective applications of supergeometry to Lie and Courant algebroids.

\subsection{Graded manifolds}
Given a non-negative integer $n\in\bN$, let $\mathbf{V}=\oplus_{i=1}^n V_{i}$ be a graded vector space.
An \emph{$\mathbb{N}$-manifold of degree $n$} (or simply an \emph{$n$-manifold})  is a ringed space $\cM=(M, C^\bullet(\cM))$, where $M$ is a smooth manifold and $C^\bullet(\cM)$ is a sheaf of graded commutative algebras such that, for all $x\in M$, there exists a neighborhood $U$ and an isomorphism
    \begin{equation}\label{locality}
        C^\bullet(\cM)|_U\cong C^{\infty}(U)\otimes \sym \mathbf{V}
    \end{equation}
    of sheaves of graded commutative algebras, where  $\sym \mathbf{V}$ denotes the graded symmetric algebra of $\mathbf{V}$. The manifold $M$ is known as the \emph{body} of $\cM$. A global section of $C^i(\cM)$ is called a \emph{homogeneous function of degree $i$}; we write $|f|=i$ for the degree of $f$. 

A morphism of $n$-manifolds $\Psi:\cM\to \cN$ is a degree preserving morphism of ringed spaces, given by a pair $\Psi=(\psi, \psi^\sharp)$, where $\psi:M \to N$ is a smooth map and $\psi^\sharp:C^\bullet(\cN)\to\psi_* C^\bullet(\cM)$ is a degree preserving morphism of sheaves of algebras over $N$. 

For each $n \in \bN$, $n$-manifolds with their morphisms form a category that we denote by $\man{n}$. Now we collect some basic examples, for more details about $\bN$-manifolds see e.g. \cite{cat:int, raj:tes, raj:qgr, roy:on}.

\begin{exa}[1-manifolds]\label{1-man}
Given a vector bundle $E\to M$, we can define the $1$-manifold   $$E[1]=(M,\Gamma(\wedge^\bullet E^*)).$$
Recall that any $1$-manifold is canonically of this type, see e.g. \cite{bon:on}.
Two important examples of 1-manifolds are $$T[1]M=(M,\Omega^\bullet(M)) \qquad \text{and} \qquad T^*[1]M=(M,\fX^\bullet(M)),$$ 
where $\Omega^\bullet(M)$ and $\fX^\bullet(M)$ are the sheaves of differential forms and multivector fields, respectively.
\end{exa}

\begin{exa}[Shifted tangent bundles of $n$-manifolds] \label{shift-tan-cotan}
One can generalize the shifted tangent bundles of ordinary manifolds considered in
Example \ref{1-man} to arbitrary $n$-manifolds $\cM=(M, C_\cM)$. Rather than introducing  $\Omega^\bullet(\cM)$ (see e.g. 
\cite{raj:tes}), we will give a description using charts and coordinates as in \cite{roy:on}.

Let $U$ be a chart of $\cM$  with coordinates $\{x^\alpha, e^{\beta}\}$,
so that $\{x^\alpha\}$ are coordinates on $U$  
and $C^\bullet(\cM)|_U=C^\infty(U)\otimes \sym\mathbf{V}$ with $\mathbf{V}=\langle e^\beta\rangle$. 
Just as for ordinary smooth manifolds, we can introduce the differentials of the coordinates, $\{dx^\alpha, de^\beta\}$, with the transformation rule 
\begin{equation*}
    \left\{\begin{array}{ll}
        \widehat{x}^\alpha=  & F^\alpha(x), \\
        \widehat{e}^\beta=  & F^\beta(x, e)
    \end{array} \right.\quad \Longrightarrow \quad\left\{\begin{array}{ll}
        d\widehat{x}^\alpha=  & \frac{\partial F^\alpha(x)}{\partial x^{\alpha'}}dx^{\alpha'}, \\
        d\widehat{e}^\beta= & \frac{\partial F^\beta(x, e)}{\partial x^\alpha}dx^\alpha+\frac{\partial F^\beta(x, e)}{\partial e^{\beta'}}dx^{\beta'}. 
    \end{array}\right.
\end{equation*}

For $k\geq 1$ we define the $(n+k)$-manifold $T[k]\cM=(M, C_{T[k]\cM})$ by
$$ C^\bullet(T[k]\cM)|_U= C^\infty(U)\otimes \sym\mathbf{W}\quad\text{with}\quad \mathbf{W}=\mathbf{V}\oplus\langle dx^\alpha, de^\beta \; ;\; |dx^\alpha|=k,\  |de^\beta|=|e^\beta|+k\rangle.
$$
\end{exa}

\begin{exa}[Vector and multivector fields]

Let $\cM$ be an $n$-manifold. A \emph{vector field of degree $k$} on $\cM$ 
is a degree $k$ derivation $X$ of the graded algebra $C^\bullet(\cM)$, i.e.,  an  
$\mathbb R$-linear map $X:C^\bullet(\cM)\to C^{\bullet+k}(\cM)$ with the property that, for all 
$f, g \in C_\cM$ with $f$ homogeneous, $|X(f)| =|f|+k$ and
    \begin{equation}\label{der}
 X(fg)=X(f)g+(-1)^{|f|k}fX(g).
    \end{equation}
Vector fields give rise to a sheaf of $C^\bullet(\cM)$-modules over $M$. The sheaf of degree 
$k$ vector fields is denoted by $\fX^{1,k}(\cM)$. The graded commutator of vector 
fields, defined for homogeneous vector fields $X, Y$ by
\begin{equation*}
    [X,Y]=XY-(-1)^{|X||Y|}YX,
\end{equation*}
makes $\fX^{1,\bullet}(\cM)$ into a sheaf of graded Lie algebras. A degree $1$ vector field on an $n$-manifold $\cM$, $Q\in\fX^{1,1}(\cM),$ such that $$[Q,Q]=2Q^2=0$$ is 
called a \emph{homological vector field}.
Note that, in this case, $(C^\bullet(\cM), Q)$ is a differential complex, justifying 
the terminology. The pair $(\cM , Q)$ will be referred as a \emph{$Q$-manifold}. 

    Till now we just defined vector fields on an $n$-manifold $\cM$. However, we also need to introduce \emph{multivector fields}. They are the elements of the graded symmetric algebra
    $$\Big(\bigoplus_i \sym^i \fX^{1,\bullet}(\cM)=\fX^{i,\bullet}(\cM), \wedge\Big)\quad\text{with}\quad X\wedge Y=(-1)^{1\cdot 1+|X||Y|}Y\wedge X\quad X,Y\in \fX^{1,\bullet}(\cM).$$
    If we extend the Lie bracket using the Leibniz rule, this algebra carries a Poisson bracket of bi-degree $(-1, 0)$ that we denote by $[\cdot,\cdot]$ and it is known as the \emph{Schouten bracket}.
\end{exa}

\subsection{Lie groupoids and algebroids in the category of graded manifolds}
 Following \cite{raj:tes, raj:qgr}, see also \cite{gra:gro}, we say that an \emph{$n$-graded Lie groupoid}\footnote{Do not confuse with the more familiar notion of Lie $n$-groupoid.} is a Lie groupoid object in the category $\man{n}$ such that its source and target maps are submersions. More concretely, this means a collection of data $(\cG\rightrightarrows \cM, \gs,\gt,\gm, \gu, \gi)$ where $\cG=(G, C^\bullet(\cG))$ and $\cM=(M, C^\bullet(\cM))$ are two $n$-manifolds and $\gs,\gt:\cG\to \cM, \ \gm:\cG^{(2)}=\cG_\gs\times_\gt\cG\to \cG,\ \gu:\cM\to \cG $ and $\gi:\cG\to \cG$ are graded morphisms that make the following diagrams commute
\begin{equation*}
    \xymatrix{\cG^{(3)}\ar[d]_{\gm\times \id}\ar[r]^{\id\times \gm}&\cG^{(2)}\ar[d]^\gm\\
	\cG^{(2)}\ar[r]_\gm&\cG}\quad \xymatrix{\cG\ar[dr]_{\id}\ar[r]^{(\id,\gu\circ \gs)}&\cG^{(2)}\ar[d]^\gm\\ &\cG}\quad \xymatrix{\cG\ar[dr]_{\id}\ar[r]^{(\gu\circ \gt,\id)}&\cG^{(2)}\ar[d]^\gm\\ &\cG}\quad  \xymatrix{\cG\ar[dr]_{\gu\circ \gs}\ar[r]^{(\gi,\id)}&\cG^{(2)}\ar[d]^\gm\\ &\cG}\quad \xymatrix{\cG\ar[dr]_{\gu\circ \gt}\ar[r]^{(\id, \gi)}&\cG^{(2)}\ar[d]^\gm\\ &\cG}
\end{equation*}

It follows from the definition that the underlying body $G\rightrightarrows M$ of an $n$-graded groupoid is a usual Lie groupoid.

An \emph{$n$-graded Lie groupoid morphism between $\cG\rightrightarrows\cM$ and $\cH\rightrightarrows\cN$} is a pair of $n$-manifold morphisms $\Phi_1:\cG\to\cH$ and $\Phi_0:\cM\to\cN$ for which the following diagrams commute
\begin{equation*}
\begin{array}{c}
    \xymatrix{\cG\ar^{\Phi_1}[r]\ar@<-.5ex>_{\gs}[d] \ar@<.5ex>^\gt[d]& \cH\ar@<-.5ex>_{\gs}[d] \ar@<.5ex>^{\gt}[d]\\ \cM\ar_{\Phi_0}[r]&\cN}
\end{array}\quad \text{and}\quad \begin{array}{c}
     \xymatrix{\cG^{(2)}\ar[d]_{\gm_\cG}\ar[r]^{\Phi_1^{(2)}}&\cH^{(2)}\ar[d]^{\gm_\cH}\\ \cG\ar[r]_{\Phi_1}&\cH.}
\end{array}
\end{equation*}

\begin{exa}[Shifted tangent groupoids]
Let $(\cG\rightrightarrows\cM, \gs,\gt, \gm, \gu, \gi)$ be an $n$-graded Lie groupoid. Then for any $j\in\bN$ we have that $(T[j]\cG\rightrightarrows T[j]\cM, T[j]\gs, T[j]\gt, T[j]\gm, T[j]\gu, T[j]\gi)$ defines an  $(n+j)$-graded Lie groupoid.
\end{exa}

\begin{exa}[Morita equivalence for graded Lie groupoids]
    As in the classical case, one can declare two $n$-graded Lie groupoids $\cG$ and $\cH$ to be Morita equivalent if there is a zig-zag of weak equivalences $\cG\leftarrow\cK\to \cH$. The definition of a weak equivalence is exactly the same as the classical one given in \cite[$\S 5.4$]{moeint}.
\end{exa}

Analogously to the classical theory, we have that an \emph{$n$-graded Lie algebroid} is a Lie algebroid object in the category $\man{n}$, see \cite{raj:tes, raj:qalg}. More concretely, such a structure consists of a collection of data $(\cA\to\cM,\rho, [\cdot,\cdot])$ where $\cA$ and $\cM$ are two $n$-manifolds that forms a vector bundle $\cA\to \cM$, $\rho:\cA\to T\cM$ is a vector bundle morphism covering the identity and $[\cdot,\cdot]:\Gamma^i\cA\times\Gamma^j\cA\to\Gamma^{i+j}\cA$ is a bracket, where $\Gamma^i\cA$ denotes the degree $i$ sections of $\cA\to\cM$,  satisfying:
\begin{enumerate}
	\item $[a_1, a_2]=-(-1)^{|a_1||a_2|}[a_2, a_1]$,
	\item $[a_1, f a_2]=\rho(a_1)(f) a_2+(-1)^{|a_1||f|}f[a_1,a_2]$,
	\item $[a_1,[a_2,a_3]]=[[a_1,a_2],a_3]+(-1)^{|a_1||a_2|}[a_2,[a_1,a_3]],$
\end{enumerate}
for $a_1, a_2, a_3\in\Gamma \cA, \ f\in C_\cM$.

The construction of the $n$-graded Lie algebroid of an $n$-graded Lie groupoid is identical to the classical one \cite[$\S 3.2.1$]{raj:tes}. Let $\cG\rightrightarrows \cM$ be an $n$-graded Lie groupoid, then the vector bundle $\cA=\gu^*\ker(T\gs)\to \cM$ has the structure of a graded algebroid where the anchor is given by $\rho=T\gt$ and the bracket is given by the identification 
$$\Gamma\cA\cong\{ \text{ Right invariant vector fields on } \cG\}$$
where a \emph{right (resp. left) invariant vector field} $X\in\fX^{1,j}(\cG)$ is defined by the equations 
\begin{eqnarray}
    T[j]\gs\circ X=0 \quad \text{and}\quad T[j]\gm\circ(X\times 0)\circ ((\gu\circ\gt)\times\id)\circ \Delta=X\label{eq:R}\\
    (resp. \quad T[j]\gt\circ X=0 \quad \text{and}\quad T[j]\gm\circ(0\times X)\circ (\id\times (\gu\circ\gs))\circ \Delta=X),
\end{eqnarray}
with $0:\cG\to T[j]\cG$ the zero section and $\Delta:\cG\to\cG\times\cG$ the diagonal map.

\begin{exa}(The shifted tangent algebroids)
Let $\cG\rightrightarrows \cM$ be an $n$-graded Lie groupoid with $n$-graded Lie algebroid $\cA\to\cM$ then we have that $T[k]\cA\to\cM$ are $(n+k)$-graded Lie algebroids. In fact, $\cA_{T[k]\cG}\cong T[k]\cA$. Moreover, we can see that $T[k]\cG\to T[k]\cM$ is a VB-groupoid with base groupoid $\cG\to\cM$ and core given by $\cA[k]\to\cM$. A similar description holds for the corresponding $(n+k)$-graded Lie algebroids, see \cite[$\S 11.2$]{mac:book} for the classical picture.
\end{exa}

\subsection{Homological vector fields on graded Lie groupoids}
The next step is to study vector fields on $n$-graded Lie groupoids. Let $\cG\rightrightarrows \cM$ be an $n$-graded Lie groupoid with $n$-graded Lie algebroid $\cA\to\cM$. There are three classes of distinguished vector fields \emph{left invariant, right invariat } and \emph{multiplicative}. 

As we already pointed out, left and right invariant vector fields can be identified with sections of the $n$-graded Lie algebroid. More explicitly, if we take a section of degree $j$, $a\in\Gamma^j \cA$, then we can define its corresponding right and left invariant vector fields on $\cG$ by the formulas $$a^r=T[j]\gm\circ(a\circ \gt\times 0)\circ \Delta \quad \text{ and }\quad a^l=T[j]\gm\circ(0\times a\circ \gs)\circ \Delta.$$

We say that $X\in\fX^{1,k}(\cG)$ is a \emph{multiplicative vector field} if there is $X_0\in\fX^{1,k}(\cM)$ such that
\begin{equation*}
	\xymatrix{\cG\ar^{X}[r]\ar@<-.5ex>[d] \ar@<.5ex>[d]& T[k]\cG\ar@<-.5ex>[d] \ar@<.5ex>[d]\\ \cM\ar_{X_0}[r]&T[k]\cM}
\end{equation*}
is a groupoid morphism. Notice that for any $a\in\Gamma^j\cA$, the vector field $X=a^l-a^r$ is multiplicative with $X_0=\rho(a)$. We denote the multiplicative vector fields by $\fX^{1,\bullet}_{mul}(\cG)$

\begin{rema}
Our definition of multiplicative vector fields follows the classical definition introduced in \cite[$\S 3.3$]{mac:cla}. In \cite[$\S 3.7$]{raj:qgr} the simplicial approach is adopted. It is straightforward to see that both definitions coincide.
\end{rema}

\begin{defi}\label{def:qqgpd}
An  \emph{$n$-graded quasi $Q$-groupoid } $(\cG\rightrightarrows\cM, Q, \qq)$ is an $n$-graded Lie groupoid $\cG \rightrightarrows \cM$ endowed with a degree $1$ multiplicative vector field $Q\in\fX^{1,1}_{mul}(\cG)$ and a degree $2$ section of the graded algebroid $\qq\in\Gamma^2\cA$ satisfying $$Q^2=\qq^l-\qq^r \qquad \text{ and }\qquad [Q,\qq^r]=0.$$
\end{defi}
When $\qq=0$ we recover the $Q$-groupoids introduced in \cite{raj:tes,raj:qgr}. We postpone the discussion of non-trivial examples to $\S\ref{sec:1q-groupoids}$ and $\S\ref{sec:2q-groupoids}$.

Multiplicative vector fields have an infinitesimal counterpart given by IM\footnote{IM stands from Inifnitesimally Multiplicative.} vector fields. To be precise, we say that $X\in\fX^{1,k}(\cA)$ is an \emph{ IM vector field} if there is $X_0\in\fX^{1,k}(\cM)$ such that
\begin{equation*}
	\xymatrix{\cA\ar[d]\ar[r]^{X}&T[k]\cA\ar[d]\\ \cM\ar[r]_{X_0}&T[k]\cM}
\end{equation*}
gives a graded Lie algebroid morphism. 
\begin{defi}
An  \emph{$n$-graded quasi $Q$-algebroid } $(\cA\to\cM, Q, \qq)$ is an $n$-graded Lie algebroid $\cA\to\cM$ endowed with a degree $1$ IM vector field $Q\in\fX^{1,1}(\cA)$ and a degree $2$ section $\qq\in\Gamma^2\cA$ satisfying $$Q^2=[\qq, \cdot] \qquad \text{ and }\qquad Q(\qq)=0.$$
\end{defi}

When $\qq=0$ we recover the $Q$-algebroids introduced in \cite{raj:tes,raj:qalg}. If in addition we have $([\cdot,\cdot]=0,\rho=0)$ we get the $Q$-bundles of \cite{kot:Qbun}. 

Thus, one obtains the following natural differentiation result.
\begin{prop}
Let $(\cG\rightrightarrows\cM, Q, \qq)$ be an $n$-graded quasi Q-groupoid then $(\cA\to\cM, \Lie(Q), \qq)$ is an $n$-graded quasi Q-algebroid.
\end{prop}

In order to shed light on the definition of quasi Q-structures we re-interpret them in several ways. In $\S \ref{subsec:qqnat}$ we show how $q$ gives a natural transformation between $Q^2$ and $0$. In $\S \ref{sec:MC}$ we demonstrate that quasi Q-structures are Maurer-Cartan elements. Finally, in Proposition \ref{prop:actR1} we characterize them as actions of $\bR[-1]$ on $n$-graded Lie groupoids. 

\subsubsection{Interpretation in terms of natural transformations}\label{subsec:qqnat} 
Here we describe the $\qq\in\Gamma^2\cA$ as a natural transformation between $Q^2$ and the zero vector field. In this way, we generalize the work of \cite{bacr}, where the Jacobiator of a Lie 2-algebra is defined as a natural transformation. We explain this relation in $\S \ref{subsec:l2alg}$.

Let $(\Phi_1,\Phi_0),(\Psi_1, \Psi_0):(\cG\rightrightarrows\cM)\to(\cH\rightrightarrows\cN)$ be two graded groupoid morphisms. Recall that a \emph{natural transformation $\alpha:(\Phi_1,\Phi_0)\Rightarrow (\Psi_1, \Psi_0)$} is a graded manifold morphism $\alpha: \cM\to \cH$ satisfying the following equations
\begin{equation}\label{eq:NT}
    \gs_\cH\circ\alpha=\Phi_0,\quad \gt_\cH\circ\alpha=\Psi_0\quad\text{and}\quad  \gm_\cH\circ(\Phi_1\times \alpha)\circ(\id\times \gs_{\cG})=\gm_\cH\circ( \alpha\times \Psi_1)\circ(\gt_{\cG}\times\id).
\end{equation}

\begin{prop}\label{prop:NT}
    Let $\cG\rightrightarrows\cM$ be a graded Lie groupoid and $Q\in\fX^{1,1}_{mul}(\cG)$. Then $\qq\in\Gamma^2\cA$ satisfying $Q^2=\qq^l-\qq^r$ is equivalent to a natural transformation \begin{center} 
\begin{tikzpicture}
\matrix[matrix of nodes,column sep=2cm] (cd)
{
   $\mathcal{G}$  & $T[2]\mathcal{G}$  \\
};
\draw[->] (cd-1-1) to[bend left=50] node[label=above:$0_{\mathcal{G} }$] (U) {} (cd-1-2);
\draw[->] (cd-1-1) to[bend right=50,name=D] node[label=below:$Q^2$] (V) {} (cd-1-2);
\draw[double,double equal sign distance,-implies,shorten >=10pt,shorten <=10pt] 
  (U) -- node[label=right:$\alpha $] {} (V);
\end{tikzpicture}
\end{center}   
such that $\pi_\cG\circ \alpha=\gu$, where $\pi_\cG:T[2]\cG\to \cG$ is the projection. Moreover $[Q, \qq^r]=0$ if and only if $\alpha$ preserves $Q$, i.e. $T[2]Q\circ\alpha=T[1]\alpha\circ Q_0$.
\end{prop}
\begin{proof}
    Let $\alpha$ be a natural transformation satisfying $\pi_\cG\circ \alpha=\gu$. Then $T[2] \mathtt{s}\circ \alpha =0_{\mathcal{M} }$ and so, combining this fact with equation $\pi_{\mathcal{G} }\circ \alpha =\gu$, we get that $\alpha $ is a degree 2 section of $\cA$. Since $\alpha$ is a natural transformation we have that the  diagram
    \begin{equation}\label{NT}
    \begin{array}{c}
         \xymatrix{ & T[2]\mathcal{G}^{(2)} 
\ar[dr]^-{T[2] \mathtt{m} } & \\ \mathcal{G}\ar[ur]^-{(0_\cG,\alpha \circ \mathtt{s})}\ar[dr]_-{(\alpha \circ \mathtt{t},Q^2 )} && T[2] \mathcal{G} \\
& T[2]\mathcal{G}^{(2)} 
\ar[ur]_-{T[2]\mathtt{m}}  &} 
    \end{array}
    \end{equation}
is commutative. Thus $Q^2=T[2]\gm ((T[2]\gi) \circ \alpha \circ \gt+ 0_\cG ,\alpha\circ \gs)$ and hence
\begin{equation}\label{eq1}
    Q^2=T[2]\gm ((T[2]\gi) \circ \alpha \circ \gt+ 0_\cG ,0_\cG+\alpha\circ \gs)=\alpha^l+ ((T[2]\mathtt{i}) \circ \alpha )^r=\alpha^l-\alpha^r,
\end{equation}
since $T[2]\mathtt{m}$ is linear on pairs of composable sections. Conversely, starting from $\qq\in\Gamma^2\cA$ satisfying $Q^2=\qq^l-\qq^r$ we define $\alpha:\cM\to T[2]\cG$ as $\qq:\cM\to\cA[2]\subset T[2]\cG$. Then we get that 
$$T[2]\gs (\alpha)=T[2]\gs (\qq)=0,\quad T[2]\gt (\alpha)=T[2]\gt(\qq)=\rho(\qq)=Q_0^2$$
and by \eqref{eq1} the diagram \eqref{NT} commutes and so $\alpha $ is a natural transformation from $0_{\mathcal{G} }$ to $Q^2$. 

For the moreover part, notice that we can view $Q$ and $\alpha$ as derivations. Thus, we get that condition $[Q, \qq^r]=0$ is equivalent to the fact that $\alpha$ preserves $Q$. 
\end{proof}

\subsubsection{Interpretation in terms of the Maurer-Cartan equation}\label{sec:MC} 
Let $G\rightrightarrows M$ be a Lie groupoid with Lie algebroid $(A\to M,\rho, [\cdot,\cdot])$. It is shown in \cite[Example 7.2]{ortwal} that $\Gamma A\to \fX_{mul}(G)$ inherits a dgla\footnote{Differential graded Lie algebra.} structure given by
$$d(a)=a^r-a^l,\quad [X,Y]=[X,Y],\quad [X, a]=[X, a^r]\circ \gu \quad \text{and}\quad [a,b]=0,$$
where $a,b\in\Gamma A, \ X,Y\in\fX_{mul}(G).$ Moreover, in \cite[Prop 8.1]{ortwal} they show that if two groupoids are Morita equivalent, then the associated dgla`s are $L_\infty$ quasi-isomorphic.  As explained in \cite{hep:vec}, the above dgla codifies the vector fields on the differentiable stack presented by $G\rightrightarrows M$. For $n$-graded Lie groupoids we have an analogous result. 

\begin{prop}\label{prop:dgla}
    Let $\cG\rightrightarrows\cM$ be an $n$-graded Lie groupoid and $\cA\to \cM$ its $n$-graded Lie algebroid. Define $L_{\cG}=\oplus_{i\in\bZ}L^i$ with $L^i=\Gamma^{i+1}\cA\oplus \fX^{1,i}_{mul}(\cG)$. Then $L_{\cG}$ inherits a dgla structure given by
    $$d(a)=a^r-a^l,\quad [X,Y]=[X,Y],\quad [X, a]=[X, a^r]\circ \gu \quad \text{and}\quad [a,b]=0,$$
    where $a,b\in\Gamma^\bullet \cA, \ X,Y\in\fX^{1,\bullet}_{mul}(\cG).$ Moreover, if two $n$-graded Lie groupoid are Morita equivalent then the associated dgla`s are $L_\infty$ quasi-isomorphic.
\end{prop}

\begin{proof}
    This is analogous to the classical case presented in \cite{ortwal}.
\end{proof}

\begin{prop}\label{prop:MC}
    Let $\cG\rightrightarrows\cM$ be an $n$-graded Lie groupoid. Then there is a canonical one to one correspondence between: 
    \begin{itemize}
        \item Maurer-Cartan elements of $(L_\cG, d, [\cdot, \cdot])$.
        \item Quasi $Q$-groupoid structures on $\cG\rightrightarrows\cM$.
    \end{itemize}
\end{prop}
\begin{proof}
    Recall that a \emph{Maurer-Cartan element} on a dgla is a degree $1$ element $X$ satisfying $dX+\frac{1}{2}[X,X]=0.$ In our case,  $L^1_\cG=\Gamma^2\cA\oplus\fX_{mul}^{1,1}(\cG)$ then $X=\qq+Q$ and we have that the Maurer-Cartan equation decomposes into 
    $$0=\qq^r-\qq^l+\frac{1}{2}[Q,Q]=\qq^r-\qq^l+Q^2 \quad\text{and}\quad 0=[Q, \qq]=[Q,\qq^r]\circ \gu$$
    which are exactly the equations defining a quasi $Q$-groupoid structure. The other direction is analogous. 
\end{proof}

\begin{rema}[Gauge transformations]
    The degree $0$ elements on a dgla act by \emph{gauge transformations} on Maurer-Cartan elements, see e.g. \cite[$\S 5.1$]{jonas:intro}. In our case,  $L_\cG^0=\Gamma^1\cA\oplus\fX^{1,0}_{mul}(\cG)$ so the gauge transformation of $X\in \fX^{1,0}_{mul}(\cG)$ is just a usual transformation by a groupoid diffeomorphism and therefore expected. But for a $b\in\Gamma^1\cA$ we get the gauge tranformation
    $$e^b\cdot(\qq,Q)=(\qq+[Q,b^r]\circ \gu-\frac{1}{2}[b,b],\ Q+b^l-b^r)\quad\text{for}\quad (\qq,Q)\in L_\cG^1.$$
    In \cite{bru:hom}, $Q$-algebroids with $Q=[b,\cdot]$ for $b\in\Gamma^1\cA$ where studied. In our language, those are the gauge equivalent to zero with $[b,b]=0$.  
\end{rema}
Some immediate and interesting consequences of the above propositions are the following.

\begin{coro}
Let $\cG\rightrightarrows\cM$ be an $n$-graded Lie groupoid and $(Q,q)$ a quasi $Q$-groupoid structure. Small deformations of the quasi $Q$-groupoid structure $(Q,q)$ are controlled by the twisted dgla $(L_\cG, d+[Q+q,\cdot], [\cdot,\cdot])$. 
\end{coro}
For more detail on $L_\infty$-algebras and deformations see \cite{jonas:intro} and references therein.

\begin{coro}\label{cor:mor}
    Quasi $Q$-groupoid structures can be transported along Morita equivalences of graded Lie groupoids. 
\end{coro}
\begin{proof}
    This follows from the fact that $L_\infty$ quasi-isomorphisms between strict Lie 2-algebras induce bijections between moduli sets of Maurer-Cartan elements, see \cite[Corollary A.16]{bone:1shp}.
\end{proof}

Observe that the concept of $Q$-groupoid is not Morita invariant. More concretely, let $(\cG\rightrightarrows \cM, Q)$ be an $n$-graded Q-groupoid and $\cH\rightrightarrows\cN$ another $n$-graded Lie groupoid Morita equivalent to $\cG\rightrightarrows \cM$. Thus $\cH\rightrightarrows\cN$ inherits naturally a quasi $Q$-groupoid structure by the above Corollary but, in general, this is not automatically a $Q$-groupoid structure.

\subsection{Quasi QP-groupoids}\label{sec:pq} On a Lie groupoid $G\rightrightarrows M$, the dgla $\Gamma A\to \fX_{mul}(G)$ described in $\S\ref{sec:MC}$ can be extended to include multivector fields as explained in \cite{bone:1shp}. The underlying graded vector space is given by $V^{i-1}_G=\Gamma\wedge^{i+1}  A\oplus \fX^i_{mul}(G)$ with differential and bracket 
\begin{equation}\label{eq:dglaV}
    d(\pi+\Pi)=0\oplus \pi^r-\pi^l, \quad [\pi_1\oplus \Pi_1, \pi_2\oplus \Pi_2]=(-1)^k \Pi_1\cdot \pi_2-(-1)^{(k+1)l}\Pi_2\cdot \pi_1\oplus [\Pi_1, \Pi_2]
\end{equation}
where $\pi_1\oplus \Pi_1\in V_G^k,\ \pi_2\oplus \Pi_2\in V_G^l$ and $\Pi_1\cdot \pi_2\in\Gamma\wedge^{k+l+2} A$ denotes the unique section such that  $[\Pi_1, \pi^r_2]=(\Pi_1\cdot \pi_2)^r$. The importance of this dgla comes from the fact that its Maurer-Cartan elements are exactly \emph{quasi Poisson structures}, i.e. $\pi+\Pi\in V^1_G=\Gamma\wedge^3 A\oplus \fX^2_{mul}(G)$ such that 
\begin{equation}\label{eq:qP}
    \pi^l-\pi^r=\frac{1}{2}[\Pi, \Pi],\quad \text{and}\quad [\Pi, \pi^r]=0.
\end{equation}
This dgla extends to the graded context as we show now. Let $\cG\rightrightarrows \cM$ be an $n$-graded Lie groupoid.  Define $$\cV^{i-(n+1)}_\cG=\bigoplus_{(n+1)j+k=i}\Gamma^{k+1}\wedge^j\cA\oplus \fX^{j,k}_{mul}(\cG)$$ with differential and bracket given by the same formulas as before. It is easy to see that
$$\cV^1_\cG=C^{n+3}(\cM)\oplus C_{mul}^{n+2}(\cG)\oplus \Gamma^2\cA\oplus\fX^{1,1}_{mul}(\cG)\oplus \Gamma^{1-n}\wedge^2\cA\oplus \fX_{mul}^{2,-n}(\cG)\oplus\Gamma^{-2n}\wedge^3\cA\cdots.$$
By inspection one can see that its Maurer-Cartan elements include quasi $Q$-structures as well as degree $-n$ quasi Poisson structures. Moreover, we hve a natural notion of compatibility for two such structures. 

\begin{defi}\label{def:qpqgro}
    Let $\cG\rightrightarrows \cM$ be an $n$-graded Lie groupoid. A \emph{quasi $QP$-structure} is a $\qq\in\Gamma^2\cA, \ Q\in\fX^{1,1}_{mul}(\cG), \ \Pi\in\fX^{2,-n}_{mul}(\cG)$ and $\pi \in\Gamma^{-2n}\wedge^3\cA$ satisfying 
    $$Q^2=\qq^l-\qq^r,\quad \frac{1}{2}[\Pi,\Pi]=\pi^l-\pi^r, \quad\text{and}\quad [Q,\qq^r]=[\Pi,\pi^r]=[Q,\Pi]=[Q,\pi^r]= [\Pi, \qq^r]=0.$$
    When $Q=0$ and $\qq=0$ we say that $(\cG\rightrightarrows \cM, \Pi, \pi)$ is a \emph{degree $n$ quasi P-groupoid} and when $\qq=0$ and $\pi=0$ then $(\cG\rightrightarrows \cM, Q, \Pi)$ is a degree $n$ \emph{QP-groupoid} as studied in \cite{lacou}.
\end{defi}

\begin{rema}
    A general Maurer-Cartan element of the dgla $(\cV^\bullet_\cG, d, [\cdot,\cdot])$ is a multiplicative analogue of the homotopy Poisson structures studied in \cite{raj:hop}. In $\S\ref{subsec:catliebia}$ we study the degree $1$ case. For other strict versions of Poisson structure on graded groupoids see \cite{gra:gro} . 
\end{rema}

A particularly relevant case for us is given by assuming that $\pi=0$ and $\Pi$ is a non-degenerate bivector field, i.e. that $\Pi$ is the inverse of a symplectic structure. In this case, the vector field $Q$ is always hamiltonian, i.e. $Q=\{\Theta,\cdot\}$ for some $\Theta\in C^{n+1}_{mul}(\cG)$, see \cite{roy:on}. Therefore one can also impose the following condition on $q$.

\begin{defi}\label{def:deg2sym} A {\em degree $n$ symplectic quasi $Q$-groupoid} $(\cG\rightrightarrows \cM,\Pi,\Theta,\theta)$ is an $n$-graded Lie groupoid $\cG\rightrightarrows \cM$ equipped with $\Pi\in\fX^{2,-n}_{mul}(\cG)$ which is Poisson and non-degenerate, a $\Theta\in C_{mul}^{n+1}(\cG)$ and a $\theta\in C^{n+2}(\cM)$ satisfying
\begin{align} \frac{1}{2}\{\Theta,\Theta\}=\delta \theta=\mathtt{s}^*\theta-\mathtt{t}^*\theta\quad\text{and}\quad \{\Theta,\gt^*\theta\}=0; \label{eq:hommaseq}
\end{align} 
where $\{\cdot,\cdot\}$ is the Poisson bracket determined by $\Pi$.
\end{defi}

\begin{coro}
A degree $n$ symplectic quasi $Q$-groupoid $(\cG\rightrightarrows \cM,\Pi,\Theta,\theta)$ defines a quasi QP-groupoid with
\[ Q=\{\Theta, \cdot \} \quad \text{and}\quad  q^r=\{\mathtt{t}^*\theta, \cdot\}.\]
\end{coro}
\begin{proof}
    Follows directly from the fact that that Hamiltonian vector fields preserves the Poisson structure together with $X_{\{f,g\}}=[X_f,X_g]$ for $f,g\in C(\cG)$. 
\end{proof}

\begin{prop}\label{prop:cot}
    Let $\cG\rightrightarrows\cM$ be an $n$-graded Lie groupoid with $n$-graded Lie algebroid $\cA$. 
    \begin{enumerate}
        \item For all $l>n, (T^*[l]\cG\rightrightarrows\cA^*[l],\Pi_{can})$ is a degree $l$ symplectic groupoid.
        \item There is a one to one correspondence between:
    \begin{itemize}
        \item Maurer-Cartan elements in $(\cV^\bullet_\cG, d, [\cdot,\cdot])$.
        \item Pairs $\Theta\in C_{mul}^{n+2}(T^*[n+1]\cG)$  and $\theta\in C^{n+3}(\cA^*[n+1])$ satisfying \eqref{eq:hommaseq}.
    \end{itemize}
    \end{enumerate}
\end{prop}
\begin{proof}
    The first statement is a graded analog of the classical result and the proof is identical, see e.g. \cite[$\S 11.3$]{mac:book}. For the second statement, notice that
    \begin{equation*}
        \begin{split}
           \cV^1_\cG=&\bigoplus_{(n+1)j+k=n+2}\Gamma^{k+1}\wedge^j\cA\oplus \fX^{j,k}_{mul}(\cG)\\
           =&\Big(\bigoplus_{(n+1)j+k=n+3}\Gamma^{k}\wedge^j\cA\Big)\oplus\Big(\bigoplus_{(n+1)j+k=n+2} \fX^{j,k}_{mul}(\cG)\Big) \\
           =& C^{n+3}(\cA^*[n+1])\oplus C^{n+2}_{mul}(T^*[n+1]\cG).
        \end{split}
    \end{equation*}
Thus, elements in $\cV^1$ are identified with pairs $(\theta, \Theta)$. Under this identification, the Maurer-Cartan equation $$d(\theta, \Theta)+\frac{1}{2}[(\theta, \Theta), (\theta, \Theta)]=0$$ decomposes into the equations \eqref{eq:hommaseq} because $d\theta=\theta^r-\theta^l=\gt^*\theta-\gs^*\theta.$
\end{proof}

\subsection{Morphisms of quasi Q-groupoids}\label{sec:mor} Here we show that quasi Q-groupoids naturally form a strict 2-category. We will show in $\S\ref{subsec:l2alg}$ that this strict $2$-category extends the category of semistrict Lie 2-algebras \cite{bacr}.

\subsubsection{1-morphisms} Given two quasi Q-groupoids $(\cG\rightrightarrows\cM ,Q_\cG,\qq_\cG )$ and $(\cH\rightrightarrows\cN ,Q_\cH,\qq_\cH)$, a morphism (or 1-morphism) between them is a pair $(F,\tau)$ consisting of a morphism of graded Lie groupoids $F:\mathcal{G} \rightarrow \mathcal{H} $ together with a natural transformation $\tau: T[1]F\circ Q_{\mathcal{G}}\Rightarrow Q_{\mathcal{H}}\circ F    $ as in the following diagram
\begin{center} 
\begin{tikzpicture}
\matrix[matrix of nodes,column sep=2cm] (cd)
{
   $\mathcal{G}$  & $T[1]\mathcal{H}$  \\
};
\draw[->] (cd-1-1) to[bend left=50] node[label=above:${T[1]F\circ Q_{\mathcal{G}}}$] (U) {} (cd-1-2);
\draw[->] (cd-1-1) to[bend right=50,name=D] node[label=below:$Q_{\mathcal{H}}\circ F$] (V) {} (cd-1-2);
\draw[double,double equal sign distance,-implies,shorten >=10pt,shorten <=10pt] 
  (U) -- node[label=right:$\tau $] {} (V);
\end{tikzpicture}
\end{center}   
such that $\pi_{F^*T[1]\cH}\circ\tau=\gu_\cG$ and the following coherence law is satisfied: the two natural transformations  $0 \Rightarrow  Q^2_{\mathcal{H}}\circ  F$ we can define canonically in this situation coincide, i.e. 
\begin{equation}\label{eq:com-1mor}
    \qq_\cH\diamond 1_{F}=\left(1_{T[1]Q_\cH} \diamond \tau \right)\bullet \left( T[1]\tau \diamond 1_{Q_\cG} \right)\bullet \left( 1_{T[2]F}\diamond \qq_\cG\right);
\end{equation}
where $\diamond$ and $\bullet $ denote, respectively, the horizontal and vertical compositions of natural transformations.

To make sense of \eqref{eq:com-1mor}  notice that $0=T[2]F\circ 0_\cG=0_\cH\circ F$ and that for a vector field $Q\in\fX^{1,k}(\cM)$ we get that $T[1]Q:T[1]\cM\to T[1]T[1]\cM$ identifies $T[1]Q=\cL_Q$ with the Lie derivative because  $C_{T[1]\cM}\cong \Omega(\cM)$. However, $T[1]T[1]\cM$ is a double vector bundle with core $T[2]\cM$ and we get $Q^2=T[1]Q\circ Q$ restricted to the core.

With these identifications in mind, the left hand side in \eqref{eq:com-1mor} is the natural transformation
\begin{center} 
\begin{tikzpicture}
\matrix[matrix of nodes,column sep=2cm] (cd)
{
   $\mathcal{G}$  & $\mathcal{H}$ & $T[2]\mathcal{H}$ \\
};
\draw[->] (cd-1-1) to[bend left=50] node[label=above:$F$] (U) {} (cd-1-2);
\draw[->] (cd-1-1) to[bend right=50,name=D] node[label=below:$F$] (V) {} (cd-1-2);
\draw[double,double equal sign distance,-implies,shorten >=10pt,shorten <=10pt] 
  (U) -- node[label=right:$1_F$] {} (V);
    \draw[->] (cd-1-2) to[bend left=50] node[label=above:$0_{\mathcal{H}}$] (UF) {} (cd-1-3);
    \draw[->] (cd-1-2) to[bend right=50,name=D] node[label=below:$Q^2_{\mathcal{H}}$] (VF) {} (cd-1-3);
\draw[double,double equal sign distance,-implies,shorten >=10pt,shorten <=10pt] 
  (UF) -- node[label=right:${\qq_\mathcal{H}} $] {} (VF);

\end{tikzpicture}
\end{center}   
while the right hand side in \eqref{eq:com-1mor} is the composition of the natural transformations
\[ \begin{tikzcd}
{\cG} \arrow[d, Rightarrow, no head]\arrow[rrrr, "{0_{\mathcal{G}}}"] &&\arrow[d, Rightarrow, "{\qq_\mathcal{G}}"']&& {T[2]\mathcal{G}}\arrow[d, Rightarrow, no head] \arrow[rrrr, "{T[2]F}"] &&\arrow[d, Rightarrow, "{1_{T[2]F}}"] &&{T[2]\cH}\arrow[d, Rightarrow, no head]\\
{\cG} \arrow[d, Rightarrow, no head]\arrow[rr, "{Q_\mathcal{G}}"] &\arrow[d, Rightarrow, "{1_{Q_\cG}}"]&{T[1]\cG}\arrow[d, Rightarrow, no head]\arrow[rr, "{T[1]Q_\mathcal{G}}"']&{}&{T[2]\cG} \arrow[rrrr, "{T[2]F}"'] &\arrow[d, Rightarrow, "{T[1]\tau}"]& {}&&{T[2]\cH}\arrow[d, Rightarrow, no head]\\
{\cG}\arrow[d, Rightarrow, no head] \arrow[rr, "{Q_\cG}"'] &{}&{T[1]\cG}\arrow[rrrr, "{T[1]F}"]&{}\arrow[d, Rightarrow, "{\tau}"]&{}&{}& {T[1]\cH}\arrow[d, Rightarrow, no head]\arrow[rr, "{T[1]Q_\cH}"]&{}\arrow[d, Rightarrow, "{1_{T[1]Q_H}}"]&{T[2]\cH}\arrow[d, Rightarrow, no head]\\
{\cG}\arrow[rrrr, "{F}"']&{}&{}&{}& {\cH}\arrow[rr, "{Q_\cH}"']&{}&{T[1]\cH}\arrow[rr,"{T[1]Q_\cH}"' ]&{}&{T[2]\cH.}
\end{tikzcd}\]

We will see later several situations that are clarified because of the introduction of this concept. Let us start by relating $1$-morphisms with gauge transformations.

\begin{prop}\label{prop:1morgau}
    Let $(1,\tau):(\cG\rightrightarrows \cM, Q, \qq)\to (\cG\rightrightarrows \cM, Q', \qq')$ be a $1$-morphism. Then $\tau$ is given by $b\in\Gamma^1\cA$ and the quasi Q-structures are gauge equivalent, i.e. $e^b\cdot(\qq,Q)=(\qq',Q').$
\end{prop}
\begin{proof}
By definition, a $1$-morphism $(1,\tau)$ is given by a natural transformation $\tau:Q\Rightarrow Q'$ that we can also view as a natural transformation $\tau:0_\cG\Rightarrow Q'-Q$. So it is determined by a map $\tau:\cM\to T[1]\cG$ satisfying  $T[1]\gs_\cG\circ \tau=0$ and  $\pi_{T[1]\cG}\circ\tau=\gu_\cG$ so $\tau$ is given by a degree $1$ section of the graded Lie algebroid $b\in\Gamma^1\cA.$ Moreover, the last equation in \eqref{eq:NT} implies that the following diagram is commutative
\begin{equation*}
    \begin{array}{c}
         \xymatrix{ & T[1]\mathcal{G}^{(2)} 
\ar[dr]^-{T[1] \mathtt{m} } & \\ \mathcal{G}\ar[ur]^-{(0_\cG,\tau \circ \mathtt{s})}\ar[dr]_-{(\tau \circ \mathtt{t},Q'-Q )} && T[1] \mathcal{G} \\
& T[1]\mathcal{G}^{(2)} 
\ar[ur]_-{T[1]\mathtt{m}}  &} 
    \end{array}
    \end{equation*}
Thus, using the same trick as in \eqref{NT}, we get
$Q'-Q=b^l-b^r.$ Since $(\qq,Q)$ and $(\qq',Q')$ are quasi $Q$-groupoids, we get 
\begin{equation}\label{aux:Mors}
\begin{split}
        (\qq')^l-(\qq')^r=&Q'\circ Q'=Q'\circ(b^l-b^r+Q) =Q'\circ(b^l-b^r)+(b^l-b^r+Q)\circ Q\\
        =&Q'\circ(b^l-b^r)+(b^l-b^r)\circ Q+\qq^l-\qq^r.   
    \end{split}
\end{equation}
Since $F=\id,$ the horizontal multiplication $\diamond$ of natural transformations corresponds to composition of vector fields while the vertical multiplication $\bullet$ is just the sum of vector fields. These observations imply that \eqref{eq:com-1mor} is equivalent to \eqref{aux:Mors}. Finally, observe that if we apply $Q'=b^l-b^r+Q$ in the last term of  \eqref{aux:Mors} we get the equation   
\begin{equation*}
    \begin{split}
    (\qq')^l-(\qq')^r=&(b^l-b^r+Q)\circ(b^l-b^r)+(b^l-b^r)\circ Q+\qq^l-\qq^r\\
    =&-\frac{1}{2}[b,b]^l-[b^r,b^l]+\frac{1}{2}[b,b]^r+[Q,b^l]-[Q,b^r]+\qq^l-\qq^r;
\end{split}
\end{equation*}
meaning that $(\qq',Q')=e^b\cdot(\qq,Q)$ as we want.
\end{proof}

\subsubsection{Composition of 1-morphisms} Let $(F_1,\tau_1):(\mathcal{G} ,Q_{\mathcal{G} },\qq_{\mathcal{G} }) \rightarrow (\mathcal{H} ,Q_{\mathcal{H} },\qq_{\mathcal{H} })$ and $(F_2,\tau_2):(\mathcal{H} ,Q_{\mathcal{H} },\qq_{\mathcal{H} }) \rightarrow (\mathcal{K} ,Q_{\mathcal{K} },\qq_{\mathcal{K} })$ be morphisms of quasi Q-groupoids. Then their composition is $(F_2\circ F_1,(\tau_2\diamond 1_{F_1})\bullet(1_{T[1]F_1}\diamond \tau_1))$, according to the diagram below
\[ \begin{tikzcd}                                                                 &  &  & {} \arrow[dd, "{1_{T[1]F_2}\diamond \tau_1}" description, Rightarrow]     &  &  &   \\                                                                   &  &  & {}  &  &  &   \\
\mathcal{G} \arrow[rrrrrr, "{T[1]F_2\circ T[1]F_1\circ Q_{\mathcal{G}}}", bend left=49, shift left=6] \arrow[rrrrrr, "{Q_{\mathcal{K}}\circ F_2\circ F_1}"', bend right=49, shift right=6] \arrow[rrrrrr, "{T[1]F_2\circ Q_{\mathcal{H}}\circ F_1}" description] &  &  & {} \arrow[dd, "{\tau_2 \diamond 1_{F_1}}" description, Rightarrow] &  &  & {T[1]\mathcal{K}} \\                                                   &  &  &  &  &  &                   \\
  &  &  & {}   &  &  &                  
\end{tikzcd} \]

\subsubsection{2-morphisms} Finally, given two 1-morphisms $(F,\tau),(F',\tau'):(\mathcal{G} ,Q_{\mathcal{G} },\qq_{\mathcal{G} }) \rightarrow (\mathcal{H} ,Q_{\mathcal{H} },\qq_{\mathcal{H} })$ a 2-morphism between them is a natural transformation $\Theta:F'\Rightarrow F$ such that
\[ \begin{tikzcd}                                                                     &  &  & {} \arrow[dddd, "{ (1_{Q_{\mathcal{H}}}\diamond \Theta)\bullet\tau'= \tau \bullet(T[1]\Theta\diamond 1_{Q_{\mathcal{G}}})}" description, Rightarrow] &  &  &                   \\                                         &  &  &   &  &  &                   \\
\mathcal{G} \arrow[rrrrrr, "Q_{\mathcal{H}}\circ F'", bend left=49, shift left=3] \arrow[rrrrrr, "{T[1]F\circ Q_{\mathcal{G}}}"', bend right=49, shift right=3] &  &  &       &  &  & {T[1]\mathcal{H}}\\                                 &  &  &    &  &  &                   \\
    &  &  & {}   &  &  &                  
\end{tikzcd} \]
Since the 2-morphisms are given by natural transformations with an additional property, we have that quasi Q-groupoids form a strict 2-category as claimed at the beginning of the section.

\section{Degree one quasi Q-groupoids: quasi LA-groupoids}\label{sec:1q-groupoids}
In this section we focus on $1$-graded quasi $Q$-groupoids. We show how these objects are a common generalization of LA-groupoids and $L_2$-algebroids. Also, we prove that these objects are equivalent to linear quasi Poisson structures on the corresponding dual VB-groupoids and give interesting examples that motivate their study. Based on Corollary \ref{cor:mor}, we show that these objects provide us with the appropriate Morita invariant concept of Lie algebroid. Therefore we can define a Lie algebroid in the category of differentiable stacks as the Morita class of a $1$-graded quasi $Q$-groupoid. 
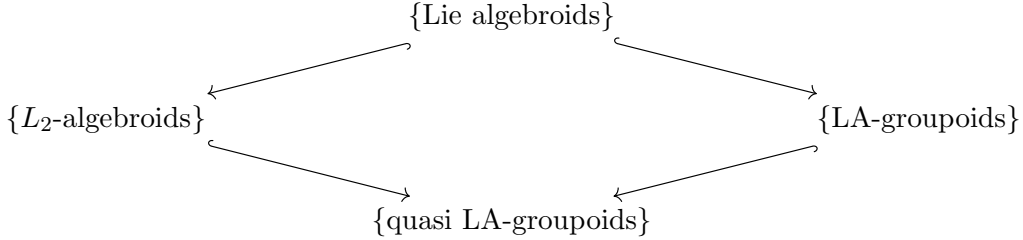
\[ \begin{tikzcd}
  &  & \{\text{Lie algebroids}\} \arrow[lld, hook] \arrow[rrd, hook] &  &              \\
\{\text{$L_2$-algebroids}\} \arrow[rrd, hook] &  &      &  & \{\text{LA-groupoids}\} \arrow[lld, hook] \\
&  & \{\text{quasi LA-groupoids}\}   &  &  
\end{tikzcd}\]

\subsection{The functor $[1]$}\label{sec:fun1}
In this section we summarize some well known facts about vector bundles and $1$-manifolds that we will use later. Some standard references are \cite{raj:tes, raj:qgr, roy:on, vai:lie}.

Denote by $\cV\cB$ the category of vector bundles and by $1$-$\cM an$ the category of degree $1$-manifolds. In Example \ref{1-man} we show that, given a vector bundle $E\to M$, one can produce the $1$-manifold $E[1]=(M,\Gamma\wedge^\bullet E^*)$. It is not difficult to see that this construction is functorial. Indeed, the functor $$[1]:\cV\cB\to 1\text{-}\cM an,\quad  E\to M\rightsquigarrow E[1]$$  is an equivalence of categories. Therefore, the study of $1$-manifolds is equivalent to the study of vector bundles. In what follows we recall how this functor can be upgraded to include Lie algebroids \cite{vai:lie} and LA-groupoids \cite{raj:tes, raj:qgr}.

\subsubsection{Vector fields on $1$-manifolds}\label{sec:vecon1}  
Recall that a \emph{ Lie algebroid} is a vector bundle $A\to M$ with a vector bundle map $\rho:A\to TM$, called the \emph{anchor}, and a \emph{bracket} $[\cdot,\cdot]:\Gamma A\times \Gamma A\to \Gamma A$ satisfying
\begin{eqnarray}
    &\left[a,b\right]=-\left[b,a\right],\quad \left[a,fb\right]=f\left[a,b\right]+\rho(a)(f)b,\label{A1}\\
    &\left[a,\left[b,c\right]\right]=\left[\left[a,b\right],c\right]+\left[b,\left[a,c\right]\right],\label{A3}
\end{eqnarray}
for all $a,b,c\in\Gamma A$ and   $f\in C^\infty(M)$. If  only the equations in \eqref{A1} are satisfied then we say that  $(A\to M, [\cdot,\cdot],\rho)$ is a \emph{ pre-Lie algebroid}, see \cite[Def. 2.1]{gra:pre} or \cite[Def. 3.1]{gra:skew} where pre-Lie algebroids are called \emph{skew algebroids}. 

A pre-Lie algebroid is the same as a linear $\pi\in\fX^2(A^*)$. On the other hand, Equation \eqref{A3} is equivalent to $[\pi,\pi]=0$ \cite{coudir}, see e.g. \cite[Thm. 3.1]{gra:skew};   the same result also shows that $Q\in\fX^{1,1}(A[1])$ determines a pre-Lie algebroid structure on $A\to M$ and equation \eqref{A3} is satisfied if and only if $Q^2=0$, see also \cite{vai:lie}. The relation between $([\cdot,\cdot],\rho)$ and $Q$ is given by the usual Chevalley-Eilenberg formula for $\omega\in C^k(A[1])$
\begin{equation}\label{cart}
    \begin{split}
    Q(\omega)(a_0,\cdots a_k)=&\sum_{i=0}^k(-1)^i \rho(a_i)(\omega(a_0,\cdots,\widehat{a}_i,\cdots))\\
    &+\sum_{i<j}(-1)^{i+j}\omega([a_i,a_j],a_0,\cdots,\widehat{a}_i,\cdots,\widehat{a}_j,\cdots, a_k).
\end{split}
\end{equation}
Thus the category of Lie algebroids and the category of degree $1$ $Q$-manifolds are equivalent \cite{vai:lie}. One can extend the above observations to multivector fields \cite[Prop. $3.7$]{quapoi}, see also \cite[$\S 6.1$]{burcab}.

\begin{prop}\label{rem:linmul} Let $A\to M$ be a vector bundle and take a multivector field on $A^*$, $\beta \in \mathfrak{X}^{n}(A^*)$. Then the following are equivalent:
\begin{enumerate} 
    \item $\beta $ is a linear multivector field,
    \item the contraction $\beta :\oplus^{n}_{A^*} T^*A^* \rightarrow \mathbb{R}$ is a vector bundle map over $\oplus^{n} A $,
    \item $\beta $ is equivalent to $\sigma_\beta\in \fX^{1,n-1}(A[1])$, which is defined by $\sigma_\beta(\omega)^\uparrow= [\beta,\omega^\uparrow]$ for $\omega\in\Gamma \wedge^\bullet A^*$, where $(\cdot)^\uparrow$ denotes the vertical lift as in \cite[Eq.(3)]{mac:cla}.
\end{enumerate} 
Moreover, the assignment $\beta\to\sigma_\beta$ satisfies $\sigma_{[\beta_1,\beta_2]}=[\sigma_{\beta_1},\sigma_{\beta_2}].$
\end{prop}

\subsubsection{LA-groupoids as $Q$-groupoids}
Here we collect some basic facts about groupoids in the  category $\cV\cB$.
Recall that a \emph{VB-groupoid $(H\rightrightarrows E; G\rightrightarrows M)$} is a diagram of groupoids and vector bundles like \eqref{VB-groupoid} such that $\gr(\gm_H)\to\gr(\gm_G)$ is a vector subbundle of $H\times H\times H\to G\times G\times G$, see e.g. \cite{raj:VB, mac:book} for equivalent definitions and other properties. 
\begin{equation}\label{VB-groupoid}
\begin{array}{c}
     \xymatrix{H\ar@<-.5ex>[r]\ar@<.5ex>[r]\ar[d]&E\ar[d]\\
 G\ar@<-.5ex>[r]\ar@<.5ex>[r]&M}
\end{array}
\end{equation}
The \emph{core} of a VB-groupoid is the vector bundle over $M$ defined by $C=\gu_G^*\ker(\gs_H)\to M$. One can show that the vector bundle $H^*\to G$ inherits a VB-groupoid structure $(H^*\rightrightarrows C^*; G\rightrightarrows M)$ with core $E^*\to M$ known as the \emph{dual VB-groupoid} \cite[$\S 11.2$]{mac:book}. As Lie groupoids differentiate to Lie algebroids, the infinitesimal counterpart of a VB-groupoid is a \emph{VB-algebroid}. We denote the VB-algebroid of \eqref{VB-groupoid} by \eqref{VB-algebroid}, we follow the convention from \cite{burcabhoy} according to which double arrows indicate a Lie algebroid structure over a manifold. VB-groupoids and VB-algebroids form a category with morphisms defined in the natural way.
\begin{equation}\label{VB-algebroid}
\begin{array}{c}
     \xymatrix{A_H\ar@{=>}[r]\ar[d]&E\ar[d]\\
 A_G\ar@{=>}[r]&M}
\end{array}
\end{equation}
In this work we will need an additional structure on the vertical bundles in the diagram \eqref{VB-groupoid}. 
\begin{defi}
$(H\rightrightarrows E; G\rightrightarrows M)$ is an \emph{LA-groupoid} if $(H\rightrightarrows E; G\rightrightarrows M)$ is a VB-groupoid and  $(H\to G, [\cdot,\cdot],\rho)$ is a Lie algebroid such that $\gr(\gm_H)\to\gr(\gm_G)$ is a Lie subalgebroid of $H\times H\times H\to G\times G\times G,$ see \cite[Def. $2.1$]{liedir}.
    
    If $(H\to G, [\cdot,\cdot],\rho)$ is just a  pre-Lie algebroid, see $\S\ref{sec:vecon1}$, then we refer $(H\rightrightarrows E; G\rightrightarrows M)$ as a \emph{pre-LA-groupoid.}
\end{defi}
LA-groupoids have a rich Lie theory, see e.g. \cite{burcabhoy}. Their double nature allows us to ask questions about integration and differentiation. Let us mention here that their infinitesimal counterparts are double Lie algebroids whereas their global counterparts are double Lie groupoids $\S\ref{sec:douqpoi}$.

Another property that we want to emphasize is the duality between LA-groupoids and PVB-groupoids displayed in diagram \eqref{1dual}. In other words, $(H\rightrightarrows E; G\rightrightarrows M)$ is an LA-groupoid if and only if $(H^*\rightrightarrows C^*; G\rightrightarrows M)$ is a \emph{PVB-groupoid}, i.e. $H^*$ has a linear and multiplicative Poisson structure.
\begin{equation}\label{1dual}
    \text{LA-groupoid}\begin{array}{c}
       \xymatrix{H\ar@<-.5ex>[r]\ar@<.5ex>[r]\ar@{=>}[d]&E\ar@{=>}[d]\\
 G\ar@<-.5ex>[r]\ar@<.5ex>[r]&M} 
    \end{array}\overset{dual}{\longleftrightarrow} \begin{array}{c}
          \xymatrix{H^*\ar@<-.5ex>[r]\ar@<.5ex>[r]\ar@{~>}[d]&C^*\ar@{~>}[d]\\
 G\ar@<-.5ex>[r]\ar@<.5ex>[r]&M} 
    \end{array}\text{PVB-groupoid.}
\end{equation}

Using the functor $[1]$ one can translate the above definitions into supergeometric concepts. This was done in \cite{raj:tes,raj:qalg, raj:qgr}, we summarize here the relevant results for us in the following proposition.

\begin{prop}[See \cite{raj:tes,raj:qalg, raj:qgr}]\label{prop:raj}
    The functor $[1]$ has the following the properties:
    \begin{enumerate}
        \item It sends the VB-groupoid $(H\rightrightarrows E; G\rightrightarrows M)$ to the  $1$-graded Lie groupoid $H[1]\rightrightarrows E[1]$.
        \item Pre-LA-groupoid structures on $(H\rightrightarrows E; G\rightrightarrows M)$  are in one to one correspondence with $Q\in\fX^{1,1}_{mul}(H[1])$.
        \item The diagram 
        \begin{equation*}
            \xymatrix{\{\text{LA-groupoids}\}\ar[d]^{Lie}\ar[r]^{[1]\qquad}&\{1\text{-graded }Q\text{-groupoids}\}\ar[d]^{Lie}\\ \{\text{Double Lie algebroids}\}\ar[r]^{[1]\quad}& \{1\text{-graded } Q\text{-algebroids}\}}
        \end{equation*}
        is commutative and the horizontal arrows are equivalences of categories.
    \end{enumerate}
\end{prop}

\subsection{Quasi LA-groupoids} \label{sec:qla}

 Motivated by Proposition \ref{prop:raj}, we can make the following definition.

\begin{defi}\label{def:qla}
    Let $(H\rightrightarrows E, G\rightrightarrows M)$ be a VB-groupoid. We say that  $(H\rightrightarrows E, G\rightrightarrows M)$ is a \emph{quasi LA-groupoid} if $H[1]\rightrightarrows E[1]$ is a quasi $Q$-groupoid.
\end{defi}

Since the quasi Q-groupoids introduced here are a generalization of Q-groupoids. We obtain that quasi LA-groupoids, i.e. $1$-graded quasi Q-groupoids, are a generalization of LA-groupoids. Some immediate consequences of the definition are the following. 
\begin{enumerate}
    \item Quasi LA-groupoids form a strict 2-category with morphisms given as in $\S\ref{sec:mor}$.
    \item Given VB-groupoids $(H^i\rightrightarrows E^i; G^i\rightrightarrows M^i)\ i\in\{1,2\},$ they are VB-Morita equivalent as defined in \cite[$\S 3$]{mat:vb} if and only if $H^1[1]\rightrightarrows E^1[1]$ and $H^2[1]\rightrightarrows E^2[1]$ are Morita equivalent as $1$-graded Lie groupoids. Thus, Corollary \ref{cor:mor} implies that quasi LA-groupoid sructures can be transported along VB-Morita equivalences. 
    \item Therefore, we can make the following definition. A \emph{Lie algebroid in the category of differentiable stacks} is the Morita class of a quasi LA-groupoid. Compare with \cite{wal}. 
\end{enumerate}

Let $(H\rightrightarrows E, G\rightrightarrows M)$ be a VB-groupoid with VB-algebroid $(A_H\Rightarrow E; A_G\Rightarrow M)$.
Recall that a quasi $Q$-groupoid structure on $H[1]\rightrightarrows E[1]$ is given by a vector field  $Q\in\fX^{1,1}_{mul}(H[1])$ and a section $\qq\in\Gamma^2 (\cA_{H[1]})=\Gamma^2 (A_H[1]\to E[1])$ satisfying 
\begin{equation}\label{qeq}
    Q^2=\qq^l-\qq^r\quad \text{and}\quad[Q,\qq^r]=0.
\end{equation}
By Proposition \ref{prop:raj}, we know that $Q$ is equivalent to a pre-LA-groupoid structure on   $(H\rightrightarrows E, G\rightrightarrows M)$, i.e. a bracket $[\cdot,\cdot]$ and an anchor $\rho$ in $H$ satisfying that the graph of the multiplication is a subalgebroid. We still have to interpret $\qq\in\Gamma^2 (A_H[1])$ and \eqref{qeq}.

Denote by $\Omega\to M$ the dual of the vector bundle whose space of sections are the double linear functions on $A_H$. It is easy to see (for example using the geometrization functor given in \cite{bur:frob}) that a section $\qq\in\Gamma^2(A_H[1])$ is equivalent to two vector bundle maps
 $$\lambda:\wedge^2E\to A_G\quad \text{and}\quad \Psi:\wedge^3  E\to \Omega$$
satisfying
\begin{equation*}
	\pr\circ\Psi(e_1\wedge e_2\wedge e_3)=\frac{1}{3}\Big(\lambda(e_1\wedge e_2)\otimes e_3-\lambda(e_1\wedge e_3)\otimes e_2+\lambda(e_2\wedge e_3)\otimes e_1\Big)
\end{equation*}
for $e_1, e_2, e_3\in E$ and $\pr:\Omega\to A_G\otimes E$.
  Equations \eqref{qeq} can be expressed in terms of $([\cdot,\cdot],\rho,\lambda,\Psi)$ but the resulting characterization is not particularly useful except in some special situations that we shall describe below. Nevertheless, we can say that $\lambda$ controls the failure of the anchor map $\rho$ to be bracket preserving whereas $\Psi$ controls the failure of the Jacobi identity for the bracket.

\subsubsection{Quasi PVB-groupoids and duality}\label{sec:qPVB}
The aim of this section is to generalize the duality  of diagram \eqref{1dual} to quasi LA-groupoids. 

	Given a double vector bundle $(D\to B;A\to M)$, consider  the vector bundle $\wedge^k D^*\to B$. We say that $\omega\in\Gamma(\wedge^k D^*)$ is \emph{linear} if  
 \begin{equation*}
 \omega(s_1,\cdots,s_k)\in C_{lin}^\infty(B),\quad \omega(c_1,s_2,\cdots,s_k)\in C^\infty(M)\quad\text{and}\quad \omega(c_1,c_2, d_3,\cdots,d_k)=0,
	\end{equation*}	 
for $s_i\in\Gamma_{lin}(D,B)$ linear sections, $c_1,c_2\in\Gamma_{cor}(D)$ core sections and $d_i\in\Gamma (D,B)$ arbitrary sections.

\begin{defi}\label{def:qPVB}
A \emph{quasi PVB-groupoid} is a VB-groupoid $(H\rightrightarrows E;G\rightrightarrows M)$ endowed with $\Pi\in\fX^2_{mul}(H)$ and $ \pi\in\Gamma\wedge^3 (A_H\to E)$  such that $\Pi$ and $\pi$ are linear sections and satisfy the equations \eqref{eq:qP}.
\end{defi}
Now we can state the main result of this section which generalizes \cite[Thm. $5.9$]{strhom}.

\begin{thm}\label{thm:main} 
Let $(H\rightrightarrows E;G\rightrightarrows M)$ be a VB-groupoid with core $C$. Then the correspondence of Proposition \ref{rem:linmul} determines the following duality
\begin{equation*}
    \text{quasi LA-groupoid}\begin{array}{c}
       \xymatrix{H\ar@<-.5ex>[r]\ar@<.5ex>[r]\ar@{==>}[d]&E\ar@{==>}[d]\\
 G\ar@<-.5ex>[r]\ar@<.5ex>[r]&M} 
    \end{array}\overset{dual}{\longleftrightarrow} \begin{array}{c}
 \xymatrix@1{H^*\ar@<-.5ex>[r]\ar@<.5ex>[r]\ar@{{}*\composite{{\aaa}}{>}}  [d]&C^*\ar@{{}*\composite{{\aaa}}{>}}[d] \\ G\ar@<-.5ex>[r]\ar@<.5ex>[r]&M} 
\end{array}\text{quasi PVB-groupoid;}
\end{equation*}
which extends the one of \eqref{1dual}.
\end{thm}

We shall prove this theorem by relying on the description of infinitesimal derivations in terms of classical differential geometry proven in Proposition \ref{pro:infder}. We give another proof of Theorem \ref{thm:main} using a different idea in Proposition \ref{pro:symdua}.

\begin{proof}[Proof of Theorem \ref{thm:main}] Let $Q$ be a multiplicative vector field of degree 1 on the graded groupoid $H[1] \rightrightarrows E[1]$, so it is equivalent to a differential of degree 1 on $\Gamma (\wedge^\bullet H^*)$. According to Proposition \ref{rem:linmul}, $Q$ corresponds to a linear bivector field $\Pi$ on $H^*$. The fact that $Q$ is multiplicative is equivalent to the multiplicativity of $\Pi$ with respect to the dual VB-groupoid $H^* \rightrightarrows C^*$. Note that the graded Jacobi identity implies that $Q^2(\cdot )^\uparrow=\frac{1}{2}[[\Pi,\Pi],\cdot^\uparrow]$. Since Proposition \ref{pro:infder} implies that any $\qq\in \Gamma^2\cA_{H[1]}=\Gamma^2(A_{H}[1])$ satisfies $\qq^r(\cdot)^\uparrow=[\pi^r,\cdot^\uparrow]$ for some linear section $\pi$, we have that the condition $Q^2=q^l-q^r$ is equivalent to $\frac{1}{2}[\Pi,\Pi]=\pi^l-\pi^r$. In such a situation, applying Proposition \ref{rem:linmul} to the right invariant derivation $[Q,q^r]$ allows us to conclude that $[Q,q^r]=0$ is equivalent to $ [\Pi,\pi^r]=0$ (notice that the commutator of derivations corresponds to the Schouten bracket of linear multivector fields).\end{proof}

 \subsubsection{Lie theory of quasi LA-groupoids}\label{subsec:liethe}
The infinitesimal counterpart of a quasi LA-groupoid is a generalization of the {\em double Lie algebroids} introduced in \cite{macdoualg}. Recall from \cite{quapoi} that a {\em Lie quasi-bialgebroid} is a triple $(A\Rightarrow M, \delta,\psi)$, where $A\Rightarrow M$ is a Lie algebroid, $\delta:\Gamma(\wedge^\bullet A)\to\Gamma(\wedge^{\bullet+1}A)$ is a degree 1 derivation and $\psi\in\Gamma(\wedge^3 A)$ satisfying:
$$\delta([\cdot,\cdot])=[\delta\cdot,\cdot]+[\cdot,\delta\cdot], \quad \delta^2(\cdot)=[\psi,\cdot]\quad \text{and}\quad \delta(\psi)=0.$$ 

 \begin{defi} 
 A \emph{quasi LA-algebroid} $((D\Rightarrow E; A\Rightarrow M),(\mathtt{a}^A,[\,,\,]^A),\pi)$ is given by a VB-algebroid $(D\Rightarrow E; A\Rightarrow M)$ with bracket and anchor $(\mathtt{a}^E,[\,,\,]^E)$, a linear pre-Lie algebroid structure $(\mathtt{a}^A,[\cdot,\cdot]^A)$ on $(D\rightarrow A; E\rightarrow M)$ and $\pi\in \Gamma (\wedge^3 D^\uparrow\to C^*)$ a linear section with the following properties:
\begin{enumerate} \item the anchor $\mathtt{a}^E$ defines a morphism of pre-Lie algebroids: $(D \rightarrow A;E \rightarrow M) \rightarrow (TE \rightarrow TM ;E \rightarrow M)$, where $TE \rightarrow TM $ is the tangent prolongation of $E \rightarrow M$,  
\item the anchor $\mathtt{a}_A$ defines a morphism of Lie algebroids $(D \rightarrow E;A \rightarrow M) \rightarrow (TA \rightarrow TM ;A \rightarrow M)$, where $TA \rightarrow TM $ is the tangent prolongation of $A \rightarrow M$,
\item there is a Lie quasi-bialgebroid structure $(D^\uparrow,\delta_D,\pi)$ on the VB-algebroid $(D^\uparrow \Rightarrow C^*;A \Rightarrow M)$ dual to $(D \Rightarrow E;A \Rightarrow M)$ see \cite[$\S 9.2$]{mac:book}, where the differential $\delta_D$ on $\wedge^\bullet D^\uparrow \rightarrow C^*$ is given by $\delta_D(\cdot)^\uparrow=[\Pi_A,\cdot^\uparrow]$ and $\Pi_A\in \mathfrak{X}^2(D^\uparrow) $ is the bivector field associated to $(\mathtt{a}^A,[\,,\,]^A)$.   
\end{enumerate}  
\end{defi}
\begin{rema}  The previous definition is inspired on \cite{macdoualg}. Alternatively, one can introduce a quasi LA-algebroid as a VB-algebroid such that applying the functor $[1]$ to the vertical bundle structures produces a quasi Q-algebroid. The equivalence between these two definitions is a long exercise, see \cite[\S 4.2.2]{raj:tes}. See \cite{vor:bia} for a similar situation.
\end{rema}
 
It is proven in \cite[$\S 1$]{macdou2}
 that an LA-groupoid differentiates to a double Lie algebroid. Similarly, we can differentiate a quasi LA-groupoid and obtain a quasi LA-algebroid. We can prove this result following the approach of \cite{macdou2}. Alternatively, we can dualize the involved structures and apply Theorem \ref{thm:main} which is the approach we prefer.

\begin{defi}
A \emph{quasi PVB-algebroid} is a VB-algebroid $(D\Rightarrow E; A\Rightarrow M)$  endowed with a linear $\Pi\in\fX^2(D)$ and a linear $\pi\in\Gamma(\wedge^3 D\to E)$ such that $(D\Rightarrow E, \sigma_\Pi,\pi)$ (as in Proposition \ref{rem:linmul}) is a Lie quasi bialgebroid.
\end{defi}

\begin{rema}
Observe that in the previous definition $TD\to D$ is a double vector bundle in two different ways $(TD\to D; TE\to E)$ and $(TD\to D; TA\to A)$ and we demand that $\pi\in \mathfrak{X}^2(D)$ is linear with respect to both structures.
\end{rema}

Some immediate consequences of the definitions given above are the following.
\begin{enumerate}
    \item The vertical dual of a quasi LA-algebroid is a quasi PVB-algebroid.
    \item Let $((H\rightrightarrows E,G\rightrightarrows M),\Pi, \pi)$ be a quasi PVB-groupoid. Then $((A_H\Rightarrow E, A_G\Rightarrow M), Lie(\Pi),\pi)$ is a quasi PVB-algebroid.
    \item As observed before, this implies the differentiation of quasi LA-groupoids.
\end{enumerate}
\begin{thm} The Lie functor establishes a bijection between the set of compatible quasi LA-groupoid structures on a source-simply connected VB-groupoid and the set of quasi LA-algebroid structures compatible with the corresponding VB-algebroid. \qed\end{thm}
Using the dualization provided by Theorem \ref{thm:main}, this result can be proven following the ideas in \cite[Prop. 5.3.1]{burcabhoy} and \cite[Thm. 5.3.5]{burcabhoy}. In fact, the same arguments that hold for linear Poisson groupoids structures carry over verbatim to linear quasi-Poisson groupoid structures.
\begin{rema}[Lie groupoids in the category of stacks]\label{rem:stalie}
    Recall that an LA-groupoid can be regarded as the Lie algebroid of a double Lie groupoid, see e.g. \cite{macdou}. Hence, quasi LA-groupoids should have a global counterpart. If we think of quasi LA-groupoids as Lie algebroids in the category of differentiable stacks, we expect that they should integrate to Lie groupoids in the category of differentiable stacks which are objects that generalize double Lie groupoids, see \cite{cc:lie2,cc:ngro} for some work in this direction. 
\end{rema}

\subsection{$L_2$-algebroids}\label{subsec:l2alg} 
In this section we extend the main result of \cite{bacr} by showing that the strict $2$-category of quasi LA-groupoids over a unit groupoid is equivalent to the strict $2$-category of $L_2$-algebroids. 

Recall that a \emph{$L_2$-algebroid}, see e.g. \cite{bon:on}, consist of two vector bundles over the same base $E\to M$ and $C\to M$ endowed with 
\begin{equation}\label{l2a}
    \rho:E\to TM,\quad \partial:C\to E, \quad [\cdot,\cdot]:\Gamma E\wedge E\to \Gamma E,\quad \nabla:\Gamma E\otimes C\to \Gamma C,\quad K: \wedge^3E\to C
\end{equation}
satisfying the Leibniz rules
\begin{equation}
    [e, fe']=f[e,e']+\rho(e)(f)e',\quad \nabla_{fe}c=f\nabla_ec,\quad \nabla_e(fc)=f\nabla_e c+\rho(e)(f)c\label{eq:LR}
\end{equation}
for $e,e'\in\Gamma E$, $c\in\Gamma C$ and $f\in C^\infty(M)$ and the integrability conditions
    \begin{align}
         \rho(\partial c)=0,\quad \partial\nabla_{e}c=[e,\partial c],\quad \nabla_{\partial c_1} c_2=-\nabla_{\partial c_2} c_1,\label{eq:L1}\\
        \partial K(e_1,e_2,e_3)=[e_1,[e_2,e_3]]-[[e_1,e_2],e_3]+[[e_1,e_3],e_2],\label{eq:L2}\\
        K(\partial c, e_1, e_2)=[\nabla_{e_1}, \nabla_{e_2}]c-\nabla_{[e_1,e_2]}c,\label{eq:L3}\\
        \big(\nabla_{e_1}K(e_2,e_3,e_4)+(-1)^i c.p.\big)=\big(K([e_1,e_2],e_3,e_4)+(-1)^i c.p.\big)+\label{eq:L4}\\
        +K([e_2,e_4],e_1,e_3)+K([e_1,e_3],e_2,e_4),\notag
    \end{align}
for $e_i\in\Gamma E$ and $c_i\in\Gamma C.$

By relying on the supergeometric viewpoint, we get a straightforward conceptual description of the $1$-morphisms between $L_2$-algebroids.  
\begin{prop}[See \cite{bon:on}]\label{pro:l2deg2}
    The functor $$[1]:\{L_2\text{-algebroids}\}\to \{\text{Degree } 2 \ Q\text{-manifolds}\}$$
    is an equivalence of categories
\end{prop}

\begin{rema}
    It is proven in \cite{lacou} that degree $2$ $Q$-manifolds are also equivalent to VB-Courant algebroids. Therefore, all the results obtained in this sections can ve re-interpret in terms of them.
\end{rema}

Thus, in the supergeometric setting we get that a $1$-morphism is a morphism of degree $2$ $Q$-manifolds while a $2$-morphism between two degree $2$ $Q$-manifold morphisms is given by a $Q$-homotopy. Using the equivalence described above, we can express morphisms between $L_2$-algebroids in a straightforward but laborious process. For the sake of clarity, we restrict here to base preserving morphisms, for the general case we refer to \cite{bon:on}. 

A {\em morphism between $L_2$-algebroids} $(\partial:C \rightarrow E,\rho,\nabla,[\cdot,\cdot],K)$ and $(\partial':C' \rightarrow E',\rho',\nabla',[\cdot,\cdot]',K')$ over the same base manifold $M$ is a triple of vector bundle morphisms $F_0:E\rightarrow E'$, $F_1:C\rightarrow C'$ and $\beta:\wedge^2E \rightarrow C'$ such that
\begin{align}
    & F_0\circ\partial=\partial'\circ F_1,\label{eq:l2mor0}\\
    & \rho'\circ F_0=\rho,   \quad F_0([e_1,e_2])=[F_0(e_1),F_0(e_2)]'+\partial'(\beta(e_1,e_2))\label{eq:l2mor1} \\
    & F_1(\nabla_e c)-\beta(e,\partial (c))=\nabla'_{F_0(e)}F_1(c) \label{eq:l2mor2}\\ 
    &F_1(K(e_1,e_2,e_3))+(\beta([e_1,e_2], e_3)+c.p.)=\notag\\
    &\qquad K'(F_0(e_1),F_0(e_2),F_0(e_3))+(\nabla'_{F_0(e_1)}\beta(e_2,e_3)+ c.p.);\label{eq:l2mor3}
\end{align}
for all $e,e_i\in \Gamma E $ and all $c\in \Gamma C$. The composition of $L_2$-algebroid morphisms is defined in the natural way. Finally, a {\em 2-morphism} $\theta$ between the two $L_2$-algebroid morphisms 

$$(F_0, F_1,\beta),(G_0, G_1,\gamma):(\partial:C \rightarrow E,\rho,\nabla,[\cdot,\cdot],K)\to(\partial':C' \rightarrow E',\rho',\nabla',[\cdot,\cdot]',K')$$ is a vector bundle morphism $\theta:E \rightarrow C'$ which is a chain homotopy between the chain complex morphisms $(F_0,F_1)$ and $(G_0,G_1)$ and is such that
\begin{align}
    \beta(e_1,e_2)-\gamma(e_1,e_2)=\nabla'_{F_0(e_1)}\theta(e_2)-\nabla'_{F_0(e_2)}\theta(e_1)-\theta([e_1,e_2]) \quad \forall e_1,e_2\in \Gamma E.
\end{align}
Thus $L_2$-algebroids form a strict $2$-category that we denote by {\bf $L_2$-Alg}. In particular, if $M$ is a point, we recover the strict 2-category of $L_2$-algebras (2-term $L_\infty$-algebras) defined in \cite[$\S 4.3$]{bacr}.

As we have pointed out, quasi LA-groupoids form a strict 2-category that we denote by {\bf qLA}. And let us denote by {\bf qLA$_{u}$} the full sub $2$-category whose objects are just the quasi LA-groupoids that are VB-groupoids over a unit groupoid. 

\begin{thm}\label{l2algqla}
    The strict $2$-categories {\bf $L_2$-Alg} and {\bf qLA$_{u}$} are $2$-equivalent.
\end{thm}

\begin{proof}
    Consider a $L_2$-algebroid given by \eqref{l2a}. Define the VB-groupoid $(C\oplus E\rightrightarrows E, M\rightrightarrows M)$ with
    \begin{equation}\label{gmaps}
        \gs_H(c,e)=e,\quad \gt_H(c,e)=e+\partial c,\quad \gu_H(e)=(0,e),\quad \gm_H((c', e+\partial c),(c,e))=( c+c',e).
    \end{equation}
    Define a bracket and an anchor on $H=C\oplus E$ by means of
    $$\rho_H(c,e)=\rho(e),\quad [(c,e), (c',e')]_H=( \nabla_e c'-\nabla_{e'}c+\nabla_{\partial c}c',[e,e'])$$
    for $(c,e),(c',e')\in\Gamma H.$ One can easily verify that $(H,[\cdot,\cdot]_H, \rho_H)$ is a pre-Lie algebroid due to \eqref{eq:LR} and that 
    $$\gr(\gm_H)\to\gr(\gm_M)\subset H\times H\times H\to M\times M\times M$$
    is a subalgebroid  due to \eqref{eq:L1}. Thus, we get that $H[1]\rightrightarrows E[1]$ has $Q\in\fX^{1,1}_{mul}(H[1])$ given by \eqref{cart}.

    Since the VB-algebroid of $(H\rightrightarrows E, M\rightrightarrows M)$ is $(C\oplus E\Rightarrow E; 0_M\Rightarrow M)$, it is straightforward to verify that $K:\wedge^3 E\to C$ determines exactly a $\qq\in\Gamma^2( \cA_{H[1]})=\Gamma^2( A_{H}[1])$. The groupoid structure on $H$ allows us to describe $\qq^r$ and $\qq^l$ explicitly as well. Given $f\in C^\infty(M), \ \theta\in C^1(H[1])=\Gamma H^*$ and $h_i=(c_i,e_i)\in\Gamma H$ for $i=0,1,2$ we get
    \begin{equation}\label{eqK}
        \begin{split}
            \qq^r(f)&=a^l(f)=0,\\
            \qq^r(\theta)(h_0,h_1,h_2)&=\langle\theta, \big( K(e_0+\partial c_0,e_1+\partial c_1, e_2+\partial c_2), 0\big)\rangle,\\
            \qq^l(\theta)(h_0,h_1,h_2)&=\langle\theta, \big( K(e_0,e_1, e_2), -\partial K(e_0,e_1, e_2)\big)\rangle.
        \end{split}
    \end{equation}
    Hence, a straightforward computation shows that equations \eqref{eq:L2} and \eqref{eq:L3} imply that $Q^2=\qq^l-\qq^r$ and equation \eqref{eq:L3} implies $[Q, \qq^r]=0.$

    Conversely, let $(H\rightrightarrows E, M\rightrightarrows M)$ a VB-groupoid with core $C$ such that $(H[1]\rightrightarrows E[1], Q, \qq)$ is a $1$-graded quasi $Q$-groupoid. The unit subbundle induces a splitting of the exact sequence
   $$0\to C\to H\xrightarrow{\gs_H} E\to 0,$$
   thus $H=C\oplus E$ and the target $\partial:=\gt_{|C}:C\to E$ is a vector bundle map. In terms of this splitting, the structure maps become the ones given in \eqref{gmaps}.
   
   By Proposition \ref{prop:raj} we have that $Q\in\fX^{1,1}_{mul}(H[1])$ is equivalent to a pre-Lie algebroid structure  $([\cdot,\cdot]_H, \rho_H)$ on $H$ satisfying 
   \begin{equation}\label{eq:gr}
     \gr(\gm_H)\to\gr(\gm_M)\subset H\times H\times H\to M\times M\times M  
   \end{equation}
   is a subalgebroid. Thus we can define
   $$[e,e']=[(0,e),(0,e')]_H,\quad \nabla_e c=[(0,e), (c,0 )]_H,\quad \rho(e)=\rho_H(0,e).$$
   This definition satisfies \eqref{eq:LR} because $\rho_H(0,c)=0$. Indeed, this follows from the fact that $(TM\rightrightarrows TM, M\rightrightarrows M)$ has zero core. Moreover \eqref{eq:gr} implies that equations \eqref{eq:L1} hold. As before, the VB-algebroid of $(H\rightrightarrows E, M\rightrightarrows M)$ is $(C\oplus E\Rightarrow E; 0_M\Rightarrow M)$. Thus $\qq\in\Gamma^2 (A_{H}[1])$ is given exactly by a bundle map $K:\wedge^3 E\to C$. Since the groupoid structure is the same as before, \eqref{eqK} still holds and then we have that $Q^2=\qq^l-\qq^r$ is equivalent to \eqref{eq:L2} and \eqref{eq:L3}, whereas $[Q, \qq^r]=0$ is equivalent to \eqref{eq:L4}.

   Let us consider the equivalence at the level of $1$-morphisms.
   Consider $L_2$-algebroids $(\partial:C \rightarrow E,\rho,\nabla,[\cdot,\cdot],K)$ and $(\partial':C' \rightarrow E',\rho',\nabla',[\cdot,\cdot]',K')$ over $M$ with respective $1$-graded quasi $Q$-groupoids $(H[1]\rightrightarrows E[1], Q, \qq)$ and $(H'[1]\rightrightarrows E[1], Q', \qq')$ as before.  In order to simplify the computations, we consider morphisms of type $(F_0, F_1, 0)$ and $(\id, \id, \beta)$, the general situation reduces to these two cases, thanks to Lemma \ref{lem:fact1mor}.
   
   For the morphism $(F_0, F_1, 0)$, equation \eqref{eq:l2mor0} implies that $$F_0:E[1]\to E'[1]\quad \text{with}\quad F_1\oplus F_0:H[1]\to H'[1]$$ defines a graded groupoid morphism. Since $\beta=0$, equations \eqref{eq:l2mor1} and \eqref{eq:l2mor2} implies that $F_1\oplus F_0$ is a $Q$-morphism, thus $\id:T[1](F_1\oplus F_0)\circ Q_{H[1]}\Rightarrow Q_{H'[1]}\circ (F_1\oplus F_0)$ is a natural transformation and it is clear that \eqref{eq:l2mor3} becomes then \eqref{eq:com-1mor}.
   
   For the morphism $(\id, \id, \beta)$, we need $H[1]=H'[1]$. By the previous discussion, we get a morphism of the form $(\id,\tau)$. Using Proposition \ref{prop:1morgau}, we need to check that $\beta$ defines a $b\in\Gamma^1(A_H[1])$ which relates the quasi $Q$-structures as a gauge transformation. Since $A_H[1]$ is $(C\oplus E\Rightarrow E; 0_M\Rightarrow M)$, one can show that $\beta:\wedge^2 E\to C$ gives exactly such a $b\in\Gamma^1( A_{H}[1])$. Notice that $A_{H}[1]$ has zero bracket and $b^r$ and $b^l$ are given by 
    \begin{equation*}
        \begin{split}
            &b^r(f)=b^l(f)=0,\quad 
            b^r(\theta)(h_0,h_1)=\langle\theta, \big(\beta(e_0+\partial c_0,e_1+\partial c_1), 0\big)\rangle,\\
            &b^l(\theta)(h_0,h_1)=\langle\theta, \big( \beta(e_0,e_1), -\partial K(e_0,e_1)\big)\rangle.
        \end{split}
    \end{equation*}
   for $f\in C^\infty(M), \ \theta\in C^1_{H[1]}=\Gamma H^*$ and $h_i=(c_i,e_i)\in\Gamma H$ with $i=0,1$. Thus equations \eqref{eq:l2mor1} and \eqref{eq:l2mor2} are equivalent to $Q'=Q+b^l-b^r$, whereas \eqref{eq:l2mor3} is equivalent to $\qq'=\qq+[Q,b^r]$.

   Conversely, given $1$-morphisms of the type $(F,\id)$ and $(\id, \tau)$ we can produce the corresponding morphisms $(F_0,F_1, 0)$ and $(\id, \id,\beta)$ just as before. The equivalence between $2$-morphisms is left to the interested reader.
\end{proof}
\begin{lem}\label{lem:fact1mor} Every morphism in {\bf qLA} can be factored as a composite of a gauge transformation and a morphism of the form $(\ast,0)$. \end{lem}
\begin{proof} Let $(F,\tau)$ be a quasi LA-groupoid morphism between the quasi $Q$-groupoids $(H[1]\rightrightarrows E[1], Q, \qq)$ and $(H'[1]\rightrightarrows E[1], Q', \qq')$ with corresponding VB groupoids $H \rightrightarrows E$, $H' \rightrightarrows E'$ with respective cores $C$ and $C'$. Consider the gauge transformation $(1,{-b}): (H[1] \times H'[1], e^b(Q \times Q',q\times q'))  \rightarrow (H[1] \times H'[1], Q \times Q',q\times q') $ defined by any extension $b$ of $\tau$:
\[ b: \wedge^2 (E\times E') \rightarrow C\times C' \, \text{with $b((e_1,F_0(e_1))\wedge (e_2,F_0(e_2))=-\tau(e_1\wedge e_2) $ for all $ e_1\wedge e_2 \in \wedge^2 E$}. \]
Then we have that $(F,\tau)=(\text{pr}_{H'[1]},0)\circ(1, {-b})\circ (i_F,0)$, where $i_F: H \rightarrow H\times H'$ is the inclusion of $H$ as the graph of $F$.   
\end{proof}

Let us point out the following consequence of Proposition \ref{pro:l2deg2} and Theorem \ref{l2algqla}. 

\begin{coro} 
The category of degree 2 $Q$-manifolds and the category of quasi LA-groupoids over a unit manifold are equivalent. \qed
\end{coro}

\subsection{The category of multiplicative sections}\label{subsec:mulsec} Let $(H \rightrightarrows E, G\rightrightarrows M)$ be a VB-groupoid with core $C$. As pointed out in \cite[$\S 3$]{ortwal}, there is a 2-term complex of vector spaces
\begin{align} \Gamma (C) \xrightarrow{\delta_H} \Gamma_{mul}(H), \quad u\mapsto u^r-u^l \quad \forall u\in \Gamma (C). \label{eq:catsec}  
\end{align} 
If $H \rightrightarrows E$ is an LA-groupoid, it was shown in \cite[$\S 7.1$]{ortwal} that the Lie bracket on sections naturally induces a crossed module structure on the above 2-term complex which then induces a strict $L_2$-algebra as in $\S$
 \ref{sec:douqpoi}. Here we show that if $H \rightrightarrows E$ is equipped with a quasi LA-groupoid structure, \eqref{eq:catsec} becomes an $L_2$-algebra. We will give a conceptually simple proof of this result by combining Theorem \ref{thm:main}  with some observations from \cite{l2algqpoi}. In the subsequent discussion, we will use the notation $[\cdot,\cdot]_2$ and $[\cdot,\cdot,\cdot]_3$ to represent the 2-bracket and 3-bracket of a Lie 2-algebra, respectively. 
 
 \begin{lem}[See \cite{l2algqpoi}]\label{lem:l2algqpoi} 
 Let $(G\rightrightarrows M,\Pi,\pi)$ be a quasi Poisson groupoid and consider the 2-term complex
 \begin{align}
     C^\infty(M) \xrightarrow{\delta} C_{mul}^\infty(G), \quad k\mapsto \mathtt{s}^*k -\mathtt{t}^*k \quad  \forall k\in C^\infty(M). \label{eq:2tercomqpoi}
 \end{align} Then the 2-bracket $[\cdot,\cdot]_2$ defined by
 \[ [f,g]_2=\{f,g\}:=\mathcal{L}_{\Pi^\sharp(df)}g, \quad \mathtt{s}^* [f,k]_2=\{f,\mathtt{s}^*k\}  \quad \forall f,g\in C_{mul}^\infty(G),\, k\in C^\infty(M) \]
 together with the 3-bracket
 \[ [f,g,h]_3= \Theta(f,g,h), \quad \text{where}\quad i_{df\wedge dg \wedge dh}\pi^r=\mathtt{t}^* \Theta(f,g,h), \quad  \forall f,g,h\in C_{mul}^\infty(G); \]
 turn \eqref{eq:2tercomqpoi} into an $L_2$-algebra.
 \end{lem}

\begin{proof} 
One can prove this directly. Alternatively, one can include $C^\infty(M) \xrightarrow{\delta} C_{mul}^\infty(G)$ into $\Omega^1(M) \xrightarrow{\delta} \Omega_{mul}^1(G)$ using the de Rham differential and then use \cite[Thm 3.1]{l2algqpoi}.
\end{proof}
 
\begin{prop}\label{pro:catsec}
Let $H \rightrightarrows E$ be a quasi LA-groupoid with core $C$. Then its corresponding 2-term complex \eqref{eq:catsec} inherits a canonical $L_2$-algebra structure. 
 \end{prop}

 \begin{proof} 
 By Theorem \ref{thm:main}, the quasi LA-groupoid structure on $H \rightrightarrows E$ is equivalent to a linear quasi-Poisson structure on the dual VB-groupoid $H^* \rightrightarrows C^*$. Now let us observe that we can embed the $2$-term complex \eqref{eq:catsec} into the complex of multiplicative functions of $H^* \rightrightarrows C^*$ by taking each section to its corresponding linear function:
 \[ \begin{tikzcd}
\Gamma(C) \arrow[dd, "\delta_H"'] \arrow[rrr, hook] &  &  & C^\infty(C^*) \arrow[dd, "\delta"] \\        &  &  &                                    \\
\Gamma_{mul}(H) \arrow[rrr, hook]                   &  &  & C^\infty_{mul}(H^*)               
\end{tikzcd}\]
as in Lemma \ref{lem:l2algqpoi}. By linearity of the quasi-Poisson structure on $H^* \rightrightarrows C^*$, the image of \eqref{eq:catsec} is an $L_2$-subalgebra. 
\end{proof}

\begin{rema}  
Naturally, the 2-bracket $[\cdot,\cdot]_2$ on $ \Gamma (C) \xrightarrow{\delta_H} \Gamma_{mul}(H)$ is determined by the pre-Lie algebroid structure given by $Q$ and the 3-bracket $[\cdot,\cdot,\cdot]_3$ is given by the homotopy $\qq$ but the argument above simplifies the proof that they determine an $L_2$-algebra structure.
\end{rema}

Recall that \cite[Thm. 7.3]{ortwal} shows how Morita equivalences of LA-groupoids induce quasi-isomorphism between the corresponding 2-term complexes as in \eqref{eq:catsec}. Here we extend that result to Morita equivalent quasi LA-groupoids by showing that their associated $L_2$-algebra structures as in Proposition \ref{pro:catsec} are quasi-isomorphic. 

\begin{prop}\label{pro:invcatsec} 
Let $(F,\tau):(H[1],Q,q) \rightarrow (\widetilde{H}[1],\widetilde{Q},\widetilde{q})$ be a quasi LA-groupoid morphism such that $F$ is a VB-Morita equivalence over $\phi:G \rightarrow \widetilde{G}$. Then it induces a quasi-isomorphism between their corresponding $L_2$-algebras.
\end{prop}
\begin{proof} A VB-Morita equivalence does not induce, in general, a morphism between the corresponding $2$-term complexes. Instead, it determines a pair of quasi-isomorphisms as follows. We can transport the quasi LA-groupoid structure $(\widetilde{H}[1],\widetilde{Q},\widetilde{q})$ to a quasi-LA groupoid structure on the pullback VB-groupoid structure over $G$ to produce the following diagram 
$$\Big(\Gamma (C) \xrightarrow{\delta_H} \Gamma_{mul}(H)\Big)\xrightarrow{F'}\Big(\Gamma (\phi^*_0\widetilde{C}) \xrightarrow{{\phi^*\delta_{\widetilde{H}}}} \Gamma_{mul}(\phi^*\widetilde{H})\Big)\xleftarrow{i}\Big(\Gamma (\widetilde{C}) \xrightarrow{\delta_{\widetilde{H}}} \Gamma_{mul}(\widetilde{H})\Big)$$

where $F'$ and $i$ are the naturally induced maps and these are quasi-isomorphisms according to \cite[Prop. 5.4]{ortwal}. 

Let us verify now the compatibility of the quasi-isomorphism described above with the brackets. We can construct $(\phi^*\widetilde{H}\rightrightarrows \phi_0^*\widetilde{E},Q',a')$ in such a way that $(F,\tau)$ factors as the composite of $(p,\widetilde{\tau}):(\phi^*\widetilde{H}\rightrightarrows \phi^*\widetilde{H},Q',a') \rightarrow (\widetilde{H}\rightrightarrows \widetilde{H},\widetilde{Q},\widetilde{a})$ and $(F',\tau')$, where $F':H \rightarrow \phi^* \widetilde{H}$ is the map induced by $F$ and $p$ is the projection. 

Let us argue that $F'$ induces an $L_2$-algebra morphism. In this case, we get a morphism of complexes
    \[ \begin{tikzcd}
\Gamma({C}) \arrow[dd, "\delta_H"'] \arrow[rrr, "F'"] &  &  & \Gamma(\phi_0^*\widetilde{C}) \arrow[dd, "\phi^*\delta_{\widetilde{H}}"] \\
&  &  &                               \\
\Gamma_{mul}({H}) \arrow[rrr, "F'"', hook]          &  &  & \Gamma_{mul}(\phi^*\widetilde{H}).         
\end{tikzcd}\]
It follows from the natural transformation condition that $\tau'$ is determined by a vector bundle map $\tau_0':\wedge^2 E\rightarrow \phi_0^*\widetilde{C}$ such that 
\begin{align*} & Q\left((F')^*(\xi)\right)-(F')^*(Q'\xi)=T^*\xi \quad \forall \xi\in\Gamma(\phi^*\widetilde{H}^*), \\
&T:\oplus^2 H\rightarrow \phi^*\widetilde{H} \quad T(u,v)=\tau_0(\mathtt{t}_H(u),\mathtt{t}_H(v))^r-\tau_0(\mathtt{s}_H(u),\mathtt{s}_H(v))^l \quad \forall u,v\in H_g \quad g\in G.
\end{align*} 
As a consequence, the map $\Upsilon:\wedge^2 \Gamma_{mul}({H}) \rightarrow  \Gamma(\phi_0^*\widetilde{C})$ defined by $\Upsilon(X,Y)=-\tau_0(\mathtt{t}_H(X),\mathtt{t}_H(Y))$ for all $X,Y\in \Gamma_{mul}({H}) $
satisfies that
\begin{align*}
    & F'([X,Y]_2)-[F'(X),F'(Y)]_2= \phi^*\delta_{\widetilde{H}}(\Upsilon(X,Y)) \\
    & F'([X,c]_2)-[F'(X),F'(c)]_2= \Upsilon(X,\phi^*\delta_H (c)) \quad \forall X,Y\in \Gamma_{mul}({H}), \, c\in \Gamma(C).
\end{align*}
On the other hand, the coherence law \eqref{eq:com-1mor} for $(F',\tau')$ implies equation
\begin{align*}
    F'([X,Y,Z]_3)+([\Upsilon(X,Y),Z]_2+\text{c.p.})=[F'(X),F'(Y),F'(Z)]_3-(\Upsilon([X,Y]_2,Z)+\text{c.p.})
\end{align*}
for all $X,Y,Z\in \Gamma_{mul}({H})$, where `c.p.' denotes cyclic permutations. As a consequence, $F'$ together with $\Upsilon$ determine an $L_2$-algebra morphism as in \cite[Def. 34]{bacr}. One can argue similarly for $(p,\widetilde{\tau})$.
\end{proof}

\subsection{Lie bialgebroids over Lie groupoids}\label{subsec:catliebia} 
Here we study the Maurer-Cartan elements of the dgla introduced in $\S\ref{sec:pq}$ in the case $n=1$; by doing so we obtain a general categorification of the concept of Lie bialgebroid which specializes to a number of constructions previously considered in the literature.   

Let $(H\rightrightarrows E, G\rightrightarrows M)$ be a VB-groupoid and denote by $(A_H\Rightarrow E, A_G\Rightarrow M)$ its VB-algebroid. As explained in $\S\ref{sec:fun1}$ the functor $[1]$ gives rise to a $1$-graded Lie groupoid $H[1]\rightrightarrows E[1]$ with $1$-graded Lie algebroid $A_{H}[1]\Rightarrow E[1]$. Consider the dgla $(\cV^\bullet_{H[1]}, d, [\cdot,\cdot])$ constructed in $\S\ref{sec:pq}$, a degree $1$ element in $\cV^\bullet_{H[1]}$ is the same as a collection
$$\Psi=\big((f,F),(q,Q),t,(\Pi,\pi),(\Phi,\phi)\big)\in \cV^1_{H[1]}$$
where $f\in C^{4}(E[1])$, $F\in C_{mul}^{3}(H[1])$, $q\in  \Gamma^2(A_H[1])$, $Q\in \fX^{1,1}_{mul}(H[1])$, $t\in \Gamma^{0}(\wedge^2A_H[1])$, $\Pi\in \fX_{mul}^{2,-1}(H[1])$, $\pi\in\Gamma^{-2}(\wedge^3A_H[1]), $ $\Phi\in  \fX_{mul}^{3,-3}(H[1])$ and $\phi\in \Gamma^{-4}(\wedge^4A_H[1]).$ Moreover, $\Psi$ is a Maurer-Cartan element, i.e.  $d\Psi+\frac{1}{2}[\Psi,\Psi]=0$, if it satisfies the following two sets of $5+6$ equations 
\begin{equation}
\def\arraystretch{1.3}
   \begin{array}{r}
      f^r-f^l+\left[Q,F\right]=0,\\
    q^r-q^l+\frac{1}{2}[Q,Q]+[\Pi,F]=0,\\
    t^r-t^l+[\Pi, Q]+[\Phi, F]=0,\\
\pi^r-\pi^l+\frac{1}{2}[\Pi,\Pi]+[Q,\Phi]=0,\\
   \phi^r-\phi^l+[\Pi,\Phi]=0,\\
    \end{array}
    \qquad\qquad
    \begin{array}{r}
         [F,q]+[Q,f]=0,\\
    \left[F,t\right]+\left[Q,q\right]+\left[\Pi,f\right]=0,\\
    \left[F,\pi \right]+\left[Q,t\right]+\left[\Pi,q\right]+\left[\Phi,f\right]=0,\\
    \left[F,\phi \right]+\left[Q,\pi \right]+\left[\Pi,t\right]+\left[\Phi,q\right]=0,\\
    \left[Q,\phi\right]+\left[\Pi,\pi \right]+\left[\Phi,t\right]=0,\\
    \left[\Pi,\phi \right]+\left[\Phi,\pi \right]=0.
    \end{array}
\end{equation}
Let us explain a way to interpret these equations. Firstly, observe that each capital letter is accompanied by a lower letter, the lower letters determine correction terms for some equations involving the capital letters. The pairs $(f,F)$ and $(\Phi,\phi)$ give the ``twisted up to homotopy" versions of the main ``up to homotopy" structures $(q,Q)$ and $(\Pi, \pi)$. The element $t$ is somewhat mysterious to us, we can view it as a homotopy for the compatibility condition.

In the first table we summarize the structures that we codify when the base groupoid $G\rightrightarrows M$ is $*\rightrightarrows*$: 
\begin{center}
    \begin{tabular}{|c|c|l|c|}
\hline
    $\Psi$ & Core & Name & Ref.  \\
    \hline
    $\big((0,0),(0,Q),0,(0,0),(0,0)\big)$ &   $0$& Lie algebra & \\ 
    $\big((0,0),(0,Q),0,(\Pi,0),(0,0)\big)$  &  0& Lie bialgebra & \\ 
    $\big((0,F),(0,Q),0,(\Pi,0),(0,0)\big)$  &  0& quasi-Lie bialgebra & \cite{kos:quasi, roy:on} \\ 
    
    $\big((0,F),(0,Q),0,(\Pi,0),(\Phi,0)\big)$  &  0& proto-bialgebra & \cite{kos:quasi, roy:on}
    \\
     \hline
    $\big((0,0),(0,Q),0,(0,0),(0,0)\big)$ &  & strict $L_2$-algebra & \cite{bacr} \\ 
    $\big((0,0),(q,Q),0,(0,0),(0,0)\big)$  &  & $L_2$-algebra & \cite{bacr}\\ 
    \hline
    $\big((0,0),(0,Q),0,(\Pi,0),(0,0)\big)$  &  & strict Lie 2-bialgebra & \cite{cc:bia, xu:2bia, kra:bia} \\ 
    $\big((0,0),(q,Q),0,(\Pi,\pi),(0,0)\big)$  &  & weak Lie 2-bialgebra & \cite{xu:2bia}\\ 
    $\big((0,0),(q,Q),t,(\Pi,\pi),(0,0)\big)$  &  & $2$-term $L_{\infty}[0,1]$-bialgebra & \cite{cc:bia}\\ 
    $\big((f,F),(q,Q),t,(\Pi,\pi),(\Phi,\phi)\big)$  &  & quasi weak Lie $2$-bialgebra & \cite[Rmk. 2.8]{xu:2bia}\\ 
    \hline
\end{tabular}
\end{center}
The above table can be reproduced by replacing the base groupoid with $M\rightrightarrows M$; at the level of terminology, this amounts to adding \emph{-oid} at the end of each name. Instead of doing that, we give directly the case of a general base:

\begin{center}
    \begin{tabular}{|c|l|c|}
    \hline
    $\Psi$ & Name & Ref.  \\
\hline
    $\big((0,0),(0,Q),0,(0,0),(0,0)\big)$  & LA-groupoid & \\ 
    $\big((0,0),(q,Q),0,(0,0),(0,0)\big)$  &quasi LA-groupoid & $\S\ref{sec:qla}$\\
    \hline
    $\big((0,0),(0,0),0,(\Pi,0),(0,0)\big)$  & PVB-groupoid & \\ 
    $\big((0,0),(0,0),0,(\Pi,\pi),(0,0)\big)$  &quasi PVB-groupoid & $\S\ref{sec:qPVB}$\\
    \hline
    $\big((0,0),(0,Q),0,(\Pi,0),(0,0)\big)$ &Lie-bialgebroid groupoid & \cite{bur:dou}\\ 
    $\big((0,0),(0,Q),0,(\Pi,\pi),(0,0)\big)$ &quasi-Poisson LA-groupoid & $\S$ \ref{sec:douqpoi}\\
    \hline $\big((0,F),(0,Q),0,(\Pi,0),(\Phi,0)\big)$ & multiplicative proto-bialgebroid & \\
    \hline
\end{tabular}
\end{center}
We will analyse some of these objects at the global level in $\S$ \ref{sec:douqpoi}.

\subsection{Examples}\label{sec:exa} As we have seen, quasi LA-groupoids include LA-groupoids and  $L_2$-algebroids as particular cases. Now we will some more concrete examples which illustrate the scope of our viewpoint.

\subsubsection{Cotangent bundles of quasi Poisson groupoids}\label{subsec:cotqpoi} It was shown in \cite{macdou} that the cotangent prolongation of a Poisson groupoid is an LA-groupoid. Here we extend that result by showing how the cotangent of a quasi Poisson groupoid inherits a quasi LA-groupoid structure. 

\begin{prop}\label{pro:cotqpoi} 
Let $(G\rightrightarrows M ,\Pi,\pi)$ be a quasi Poisson groupoid with Lie algebroid $A$. Then $T^*[1]G \rightrightarrows A^*[1]$ with $$\Pi_{can}\in\fX^{2,-1}_{mul}(T^*[1]G),\quad \Theta=\Pi\in C^2_{mul}(T^*[1]G)\quad\text{and}\quad\theta=\pi\in C^3(A^*[1])$$ is a degree $1$ symplectic quasi Q-groupoid. Moreover, any degree  $1$ symplectic quasi Q-groupoid is of this type.
\end{prop}
\begin{proof}
    Clearly $\Pi_{can}$ is a multiplicative and non-degenerate Poisson structure with bracket given by the Schouten bracket. Thus equations \eqref{eq:qP} become \eqref{eq:hommaseq}. For the moreover part recall that a degree $1$ symplectic manifold is symplectomorphic to $(T^*[1]M, \Pi_{can})$. Therefore, any degree 1 symplectic groupoid is symplectomorphic to $T^*[1]G\rightrightarrows A^*[1]$ and this concludes the proof.
\end{proof}

We can apply Theorem \ref{thm:main} to the quasi LA-groupoid determined by a quasi Poisson groupoid as in Proposition \ref{pro:cotqpoi} to conclude the following result.

\begin{coro} 
The tangent prolongation of a quasi Poisson groupoid is canonically a quasi Poisson groupoid. \qed\end{coro}

As natural as this Corollary appears, we haven't found it in the literature. Another immediate consequence of Propositions \ref{pro:catsec} and \ref{pro:cotqpoi} is the following.

\begin{coro}[\cite{l2algqpoi}]\label{cor:2terqpoi} 
Let $(G\rightrightarrows M ,\Pi,\pi)$ be a quasi Poisson groupoid, then the 2-term complex
\[ \delta:\Omega^1(M) \rightarrow \Omega^1_{mul}(G) \]
 naturally inherits an $L_2$-algebra structure. Moreover, a Morita equivalence of quasi Poisson groupoids induces a quasi-isomorphism in the corresponding 2-term complexes \qed \end{coro}
 
 \begin{rema} 
 The explicit form of the 3-bracket can be deduced by following our previous arguments 
 see \cite[Thm. 3.1]{l2algqpoi}. The $L_2$-algebra structure above can be extended to the entire spaces of differential forms $\Omega^\bullet(M) \rightarrow \Omega_{mul}^\bullet(G)$ \cite[Thm. 3.6]{l2algqpoi}. Note that our use of Theorem \ref{thm:main}  allows us to essentially reduce that construction to the case of (degree 0) functions.
 \end{rema}

\subsubsection{Central extensions and prequantum algebroids}\label{subsec:cenext}
Recall that $(\bR\rightrightarrows 0, *\rightrightarrows *)$ is a VB-groupoid with multiplication given by addition. The corresponding degree $1$-groupoid is denoted by $\bR[1]\rightrightarrows 0$. If   $(H\rightrightarrows E, G\rightrightarrows M)$ is another VB-groupoid, their sum $(H\oplus \bR_G\rightrightarrows E, G\rightrightarrows M)$ is again a VB-groupoid, in graded terms this is given by the  cartesian product $H[1]\times \bR[1]\rightrightarrows E[1]$. The next result shows one way to construct quasi $Q$-groupoid structures on this product.

\begin{prop}\label{prop:cenEx}
    Let $(H[1]\rightrightarrows E[1], Q, \qq)$ be a degree $1$ quasi $Q$-groupoid, $h\in C^2_{mul}(H[1])$ and $\xi\in C^3(E[1])$. If 
    \begin{equation}\label{eq:cocy}
        Q(h)=\gs^*_{H[1]}\xi-\gt^*_{H[1]}\xi\quad\text{and}\quad Q(\gt^*_{H[1]}\xi)=\gt^*_{H[1]}Q_0(\xi)=\qq^r(h)
    \end{equation}
    then $(H[1]\times \bR[1]\rightrightarrows E[1], \widehat{Q}, \widehat{\qq})$ is a quasi $Q$-groupoid with
    \begin{equation*}
        \begin{split}
        \widehat{Q}(h_0+\lambda h_1)=&Q(h_0)+h\wedge h_1-\lambda Q(h_1),\\
        \widehat{\qq}^r(h_0+\lambda h_1)=&\qq^r(h_0)+(\gt^*_{H[1]}\xi)\wedge h_1+\lambda \qq^r(h_1),
    \end{split}
    \end{equation*}
    where $h_0,h_1\in C^\bullet(H[1])$ and $\lambda$ is the degree $1$ coordinate function of $\bR[1].$
\end{prop}
\begin{proof}
    Observe that $\cA_{H[1]\times \bR[1]}=\cA_{H[1]}\times \bR[1]$.  If  $\lambda$ denote the degree $1$ coordinate function of $\bR[1]$ then  $\lambda^2=0$ and  a degree $k$ function in $H[1]\times \bR[1]$ is given by $h_0+\lambda h_1$ with $h_0\in C^k(H[1])$ and $h_1\in C^{k-1}(H[1]).$ Thus one can easily check that $\widehat{Q}\in \fX^{1,1}(H[1]\times \bR[1])$ and $\widehat{\qq}\in\Gamma^2\cA_{H[1]\times \bR[1]}.$ 

    Notice that $(\widehat{Q},\widehat{\qq})$ is a quasi $Q$-structure satisfying $$\widehat{Q}_{|C^k(H[1])}=Q\quad \text{and}\quad \widehat{\qq}_{|\cA_{H[1]}}=\qq.$$ Hence, to show the result  we just need to check  how $(\widehat{Q},\widehat{\qq})$ acts on  $\lambda$. The vector field $\widehat{Q}$ is multiplicative because 
    $$\gm_{H[1]\times \bR[1]}^*\widehat{Q}(\lambda)=\gm_{H[1]}^*h=\pr_1^*h+\pr_2^*h=\pr_1^*\widehat{Q}(\lambda)+\pr_2^*\widehat{Q}(\lambda)=\widehat{Q}^{(2)}\gm_{H[1]\times \bR[1]}^*\lambda$$
    due to the multiplicativity of $h$. The first equation in \eqref{eq:cocy} implies that
    \begin{equation*}
    \begin{split}
        \widehat{Q}^2(\lambda)&=Q(h)=\gs^*_{H[1]}\xi-\gt^*_{H[1]}\xi= (\widehat{\qq}^l-\widehat{\qq}^r)(\lambda).
    \end{split}
    \end{equation*}
    Finally,
    \begin{equation*}
          [\widehat{Q}, \widehat{\qq}^r](\lambda)=\widehat{Q}\big( \widehat{\qq}^r(\lambda)\big)-  \widehat{\qq}^r\big(\widehat{Q}(\lambda)\big)=Q(\gt^*_{H[1]}\xi)-\qq^r(h)=0
    \end{equation*}
    by the second equation in \eqref{eq:cocy}.
\end{proof}

\begin{coro}[Classification of central extensions]\label{cor:ex}
    Let $(H[1]\rightrightarrows E[1], Q, a)$ be a degree $1$ quasi $Q$-groupoid and $(K[1]\rightrightarrows E[1], \widehat{Q}, \widehat{\qq})$ a central extension. 
    \begin{enumerate}
        \item The choice of an splitting $\sigma$ for the exact sequence $0\to \bR[1]\to K[1]\to H[1]\to 0$ makes the central extension isomorphic to the one constructed in Proposition \ref{prop:cenEx} with $h=\sigma^*\widehat{Q}(\lambda)$ and $\gt^*_{H[1]}\xi=\sigma^*\widehat{\qq}^r(\lambda).$
        \item The central extensions defined by $(h,\xi)$ and $(h',\xi')$ are equivalent if and only if there is an element $b\in C^2(E[1])$ such that $\xi-\xi'=Q_0(b)$ and $h-h'=\gs^*_{H[1]}b-\gt^*_{H[1]}b.$
    \end{enumerate}
\end{coro}
\begin{proof}
    Observe that, since $\lambda$ is a multiplicative function, then $h=\sigma^*\widehat{Q}(\lambda)$ and $\gt^*_{H[1]}\xi=\sigma^*\widehat{\qq}^r(\lambda)$ makes sense and so $h$ is multiplicative. Let us verify the equations in \eqref{eq:cocy}
    $$Q(h)=Q(\sigma^*\widehat{Q}(\lambda))=\sigma^*\widehat{Q}^2(\lambda)=\sigma^*(\widehat{\qq}^l-\widehat{\qq}^r)(\lambda)=\gs^*_{H[1]}\xi-\gt^*_{H[1]}\xi,$$
    $$Q(\gt^*_{H[1]}\xi)=Q(\sigma^*\widehat{\qq}^r(\lambda))=\sigma^*\widehat{Q}(\widehat{\qq}^r(\lambda))=\sigma^*\widehat{\qq}^r(\widehat{Q}(\lambda))=\qq^r(\sigma^*\widehat{Q}(\lambda))=\qq^r(h),$$
    where, in the second equality of each equation, we used that $\lambda$ is central. The isomorphism part is left to the reader. 

    For the second item, note that by degree reasons a different choice of splitting gives the same $h$ and $\xi$. Thus we can suppose that the automorphism of $H[1]\times \bR[1]$ is of the form $(1,\tau).$ Therefore, Proposition \ref{prop:1morgau} implies that we get a gauge transformation for some $b\in\Gamma^1\cA_{H[1]\times \bR[1]}$. Since such section must preserve $H[1]$, we get that  $b:E[1]\to\bR[2]$ thus $b\in C^2(E[1])$. Since $\bR[1]$ is an abelian algebra we get that the gauge transformation equation becomes $\xi-\xi'=[Q_0, b]=Q_0(b)$ and $h-h'=\gs^*_{H[1]}b-\gt^*_{H[1]}b.$
\end{proof}

Some relevant examples of the above construction are given by \emph{Atiyah algebroids} of $S^1$-\emph{gerbes} and \emph{equivariant $S^1$-gerbes}. First we recall the case of an  $S^1$-principal bundles $P\to M$. Here the Atiyah algebroid is the bundle $A_P=TP/S^1$ which fits into the exact sequence
$$0\to\bR_M\to A_P\to TM\to 0.$$
Any principal connection, i.e. a splitting of the above sequence, determines a Lie algebroid isomorphism $A_P\cong TM\oplus \bR_M=A_\omega$ with bracket $$[(X,f),(Y,g)]=([X,Y], X(g)-X(f)-i_Xi_Y\omega),$$ 
where $\omega\in \Omega^2_{cl}(M)$ is the curvature of the connection.

A relevant case of the above is given by the prequantum bundle of a symplectic manifold $(M,\omega)$. Here one starts with $(M,\omega)$, and thus with $A_\omega$, and need to  find an $S^1$-principal bundle $P\to M$ and a principal connection making $A_P\cong A_\omega$. This problem was answered in \cite{kos:quan}, see also \cite{craqua}. 

In what follows we give the analogues of $A_\omega$ for closed $3$-forms on manifolds and closed $1$-shifted $2$-forms on Lie groupoids (the latter case includes the former). The link with higher connections and Atiyah algebroids will be explained in $\S$ \ref{sec:ati}.

\begin{exa}[Closed $3$-forms]
Let $M$ be a manifold, $H\in \Omega^3(M)$ with $dH=0$ and consider the LA-groupoid $(TM\rightrightarrows TM, M\rightrightarrows M)$ then $(T[1]M\rightrightarrows T[1]M, Q=d_{dr})$ with $h=0$ and $\xi=H$ satisfy the hypothesis of Proposition \ref{prop:cenEx}. Thus we get a quasi LA-groupoid on $(TM\oplus \bR_M\rightrightarrows TM, M\rightrightarrows M)$. Using Theorem \ref{l2algqla}  one can show that this is exactly the  $L_2$-algebroid in $\bR_M\to TM$ given by 
$$\rho=\id,\quad  \partial=0,\quad [X,Y]=[X,Y],\quad \nabla=0 \quad \text{and}\quad  K(X,Y,Z)=i_Xi_Yi_ZH$$
for $X,Y,Z\in\fX(M).$ 
\end{exa}

An extension of the previous example is given by \emph{closed $1$-shifted $2$-forms} on Lie groupoids.

\begin{exa}[Closed $+1$ shifted $2$-forms]\label{ex:preq}
Recall, see e.g. \cite{BCWZ, xu:mom}, that a \emph{closed  $1$-shifted $2$-form} on a Lie groupoid $G\rightrightarrows M$  is an $\omega\in\Omega^2(G)$ and an $H\in \Omega^3(M)$ satisfying 
$$\gm^*\omega=\pr_1^*\omega+\pr^*_2\omega,\quad d\omega=\gs^*H-\gt^*H\quad\text{and}\quad dH=0.$$
Then the quasi $Q$-groupoid $(T[1]G\rightrightarrows T[1]M, Q=d_{dr})$ with $h=\omega$ and $\xi=H$ satisfy the hypothesis of Proposition \ref{prop:cenEx}. Thus we get the quasi LA-groupoid
\begin{equation}\begin{array}{c}
    \label{eq:presh}
     \xymatrix{TG\oplus \bR_G\ar@<-.5ex>[r]\ar@<.5ex>[r]\ar@{==>}[d]&TM\ar@{==>}[d]\\
 G\ar@<-.5ex>[r]\ar@<.5ex>[r]&M.} 
\end{array}
\end{equation}
On one hand, this example generalizes the LA-groupoid that encodes the prequantization of a symplectic groupoid, see \cite{weixu}; on the other hand, it provides a geometric construction for general $1$-shifted symplectic groupoids that was missing in the literature. Moreover, it was shown in \cite{kre:pre, xu:qsq} that equivariant gerbes give prequantizations for $1$-shifted symplectic groupoids, thus we expect that a connection with curving should induce an isomorphism between the Atiyah algebroid of the equivariant gerbe and \eqref{eq:presh}.
\end{exa}

\begin{exa}[Gauge transformations and shifted Lagrangians]\label{ex:gaupreq}
    Let $(G\rightrightarrows M, \omega+H)$ be a Lie groupoid endowed with a closed $+1$ shifted $2$-form as in the preivous example. Now assume that in addition we get a $B\in\Omega^2(M)$ satisfying
    $$\omega=\gs^*B-\gt^*B\quad\text{and}\quad H=dB.$$
    The above equations are the ones given in Corollary \ref{cor:ex}. Hence, we get that $B$ gives a gauge transformation between the trivial extension and the prequantum bundle of $(\omega,H)$ as defined in Example \ref{ex:preq}.  

    When $(G\rightrightarrows M, \omega+H)$ is a $1$-shifted symplectic groupoid, an example of this situation is given by their \emph{Hamiltonian spaces}, see e.g. \cite{xu:mom}. Indeed, their quantization was considered in \cite{kre:pre} using relative equivariant $S^1$-gerbes. 
\end{exa}

\begin{exa}[Quasi Poisson groupoids]
    We already saw in Proposition \ref{pro:cotqpoi} that the cotangent of a quasi Poisson groupoid $(G\rightrightarrows M ,\Pi,\pi)$
  is a  quasi LA-groupoid. Now we can reinterpret  $$\Pi\in C^2_{mul}(T^*[1]G)\quad \text{and}\quad\pi\in C^3(A^*[1])$$ 
  in the light of Proposition \ref{prop:cenEx} and note that the quasi Poisson equations give us exactly the equations in \eqref{eq:cocy}. Then we get that the cotangent bundle of a quasi Poisson groupoid has a central extension that is again a quasi LA-groupoid
 \begin{equation}\begin{array}{c}
    \label{eq:preqp}
     \xymatrix{T^*G\oplus \bR_G\ar@<-.5ex>[r]\ar@<.5ex>[r]\ar@{==>}[d]&A^*\ar@{==>}[d]\\
 G\ar@<-.5ex>[r]\ar@<.5ex>[r]&M.} 
\end{array}
\end{equation}
This gives us the higher analogue of the prequantum bundle studied in \cite{hue:poi} in the case of Poisson manifolds.
\end{exa}

\subsubsection{$L_2$-algebra actions on Lie groupoids}\label{sec:act}

Recall that a \emph{$L_2$-algebra} is a tuple\footnote{For simplicity we consider $\nabla$ as part of $[\cdot,\cdot]$.} $(\fh\xrightarrow{\partial}\fg, [\cdot,\cdot],K)$ as in \eqref{l2a}  satisfying \eqref{eq:L1}-\eqref{eq:L4}. In this case, Theorem \ref{l2algqla} recovers the result in \cite{bacr} which shows that  $(\fh\to\fg, \partial, [\cdot,\cdot],K)$ is equivalent to the quasi LA-groupoid $((\fh\oplus\fg\rightrightarrows \fg; *\rightrightarrows *), Q_{CE}, q=K).$

\begin{defi}\label{def:act}
    An \emph{infinitesimal action of an $L_2$-algebra $(\fh\xrightarrow{\partial}\fg, [\cdot,\cdot],K)$ on a Lie groupoid $G\rightrightarrows M$} is an $L_\infty$-morphism $\Psi:(\fh\xrightarrow{\partial}\fg, [\cdot,\cdot],K)\to (\Gamma A\to \fX_{mul}(G), d, [\cdot,\cdot]).$ 
\end{defi}

A characterization of $L_\infty$-algebra actions on graded manifolds was given in  \cite[Theorem 4.4]{raj:act}. Along the same lines, we can give the  following characterization.

\begin{prop}
 Let $(\fh\xrightarrow{\partial}\fg, [\cdot,\cdot],K)$ be a $L_2$-algebra with associated  quasi LA-groupoid $((\fh\oplus\fg\rightrightarrows \fg; *\rightrightarrows *), Q_{CE}, q_{CE}=K)$. There is a one to one correspondence between:
 \begin{enumerate}
     \item Infinitesimal actions of a $L_2$-algebra $(\fh\xrightarrow{\partial}\fg, [\cdot,\cdot],K)$ on a Lie groupoid $G\rightrightarrows M$.
     \item Quasi LA-groupoid structures on $(\fh\oplus\fg)\times G\rightrightarrows \fg\times M$ for which the projection to $((\fh\oplus\fg\rightrightarrows \fg; *\rightrightarrows *), Q_{CE}, q_{CE})$ is a quasi LA-groupoid morphism.
 \end{enumerate}
\end{prop}

On the one hand, the above actions are a global counterpart of the actions considered in \cite{cat:poired, raj:act}. On the other hand, they are the infinitesimal counterpart of the Lie 2-group actions on differentiable stacks considered in \cite{bur:prin,rey:ham}. 

\begin{rema}
    The above Proposition allows us to define infinitesimal actions of quasi LA-groupoids on Lie groupoids by considering projectable quasi LA-groupoids on the pullback of a groupoid morphism. 
\end{rema}

Consider an $n$-graded Lie groupoid $\cG\rightrightarrows \cM$  with $n$-graded Lie algebroid $\cA\to\cM$. Then $(\Gamma^\bullet \cA\to \fX_{mul}^{1,\bullet}(\cG),[\cdot,\cdot],d)$ is a $\bZ$-graded  $L_2$-algebra as in \cite[\S 2.1]{bone:1shp}. Therefore, one can extend Definition \ref{def:act} to graded groupoids by considering actions of $\bZ$-graded  $L_2$-algebra as $L_\infty$-morphisms that preserves the $\bZ$-degree. In particular, for the $\bZ$-graded Lie algebra $(\bR[-1])$ we get the following result. 

\begin{prop}\label{prop:actR1}
   Let $\cG\rightrightarrows \cM$ be an $n$-graded Lie groupoid.  There is a one to one correspondence between:
   \begin{itemize}
       \item Infinitesimal actions of the $\bZ$-graded Lie algebra $\bR[-1]$ on $\cG\rightrightarrows \cM$.
       \item Quasi $Q$-groupoid structures on $\cG\rightrightarrows \cM$.
   \end{itemize}
\end{prop}
\begin{proof}
    First notice that $\wedge^2\bR[-1]=\sym^2\bR=\bR$. Therefore, an infinitesimal action of the $\bZ$-graded Lie algebra $\bR[-1]$ on $\cG\rightrightarrows \cM$ is the same as two linear maps
    $$\Psi:\bR\to \fX_{mul}^{1,1}(\cG)\quad\text{and}\quad  K:\bR\to \Gamma^2\cA$$
    satisfying
    $$0=\frac{1}{2}[\Psi(1),\Psi(1)]+dK(1)\quad \text{and}\quad 0=[K(1),\Psi(1)].$$
    Hence, $Q=\Psi(1)$ and $q=K(1)$ is a quasi $Q$-structure.
\end{proof}

This observation extends the characterization of $Q$-manifolds as actions of $\bR[-1]$ to the world of quasi $Q$-manifolds.

\subsubsection{2-term representations up to homotopy}\label{subsec:ruth} 
Let $A\to M$ be a Lie algebroid and let $\mathbf{V}:=V_0\xrightarrow{\partial} V_1$ be a 2-term complex of vector bundles over $M$. A representation up to homotopy of $A$ on $\mathbf{V}$ consists of a pair of $A$-connections $\nabla^i$ on $V_i$ for $i=0,1$ and a 2-form $\omega\in \Omega^2(A,\hom(V_1,V_0))$ such that 
\begin{align*}
    \partial \circ\nabla^1=\nabla^0\circ \partial, \quad \partial\circ \omega=\omega_1, \quad \omega\circ \partial=\omega_0, \quad \nabla \omega=0;
\end{align*}
where $\omega_i$ is the curvature of $\nabla^i$ and $\nabla$ is the connection determined by $\nabla^0$ and $\nabla^1$ on $\Omega^\bullet(A,\hom(V_1,V_0))$.

It was pointed out in \cite{highol, raj:2rep} that such a representation up to homotopy can be described as a morphism of $L_2$-algebroids $A\rightarrow \mathfrak{gl}({\bf V})$ where
$$\mathfrak{gl}({\bf V}):=\hom(V_1,V_0)\xrightarrow{(\partial\circ-,-\circ\partial)}\der(V_1)\times_{TM}\der(V_0)$$
with the natural two bracket and zero 3-bracket. Thanks to Theorem \ref{l2algqla}, such objects are in one to one correspondence with morphisms between the quasi LA-groupoids 
\begin{equation*}
    \xymatrix{ A\ar@<-.5ex>[d]\ar@<.5ex>[d]\ar[rr]^{0\oplus (\nabla_0,\nabla_1)\qquad\quad}&& \mathfrak{gl}_1({\bf V})\oplus \mathfrak{gl}_0({\bf V})\ar@<-.5ex>[d]\ar@<.5ex>[d] \\
    A\ar[rr]_{(\nabla_0,\nabla_1)\qquad}\ar[rru]^\omega&& \mathfrak{gl}_0({\bf V})}
\end{equation*}

We think that this interpretation makes their integration in terms of 2-functors more natural, see \cite{int2ter}.

\subsubsection{Atiyah algebroids of 2-bundles and 2-connections}\label{sec:ati} Throughout this section we will use the following notation. Let us denote by $H \rtimes K \rightrightarrows K$ a strict Lie 2-group \cite{grp}, a \emph{principal 2-bundle} is a Lie groupoid $\mathbf{P} \rightrightarrows \mathbf{Q}$ equipped with a multiplicative free and proper action of $H \rtimes K \rightrightarrows K$ such that the quotient is a manifold $M$ regarded as a unit groupoid $M \rightrightarrows M$. Such objects were identified in \cite{higauthe} as another avatar of the nonabelian bundle gerbes of \cite{nonabger, difgeoger}, see also \cite{kon:4-2bun} for alternative definitions. 

Let us start by giving some context that will serve as a motivation for our definition. A connection on a principal $G$-bundle $P$ can be defined globally as a horizontal invariant distribution but it is often more convenient to view them as a splitting of the exact sequence
$$0\to P\times_G\fg\to \mathbb{A}_P\to TM\to 0,$$
where $\mathbb{A}_P=TP/G$ is the Atiyah algebroid of $P$, see e.g. \cite[$\S$5]{mac:book}. Moreover, this point of view allow us to characterize the flat connections as the splitting that are Lie algebroid morphisms. Now we will show how, following the same approach, we can recover the notion of connection for a 2-bundle. The Atiyah algebroids of $\mathbf{P} $ and $\mathbf{Q} $ fit into an LA-groupoid 
\[ \mathbb{A}_{\mathbf{P} } \rightrightarrows \mathbb{A}_{\mathbf{Q} }.  \]
Then we have the canonical projections $\pi_i$ to the tangent of the base manifold constitute a morphism of LA-groupoids. A splitting of those projection maps compatible with the corresponding VB-groupoid structure is given by a single bundle map $F_0$ as in the following diagram
\[ \begin{tikzcd}
\mathbb{A}_{\mathbf{P} } \arrow[d, shift right] \arrow[d, shift left] \arrow[rr, "\pi_1"'] &  & TM \arrow[d, shift right] \arrow[d, shift left] \arrow[ll, "\mathtt{u}\circ F_0"', dotted, bend right] \\
\mathbb{A}_{\mathbf{Q} } \arrow[rr, "\pi_0"']                                              &  & TM \arrow[ll, "F_0"', dotted, bend right]                             
\end{tikzcd}\]

\begin{defi}\label{def:2con} 
A {\em 2-connection} on $\mathbf{P} \rightrightarrows \mathbf{Q}$ is a pair $(F,\tau_0)$, where $\tau_0:\wedge^2 TM \rightarrow C$ is a vector bundle morphism and $C$ is the core of the VB-groupoid $\mathbb{A}_{\mathbf{P} } \rightrightarrows \mathbb{A}_{\mathbf{Q} }$, such that
\begin{align}
    d F_0 +\frac{1}{2}[F_0,F_0]=\partial\tau_0; \label{eq:fakcur}
\end{align} 
where $\partial:C \rightarrow \mathbb{A}_{\mathbf{Q} }$ is the differential in the core sequence of $\mathbb{A}_{\mathbf{P} } \rightrightarrows \mathbb{A}_{\mathbf{Q} }$. \end{defi}

\begin{rema} Equation \eqref{eq:fakcur} is called the vanishing of the `fake curvature' condition \cite{difgeoger}. In principle, we could drop such a requirement for $\tau_0$ in the definition of a 2-connection (this is the approach taken in \cite{kon:con} for instance, see Remark \ref{rem:2con2}). However, this equation is necessary for being able to define two-dimensional holonomy in a reasonable sense \cite[Prop. 2.14]{higauthe}, see also \cite{surhol}.
\end{rema}

The obstruction for $(F,\tau)$ to satisfy the coherence law \eqref{eq:com-1mor} of a quasi LA-groupoid morphism is what we define as the the {\em 3-form curvature} of the 2-connection:
\begin{align}
\Omega_{(F,\tau)}:=d \tau_0 +[F_0,\tau_0]. \label{eq:3forcur}   
\end{align}

 \begin{prop} \label{prop:curv}
 The maps $(F,\tau)$ defined above determine a quasi LA-groupoid morphism 
 \[ (F,\tau):(TM \rightrightarrows TM) \rightarrow (\mathbb{A}_{\mathbf{P} } \rightrightarrows \mathbb{A}_{\mathbf{Q} } )\]
 if and only if \eqref{eq:3forcur} vanishes. 
 \end{prop}

\begin{proof} 
Since both are VB-groupoids over a unit groupoid and $TM\rightrightarrows TM$ does not have core we can use Theorem \ref{l2algqla} to conclude that a quasi LA-groupoid morphism is given by $F_0$ and a vector bundle map $\tau_0:\wedge^2 TM \rightarrow C$, where $C$ is the core of the VB-groupoid $\mathbb{A}_{\mathbf{P} } \rightrightarrows \mathbb{A}_{\mathbf{Q} }$, satisfying equations \eqref{eq:l2mor0}-\eqref{eq:l2mor3}. In this case \eqref{eq:l2mor0} and \eqref{eq:l2mor2} are automatic,  \eqref{eq:l2mor1} is equivalent to \eqref{eq:fakcur} and the vanishing of \eqref{eq:3forcur} is exactly \eqref{eq:l2mor3}.
\end{proof}

\begin{rema}\label{rem:2con2}
It follows from the previous result that a {\em flat 2-connection} is exactly a splitting of $\pi=(\pi_0,\pi_1)$ which is a quasi-LA groupoid morphism (equivalently, an $L_2$-algebroid morphism). A flat 2-connection in our sense is also flat as in \cite[Def. 5.1.5]{kon:con}. Since we demand multiplicativity in Definition \ref{def:2con}, our 2-connections are {\em fake-flat connections} in the sense of \cite[Def. 5.1.5]{kon:con}. A closely related result is \cite[Thm. 5.43]{cri-seb} where the authors work at the level of the Lie 2-algebras of multiplicative sections rather than with the LA-groupoids themselves. 
\end{rema}

Let us briefly indicate how our approach gives a unified conceptual interpretation for the definitions in \cite{higyanmil,higauthe,difger}. 

\begin{exa}[Connections on principal 2-bundles \cite{higyanmil,higauthe}]
    Rephrasing our discussion at the beginning of this subsection, let $H \rtimes K \rightrightarrows K$ be a strict Lie 2-group and consider the trivial principal  $2$-bundle over a manifold $M$ regarded as a unit groupoid. Then a 2-connection in our sense on the corresponding Atiyah algebroid recovers exactly the definition of a $2$-connection with vanishing fake curvature \eqref{eq:fakcur}.
\end{exa}

     One can also consider principal $2$-bundles whose base space is an arbitrary groupoid $G\rightrightarrows M$ instead of just the unit groupoid $M\rightrightarrows M$ of a manifold. Those objects are called \emph{principal bundle groupoids} in \cite{gar:PBG} (see also \cite{cri-seb}), we use the same notation as before. Due to Proposition \ref{prop:curv} we can extend Definition \ref{def:2con} to principal bundle groupoids by saying that a \emph{flat $2$-connection on a principal bundle groupoid} ${\mathbf{P} } \rightrightarrows \mathbf{Q}$ is a quasi LA-groupoid morphism 
     $$(F,\tau):(TG\rightrightarrows TM)\to (\mathbb{A}_{\mathbf{P} } \rightrightarrows \mathbb{A}_{\mathbf{Q} } )$$
    that is an splitting for the anchor of the Atiyah algebroid. If $(F,\tau)$ does not satisfy the coherence law \eqref{eq:com-1mor}, we say that  $(F,\tau)$ is a $2$-connection.

    A relevant example of this more general framework is the following one.

\begin{exa}[Connective structures on gerbes over differentiable stacks \cite{difger}] 
An $\mathbb{S}^1$-central extension of a Lie groupoid $G\rightrightarrows M$ can be viewed as a Lie groupoid $P \rightrightarrows M$ equipped with a free multiplicative $\mathbb{S}^1$-action:
\[ \begin{tikzcd}
P \arrow[d, "\pi"'] \arrow["\mathbb{S}^1"', loop, distance=2em, in=215, out=145] \arrow[rd, shift right] \arrow[rd, shift left] &   & {\text{$\mathtt{m}(a\cdot p,b\cdot q)=(ab)\cdot\mathtt{m}(p,q)$ for all $a,b\in\mathbb{S}^1$ and all $p,q\in P$;}} \\
G \arrow[r, shift right] \arrow[r, shift left]     & M &     
\end{tikzcd}\]
where $\pi$ is the quotient map. Hence, $P\rightrightarrows M$ gives a principal  bundle groupoid.

A \emph{connection and a curving} on such a principal bundle \cite[Def. 4.20]{difger} is a form $\theta\in \Omega^1(P)$ which is a usual principal connection together with $B\in \Omega^2(M)$ such that
$$\delta\theta=0\quad\text{and}\quad d\theta=\delta B.$$
The \emph{$3$-curvature} of the connective structure $(\theta,B)$ is $H:=dB\in\Omega^3(M)$. Following Example \ref{ex:preq}, one shows that the Atiyah algebroid of $P$ is the $\mathbb{R}$-central extension of $TG \rightrightarrows TM$ defined by the closed 1-shifted 2-form $\omega+H$ where 
$\omega=d\theta \in \Omega^2(P)$. From the discussion in Example \ref{ex:gaupreq} we get that $(F,\tau):=((\text{id}_{TP},-\theta),0)$ is a 2-connection. 
\end{exa}

\begin{rema} 
A principal 2-bundle $\mathbf{P} \rightrightarrows \mathbf{Q}$ over $M$ with respect to the Lie 2-group action of $H\rtimes K \rightrightarrows K$ can be regarded as a {\em generalized morphism} from the 2-stack represented by $M$ to the 2-stack represented by $H\rtimes K \rightrightarrows K$ \cite{g-ger}. However, it is convenient to also consider the 2-stacks represented by double Lie groupoids as possible targets for such generalized morphisms. This generalization is given by considering Lie groupoids which are principal bundles for double Lie groupoid actions instead of just Lie 2-group actions; we get in this way a more general family of multiplicative bundles, see \cite[\S 3.3]{morla}  for the definition of such multiplicative (or morphic) actions. We hope to extend our description of (flat) 2-connections to that setting in future work.
\end{rema}

\section{Degree two quasi Q-groupoids: quasi CA-groupoids}\label{sec:2q-groupoids} 

\subsection{Roytenberg-Severa correspondence}\label{sec:RS}
In this subsection we summarize some well known facts about degree $2$ manifolds that we will use later. Some standard references are \cite{lacou, roycou, roy:on, sevlet}.

The most important examples of degree 2 manifolds are arguably the ones induced by Courant algebroids. Recall that a \emph{Courant algebroid} is a vector bundle $E \to M$ equipped with a nondegenerate symmetric bilinear form $\pair{\cdot}{\cdot}$, a bundle map $\rho: E \to TM$, and a bracket $\cbrack{\cdot}{\cdot}$ such that
\begin{enumerate}
\item $\cbrack{e_1}{fe_2} = \rho(e_1)(f) e_2 + f\cbrack{e_1}{e_2}$,
\item $\rho(e_1)(\pair{e_2}{e_3}) = \pair{\cbrack{e_1}{e_2}}{e_3} + \pair{e_2}{\cbrack{e_1}{e_3}}$,
\item $\cbrack{e_1}{\cbrack{e_2}{e_3}}=\cbrack{\cbrack{e_1}{e_2}}{e_3} + \cbrack{e_2}{\cbrack{e_1}{e_3}}$,
\item $\cbrack{e_1}{e_2} + \cbrack{e_2}{e_1} = \cD \pair{e_1}{e_2}$,
\end{enumerate}
for all $f \in C^\infty(M)$ and $e_i \in \Gamma(E)$, where $\cD : C^\infty(M) \to \Gamma(E)$ is defined by
\begin{equation*}
\pair{\cD f}{e} = \rho(e)(f).
\end{equation*}

With this definition we can state  \cite[Theorem 4.5]{roy:on} as follows.

\begin{prop}
    There is a one to one correspondence between:
    \begin{enumerate}
        \item Vector bundles $E \to M$ equipped with a nondegenerate symmetric bilinear form $\pair{\cdot}{\cdot}$ and degree $2$ symplectic manifolds.
        \item Courant algebroids and degree $2$ symplectic $Q$-manifolds. 
    \end{enumerate}

\end{prop}

The above result was extended to the multiplicative case in \cite{lacou, raj:qgr,  muldir} through the following definitions.

A VB-groupoid $(H\rightrightarrows E; G\rightrightarrows M)$ has a \emph{multiplicative pairing} if the vector bundle $H\to G$ is equipped with a nondegenerate symmetric bilinear form $\pair{\cdot}{\cdot}$ and $\gr(\gm_H)\to\gr(\gm_G)$ is a lagrangian subbundle of $\overline{H}\times H\times H\to G\times G\times G$ where the $\Bar{\cdot}$ denotes minus the pairing on that factor and the inclusion $\gr(\gm_G) \hookrightarrow G \times G \times G$ is given by $(g,h)\mapsto (\gm_G(g,h),g,h)$ for all $g,h \in G$ composable.

\begin{defi}
    A VB-groupoid $(H\rightrightarrows E; G\rightrightarrows M)$ is a \emph{CA-groupoid} if  $(H\to G,\pair{\cdot}{\cdot}, \cbrack{\cdot}{\cdot},\rho)$ is a Courant algebroid such that $\gr(\gm_H)\to\gr(\gm_G)$ is a Dirac structure of $\overline{H}\times H\times H\to G\times G\times G,$ see \cite[Def. $2.1$]{liedir}.
\end{defi}

\begin{prop}\label{prop:CAg}
    There is a one-to-one canonical correspondence between:
    \begin{enumerate}
        \item VB-groupoids endowed with a multiplicative pairing and degree $2$ symplectic groupoids. 
        \item CA-groupoids and degree $2$ symplectic $Q$-groupoids. 
    \end{enumerate}
\end{prop}

\subsection{Degree 2 symplectic quasi Q-groupoids}

Let $((H\rightrightarrows E; G\rightrightarrows M), \pair{\cdot}{\cdot})$ be a VB-groupoid endowed with a multiplicative pairing and denote by $(\cG\rightrightarrows\cM, \{\cdot,\cdot\})$ the corresponding degree $2$ symplectic groupoid. Motivated by Proposition \ref{prop:CAg} we can introduce the following definition.

\begin{defi}\label{def:qCA}
    We say that $((H\rightrightarrows E; G\rightrightarrows M), \langle\cdot,\cdot\rangle)$ carries a \emph{quasi CA-groupoid} structure if $(\cG\rightrightarrows\cM, \{\cdot,\cdot\})$ is a degree 2 symplectic quasi $Q$-groupoid for some $\Theta\in C^3_{mul}(\cG)$ and $\theta\in C^4(\cM)$.
\end{defi}

As in the case of quasi LA-groupoids we will not give a classical description of them. Instead we will give several examples that motivate our definition.

Based on \cite{liuweixu}, it is natural to ask how to characterize the doubles of the categorifed Lie bialgebroids described in $\S\ref{subsec:catliebia}$. Using Definition \ref{def:qCA} together with Proposition \ref{prop:cot} we get a satisfactory answer that we summarize in the next example.

\begin{exa}[Doubles of categorified Lie bialgebroids]\label{ex:doubles}
   Let $(H\rightrightarrows E; G\rightrightarrows M)$ be a VB-groupoid with core $C$ and denote by $H[1]\rightrightarrows E[1]$ the associated $1$-graded Lie groupoid. As explained in $\S\ref{subsec:catliebia},$ a categorified Lie bialgebroid structure (in its most general form) corresponds to a Maurer-Cartan element
   $$\Psi=\big((f,F),(q,Q),t,(\Pi,\pi),(\Phi,\phi)\big)\in \cV^1_{H[1]}.$$
   Then Proposition \ref{prop:cot} implies that $(T^*[2]H[1]\rightrightarrows\cA_{H[1]}^*[2],\{\cdot,\cdot\}_{can})$ with 
   $$\Theta=F+Q+\Pi+\Phi\in C^3_{mul}(T^*[2]H[1])\quad\text{and}\quad \theta=f+q+t+\pi+\phi\in C^4(\cA^*_{H[1]}[2])$$
   is a degree $2$ symplectic quasi $Q$-groupoid. It is well known, see e.g. \cite{muldir}, that the degree $2$ symplectic groupoid  $(T^*[2]H[1]\rightrightarrows\cA_{H[1]}^*[2],\{\cdot,\cdot\}_{can})$ corresponds to the VB-groupoid $(H\oplus H^*\rightrightarrows E\oplus C^*; G\rightrightarrows M)$ and the canonical pairing is multiplicative. Combining this fact with our results we get that this VB-groupoid carries a quasi CA-groupoid structure. When $\theta=0$ we get a genuine CA-groupoid,  for an extensive study of this case and also of its infinitesimal counterpart see \cite{kar:ca}. 
\end{exa}

An interesting consequence of the above example is that we can give a more conceptual proof of Theorem \ref{thm:main}; in fact, this is a generalization of that result.

\begin{prop}\label{pro:symdua}
    Let $(H\rightrightarrows E; G\rightrightarrows M)$ be a VB-groupoid and denote by $H[1]\rightrightarrows E[1]$ the associated $1$-graded Lie groupoid. Then 
    $$(T^*[2]H[1]\rightrightarrows\cA_{H[1]}^*[2],\{\cdot,\cdot\}_{can})\quad\text{and}\quad (T^*[2]H^*[1]\rightrightarrows\cA_{H^*[1]}^*[2],\{\cdot,\cdot\}_{can})$$
    are symplectomorphic groupoids. Moreover, the isomorphism sends
    $$\Psi=\big((f,F),(q,Q),t,(\Pi,\pi),(\Phi,\phi)\big)\rightsquigarrow\Psi'=\big((\phi,\Phi),(\pi, \Pi),t,(Q, q),(F,f)\big).$$
\end{prop}
\begin{proof}
    This is a straightforward graded version of \cite[Thm. 11.5.14]{mac:book}.
\end{proof}

In $\S\ref{sec:exa}$ we already saw some examples of categorified Lie bialgebroids. Therefore Example \ref{ex:doubles} produces the doubles for these objects. Another relevant example is the following.

\begin{exa}
    Let $G\rightrightarrows M$ be a groupoid and consider the Q-groupoid $(T[1]G\rightrightarrows T[1]M, Q=d_{dr})$. A categorified Lie bialgebroid given by the Maurer-Cartan element $$\Psi=\big((f,F),(0,Q),0,(0,0),(0,0)\big)\in \cV^1_{T[1]G}$$is the same as a closed $1$-shifted 3-form on the groupoid, i.e. $F\in\Omega^3(G)$ and $f\in\Omega^4(M)$ such that
    $$\gm^*F=\pr_1^*F+\pr_2^*F,\quad dF=\gs^*f-\gt^*f\quad\text{and}\quad df=0.$$ 
    The corresponding double is a quasi CA-groupoid on
    $$TG\oplus T^*G\rightrightarrows TG\oplus A^*$$
    with the canonical pairing and anchor, Courant-Dorfman bracket twisted by the $3$-form $F$ and twisting $4$-form $\gs^*f-\gt^*f$, as in \cite{stro:4form}.
\end{exa}

Following the work of one of us, one can upgrade the previous example to show that  all the exact twisted CA-groupoids give examples of quasi CA-groupoids. 

\begin{exa}[Exact twisted CA-groupoids]\label{ex:etca}
Let $G\rightrightarrows M$ be a Lie groupoid and pick a closed 2-shifted 2-form on it, i.e. an $\omega\in \Omega^2(G^{(2)}),$ a  $F\in \Omega^3(G)$ and a $f\in \Omega^4(M)$ satisfying
\[\quad \delta \omega=0,\quad d\omega=-\delta F, \quad dF=\delta f,\quad df=0. \]
    Then we can use $\omega$ to twist the canonical VB-groupoid structure on $TG\oplus T^*G \rightrightarrows TM \oplus A^*$ by defining the new multiplication $\gm_\omega$ as follows: 
    \begin{align*} \text{Graph}(\mathtt{m}_{\omega }):=\left\{ (X\oplus \alpha ,Y\oplus \beta ,Z\oplus \gamma) \in \overline{\mathbb{T}G}\times\mathbb{T}G\times \mathbb{T}G\left| \begin{aligned} 
&\mathtt{m}^*_{G} \alpha =\text{pr}_1^*\beta + \text{pr}_2^*\gamma +i_{(Y,Z)}\omega \\
& X=T \mathtt{m}(Y,Z), \end{aligned}   \right.\right\}.
    \end{align*} 
    Now the Courant-Dorfman on $\mathbb{T}G$ can be twisted by $F$ to produce a twisted CA-groupoid structure with the $4$-form twist given by $f$. It was shown in \cite{tracou} that all the exact twisted CA-groupoids are of this form. The function $\Theta=F+d_{dr}\in C^3(T^*[2]T[1]G)$ is multiplicative with respect to the multiplication $\gm_\omega$ and, since $\theta=f\in C^4(T[1]M)\subset C^4(T[1]A^*[1])$, the following equations hold:
    $$\{\Theta,\Theta\}_{can}=\gs^* \theta-\gt^*\theta=\gs^*_\omega \theta-\gt^*_\omega\theta\quad\text{and}\quad \{\Theta,\gt_\omega^*\theta\}_{can}=\{\Theta,\gt^*\theta\}_{can}=0$$
    because ${\gs_\omega^*}|_{C^4(T[1]M)}=\gs^*|_{C^4(T[1]M)}$ and the same for the target. Therefore all these exact twisted CA-groupoids produce examples of quasi CA-groupoids. 
\end{exa}

\begin{exa}[Central extensions]
    Let $(\cG\rightrightarrows\cM,\{\cdot,\cdot\}, \Theta,\theta)$ be a degree $2$ symplectic quasi $Q$-groupoid. Then one can construct an $\bR[2]$ central extension, $\cG\times \bR[2]\rightrightarrows\cM$, that is a degree $2$ quasi $Q$-groupoid with 
    $$\widehat{Q}(f)=Q(f),\quad  \widehat{Q}(\lambda)=\Theta,\quad  \widehat{q}^r(f)=q^r(f)\quad\text{and}\quad \widehat{q}^r(\lambda)=\gt^*\theta,$$
    where $f\in C^\bullet(\cG)$ and $\lambda\in C^2(\bR[2])$ is the generator. The interested reader can verify that this is indeed a quasi $Q$-groupoid by following the same approach outlined in the proof of Proposition \ref{prop:cenEx}.
\end{exa}

\section{Epilogue: Lie and Courant algebroids through the stacky glasses}

In the course of this work we have described several structures on (graded) Lie groupoids and we have shown that they can be transported along (and compared using) Morita equivalences. This allows us to re-interpret our previous results in terms of differentiable stacks.   

Recall that a \emph{differentiable stack} is a stack $\mathfrak{m}$, i.e. a category fibered in groupoids satisfying sheaf-like properties (see \cite[Def. 2.10]{difger}), that admits a presentation, i.e. a manifold $M$ equipped with a surjective representable submersion $m:M\to\mathfrak{m}$. 

It is well known that there exists an equivalence of 2-categories between the 2-category of differentiable stacks and the 2-category of Lie groupoids up to Morita equivalence, see e.g. \cite{dorpro,difger} for the precise mathematical formulation. Roughly speaking the correspondence works as follows. For a given differentiable stack $\mathfrak{m}$ one can take a presentation $m:M\to\mathfrak{m}$ and then show that $M\times_{\mathfrak{m}}M\rightrightarrows M$ is a Lie groupoid. Moreover, different choices of presentation give Morita equivalent groupoids. Conversely, given a Lie groupoid $G\rightrightarrows M$ one shows that the category of $G$-Principal bundles $\mathfrak{m}$ is a differentiable stack with a presentation given by the quotient morphism $m:M\to\mathfrak{m}$ corresponding to the universal trivial $G$-principal bundle. 

Let $G\rightrightarrows M$ be a Lie groupoid with corresponding  differentiable stack $\mathfrak{m}$. The equivalence described above allows us to define structures on $\mathfrak{m}$ using a Lie groupoid representative $G\rightrightarrows M$. For example, it is shown in \cite{hep:vec, ortwal} that \emph{vector fields on $\mathfrak{m}$} are modeled by the dgla $$L_G^\bullet=(\Gamma A\to \fX_{mul}(G), d, [\cdot,\cdot]).$$
Recently, that work was extended to \emph{$1$-shifted multivector fields on $\mathfrak{m}$} by introducing the dgla $(V^\bullet_G, d, [\cdot,\cdot])$ given in \eqref{eq:dglaV}, see \cite{bone:1shp}. 

Recall that a \emph{homological vector field} on a degree $n$ manifold $\cM$ is a Maurer-Cartan element on $(\fX^{1,\bullet}(\cM),[\cdot,\cdot]).$ In $\S\ref{sec:qqgpd}$ we introduced \emph{$n$-graded differentiable stacks} $\mathfrak{M}$, i.e. differentiable stacks in the category of degree $n$ manifolds, as Morita equivalence classes of $n$-graded Lie groupoids $\cG\rightrightarrows\cM$. The results mentioned in the previous paragraph show that vector fields on $\mathfrak{M}$ are modeled by the dgla $(L_\cG^\bullet, d, [\cdot,\cdot])$ (see Proposition \ref{prop:dgla}) and thus we can define a \emph{homological vector field on $\mathfrak{M}$} as the class of a Maurer-Cartan element in $L_\cG^\bullet$. These Maurer-Cartan elements are, by Proposition \ref{prop:MC}, the same as quasi $Q$-structures (Definition \ref{def:qqgpd}); on the other hand, by Proposition \ref{prop:actR1} these are the same as actions of $\bR[-1]$ on $\mathfrak{M}$. Finally, when it comes to Poisson structures, we argue in $\S\ref{sec:pq}$ that, by replacing $L^\bullet_\cG$ with $\cV^\bullet_\cG$, one gets the notion of a $PQ$-structure on $\mathfrak{M}$.

Using the fact that the category of $1$-manifolds is equivalent to the category of vector bundles, in $\S\ref{sec:1q-groupoids}$ we study differentiable stacks in the category of vector bundles. It was explained in \cite{mat:vb} (see also \cite{ raj:VB}) that the VB-Morita equivalence class of a VB-groupoid can be used to define a \emph{ $2$-vector bundle\footnote{Compare with \cite{wal} where descent is used to define vector bundles over differentiable stacks.} over a differentiable stack}. 
The two key definitions of $\S\ref{sec:1q-groupoids}$ are:
\begin{itemize}
    \item A \emph{$1$-shifted Lie algebroid structure on a 2-vector bundle over a differentiable stack} is the Morita class of a quasi LA-groupoid as introduced in Definition \ref{def:qla}. 
    \item A \emph{linear $1$-shifted Poisson structure on a 2-vector bundle over a differentiable stack} is the Morita class of a quasi PVB-groupoid as introduced in Definition \ref{def:qPVB}. 
\end{itemize}

Theorem \ref{thm:main} shows that the $1$-dual of a $1$-shifted Lie algebroid on a 2-vector bundle over a differentiable stack is a linear $1$-shifted Poisson structure. In the light of the above result,  $\S\ref{subsec:catliebia}$ treats the concept of \emph{$1$-shifted Lie bialgebroid}. 

Theorem \ref{l2algqla} proves that a $1$-shifted Lie algebroid on a 2-vector bundle over a manifold is a $L_2$-algebroid up to $L_\infty$-quasi isomorphism. In the terminology of \cite{cc:ngro}, one can call these objects \emph{stacky Lie algebroids}, we expect that they are the infinitesimal counterpart of the stacky groupoids introduced in \cite{cc:ngro}.

With this language, the examples at the end of $\S\ref{sec:1q-groupoids}$ acquire the following interpretation:
\begin{enumerate}
    \item In $\S\ref{subsec:cotqpoi}$ we show that the ($1$-shifted) cotangent of a $1$-shifted Poisson structure on a differentiable stack admits a canonical $1$-shifted Lie bialgebroid structure.
    \item In $\S\ref{subsec:cenext}$ we exhibit the prequantum $1$-shifted Lie algebroids of $1$-shifted symplectic and Poisson structures on differentiable stacks.
    \item In $\S\ref{sec:act}$ we construct the $1$-shifted action Lie algebroid associated to the infinitesimal action of a stacky Lie algebra on a differentiable stack.
    \item In $\S\ref{subsec:ruth}$ we prove that representations up to homotopy of a Lie algebroid $A$ on a  2-term complex ${\bf V}$ are in 1-1 correspondence with $1$-shifted Lie algebroid morphisms from $A$ to $\mathfrak{gl}({\bf V})$.
    \item In $\S\ref{sec:ati}$ we demonstrate that flat $2$-connections on a principal 2-bundle over a manifold $\mathfrak{p}\to M$ are the same as  $1$-shifted Lie algebroid morphisms from $TM$ to the $1$-shifted Atiyah algebroid $\mathbb{A}_\mathfrak{p}.$
\end{enumerate}

Finally, $\S\ref{sec:2q-groupoids}$ is devoted to the study of  \emph{Courant algebroids over differentiable stacks}. In particular, examples of such objects are given by doubles of $1$-shifted Lie bialgebroids as in Example \ref{ex:doubles} and by Exact twisted CA-groupoids as in Example \ref{ex:etca}.

 Observe that in our Definition \ref{def:qCA} of a quasi CA-groupoid we have a genuine symplectic structure. This means that quasi CA-groupoid structures cannot be transported along Morita equivalences. The way of restoring Morita invariance into the picture, and thus defining a Courant algebroid over a differentiable stack, is by considering a shifted symplectic structure in the following way which is completely analogous to the situation of quasi-symplectic groupoids \cite{BCWZ, xu:mom}. A \emph{shifted Courant algebroid over a differentiable stack} is the Morita class of a degree $2$ quasi Q-groupoid $(\cG\rightrightarrows\cM, Q, q)$ together with a $1$-shifted $2$-form of degree $2$, i.e. a pair of differential forms $\omega_1\in\Omega^{2,2}(\cG)$ and $\omega_0\in\Omega^{3,2}(\cM)$, satisfying the following properties:
    \begin{enumerate}
        \item closedness: $\gm^*\omega_1=\pr_1^*\omega_1+\pr_2^*\omega_1,\quad d\omega_1=\gs^*\omega_0-\gt^*\omega_0$ and $d\omega_0=0;$
        \item non-degeneracy: the form $\omega_1$ is homologically non-degenerate by inducing a quasi-isomorphism between the tangent $\cA\xrightarrow{\rho} T\cM$ and cotangent $T^*[2]\cM\xrightarrow{\rho^*} \cA^*[2]$ complex of the graded Lie groupoid.
        \item $(q,Q)$ invariance: $\cL_Q\omega_1=0$, $\cL_Q\gt^*\omega_0=0$ and $\cL_{q^r}\omega_1=0$.
    \end{enumerate}
    Note that, since $\omega_1$ is not closed, we cannot conclude that $Q$ is a hamiltonian vector field for some degree $3$ function on $\cG$. Another observation is that, in this situation, $\cA$  can be viewed as a ``Dirac structure" in $T\cM\oplus T^*[2]\cM$ with bracket twisted by the $3$-form $\omega_0$, see \cite{vys:ggg} for the graded version of Courant algebroids.

\appendix 
\section{Double quasi Poisson groupoids}\label{sec:douqpoi}

Following \cite{browmac,macdou} we say that a \emph{double topological groupoid} $(D; H ,G; M)$ is a groupoid object in the category of topological groupoids and is represented by a diagram of the form
\[ \xymatrix{ D\ar@<-.5ex>[r] \ar@<.5ex>[r]\ar@<-.5ex>[d] \ar@<.5ex>[d]& H \ar@<-.5ex>[d] \ar@<.5ex>[d] \\  G \ar@<-.5ex>[r] \ar@<.5ex>[r] & M. }\]  
A {\em double Lie groupoid} is a double topological groupoid as in the previous diagram such that each of the side groupoids is a Lie groupoid and the double source map $(\mathtt{s}^H,\mathtt{s}^G):  D \rightarrow H  \times_M {G} $ is a submersion.

\begin{exa} 
Let $G \rightrightarrows M$ be a Lie groupoid. Then the pair groupoid $G \times G$ is a double Lie groupoid with sides $M \times M$ and $G$ over $M$. 
\end{exa} 

\begin{exa} \label{ex:2g}
A \emph{crossed module} of Lie groups $(G,H, \partial,\alpha)$, see e.g. \cite[$\S$ 8.4]{grp}, gives rise to a double Lie groupoid with $D=H\times G $ with sides $G$ and a point. Here $D $ and $G$ are also Lie groups, thus crossed modules can be equivalently described as a group object in the category of Lie groupoids. We refer to  them as \emph{strict Lie $2$-groups}. 

It is straightforward to see that the Lie functor establishes an equivalence of categories between the categories of 1-connected strict Lie 2-groups and strict Lie 2-algebras. As a (well-known) consequence of Theorem \ref{l2algqla}, we get that strict Lie 2-algebras are equivalent to {\em crossed modules of Lie algebras}.
\end{exa} 

If $(D; H ,G; M)$ is a double Lie groupoid, we denote by $A^{h}_D$ the Lie algebroid of $D \rightrightarrows H$ and by $A^v_D $ the Lie algebroid of $D \rightrightarrows G$. By functoriality, $(A^h_D \rightrightarrows  A_{G }; H\rightrightarrows M)$ and $(A^v_D \rightrightarrows  A_{H }; G\rightrightarrows M)$ are LA-groupoids, see \cite{macdou} for details. For a given $X\in \Gamma (\wedge^\bullet A^h_D)$, its horizontal left and right invariant extensions are denoted by $X^{\mathcal{L} }$ and by $X^{\mathcal{R}}$; while for $Y\in \Gamma (\wedge^\bullet A^v_D)$, its vertical left and right invariant extensions are denote by $Y^L$ and by $Y^R$ .

\begin{defi} 
A {\em double quasi Poisson groupoid} $((D; H ,G; M), \Pi,\Psi)$ is a double Lie groupoid $(D; H ,G; M) $ which is equipped with a bivector field $\Pi\in \mathfrak{X}^2(D )$ and $\Psi\in \Gamma (\wedge^3 A^h_D)$ satisfying:
\begin{itemize} 
    \item $\Pi$ is multiplicative over both $H $ and $G $,
    \item $[\Pi,\Pi]=\Psi^{\mathcal{L} }-\Psi^{\mathcal{R} }$
    \item $\Psi$ is multiplicative with respect to the VB-groupoid $(\bigoplus^3 A^h_D  \rightrightarrows \bigoplus^3 A_{G };H \rightrightarrows M)$: in other words, the contraction map
\[ \overline{\Psi}: \left(\bigoplus^3 (A^h_D)^* \rightrightarrows \bigoplus^3 C^*;H \rightrightarrows M\right) \rightarrow \mathbb{R} \] is a groupoid morphism, where $C$ is the core of the VB-groupoid $(A^h_D \rightrightarrows  A_{G }; H\rightrightarrows M)$. 
\end{itemize}
\end{defi}

\begin{rema} Notice that the fact that $\Psi$ is multiplicative implies immediately that $\Psi^{\mathcal{L} }-\Psi^{\mathcal{R} }$ is also vertically multiplicative. \end{rema} 

\begin{exa} A double quasi Poisson groupoid $((D; H ,G; M), \Pi,\Psi)$ in which $D \rightrightarrows {G} $ and $H \rightrightarrows M$ are vector bundles and $\Pi$ and $\Psi$ are linear  is a quasi PVB-groupoid as in Definition \ref{def:qPVB}. 
\end{exa}  

\begin{exa} 
A double quasi Poisson groupoid $((D; H ,G; M) , \Pi,0)$ in which $\Pi$ is nondegenerate is a {\em double symplectic groupoid} \cite{luwei2}.  
\end{exa}

\begin{exa} Strict Lie 2-groups as in Example \ref{ex:2g}  can be equipped with double quasi Poisson structures in two different ways. 
\begin{enumerate}
    \item A double quasi-Poisson groupoid of the form $((D; G,* ; *) , \Pi,\Psi)$ is a {\em quasi Poisson Lie 2-group} in the sense of \cite{poi2gro}; if $\Psi=0$, we get a {\em Poisson 2-group}.
    \item A double quasi-Poisson groupoid of the form $((D; *, G;*) , \Pi,\Psi)$ can be called a {\em vertical quasi Poisson 2-group} in order to distinguish this notion from the one described in the previous example. Notice that a Poisson 2-group is also a vertical quasi Poisson 2-group with respect to its group structure.
\end{enumerate}   \end{exa} 
Many other examples of double quasi Poisson groupoids arise when considering moduli spaces of representations of fundamental groupoids, see \cite{dan:poi}.

\subsection{Quasi Poisson LA-groupoids and multiplicative Lie quasi-bialgebroids} Differentiating vertically a double quasi Poisson groupoid we obtain a quasi Poisson LA-groupoid as in $\S$ \ref{subsec:catliebia}. 

\begin{prop}\label{thm:qpla} 
The vertical LA-groupoid of a double quasi Poisson groupoid inherits a canonical quasi Poisson LA-groupoid structure.   
\end{prop}

\begin{proof} 
Let $((D; H ,G; M) , \Pi,\Psi)$ be a double quasi Poisson groupoid. Since $\Pi$ is multiplicative over $G$, the formula $$\delta(X)^R=[\Pi,X^R]\quad \text{for all}\quad X\in \Gamma (\wedge^\bullet A^v_D)$$ defines a degree 1 multiplicative derivation $\delta:\Gamma (\wedge^\bullet A^v_D) \rightarrow \Gamma (\wedge^{\bullet+1} A^v_D)$ of the Schouten bracket on $\Gamma (\wedge^\bullet A^v_D)$ determined by its LA-groupoid structure \cite{burcab}. 

Note that a quasi Poisson LA-groupoid can be equivalently described as a VB-groupoid $(A^v_D \rightrightarrows A_H; G \rightrightarrows M)$ endowed with an LA-groupoid structure $((A^v_D \rightrightarrows A_H; G \rightrightarrows M),\mathtt{a}_H,[\,,\,])$ and a quasi LA-groupoid structure $\big(((A^v_D)^* \rightrightarrows C^*; G \rightrightarrows M),\delta,\psi \big)$ such that $\delta$ is a derivation of $[\,,\,]$. So it only remains to show that $\delta$ can be extended to a quasi LA-groupoid structure.

Note that the linear bivector field determined by $\delta$ as in Proposition \ref{rem:linmul} is a quasi Poisson bivector with respect to the linear section in $\wedge^3 \text{Lie} (A^v_D)$ determined by $\text{Lie} (\Psi)$. Therefore, Theorem \ref{thm:main} implies the result.
\end{proof}
 
\begin{rema} 
Note that the horizontal LA-groupoid of a double quasi Poisson groupoid is automatically a multiplicative Lie quasi-bialgebroid. This means that quasi Poisson LA-groupoids and multiplicative Lie quasi-bialgebroids are in duality as LA-groupoids. 
\end{rema}

\subsection{Integration of quasi Poisson LA-groupoids over a unit groupoid} 
Theorem \ref{thm:qpla} allows to get a quasi Poisson LA-groupoid out of a double quasi Poisson groupoid; the inverse operation involves the integration problem for LA-groupoids which is still an open question, see \cite{morla}. For a certain class of quasi Poisson LA-groupoids however we can obtain an integration result which generalizes \cite[Thm. 3.2]{poi2gro}.

\begin{thm}\label{thm:qpla2} 
A quasi Poisson LA-groupoid over a unit groupoid which consists of integrable Lie algebroids is integrable by a double quasi Poisson groupoid. 
\end{thm}

This result follows from a more fundamental observation which is itself an easy consequence of \cite[Corollary 5.4]{intinfact}.

\begin{prop}\label{thm:qpla3}
    An LA-groupoid over a groupoid consisting only of units is integrable by a double Lie groupoid if its Lie algebroid of arrows is integrable.
\end{prop}    

\begin{proof} 
Let $(A \rightrightarrows E; M \rightrightarrows M) $ be an LA-groupoid and suppose that $A$ is integrable, then so are $E$ and the core $C$ since they are Lie subalgebroids of $A$. We have that $A$ is isomorphic to the semidirect product of Lie algebroids $C \rtimes E$. Let $H$ be the source-simply-connected integration of $E$ and let $K $ be the source-simply-connected integration of $C$. Then the action of $E$ on $C$ integrates to a Lie groupoid action of $H$ on $K$ along the identity on $M$ and the associated semidirect product groupoid $D :=K \rtimes H \rightrightarrows M$ integrates $A$ \cite[Corollary 5.4]{intinfact}. So we have that the action groupoid structure $D\rightrightarrows H $ constitutes a double Lie groupoid $(D ;H, M; M)$ integrating $(A \rightrightarrows E;M \rightrightarrows M)$. 
\end{proof} 

\begin{proof}[Proof of Theorem \ref{thm:qpla2}] This follows from observing that the double Lie groupoid produced by Proposition \ref{thm:qpla3} is both horizontally and vertically source-simply-connected so one can apply to it the standard integration result of \cite{quapoi}. 
\end{proof}

\section{Infinitesimal derivations and linear sections}
Here we include the proof of a technical result that we need to give a classical proof of Theorem \ref{thm:main}.

\begin{prop}\label{pro:infder} Let $(H\rightrightarrows E;G\rightrightarrows M)$ be a VB-groupoid with core $C$. Denote by $H[1]\rightrightarrows E[1]$ the associated $1$-graded groupoid  and pick  $\qq\in \Gamma^2\cA_{H[1]}=\Gamma^2(A_{H}[1])$. Then, there is a linear section $\pi$ of $\wedge^3 A_{H^*}\to C^*$ such that 
\begin{align}  
\qq^r(\omega)^\uparrow=[\pi^r,\omega^\uparrow]\qquad\forall\omega\in\Gamma\wedge^\bullet H^*. \label{eq:infder} 
\end{align}  
Conversely, a linear section $\pi$ of $\wedge^3 A_{H^*}\rightarrow  C^*$ induces a right-invariant derivation on $H[1]\rightrightarrows E[1]$ by means of formula \eqref{eq:infder}. 
\end{prop}  

\begin{proof} 
Let $\qq\in \Gamma^2\cA_{H[1]}$ and consider its right invariant extension $\qq^r\in\fX^{1,2}(H[1])$.  Proposition \ref{rem:linmul} tells us that there is a corresponding linear $\overline{\pi}\in\fX^3(H^*)$. We have to prove that $\overline{\pi}$ factorizes as in the following diagram
\begin{align}  
\begin{array}{c}
      \xymatrix{\bigoplus^3 T^*H^* \ar[r]^{\quad\overline{\pi}}\ar[d]_{\bigoplus^3 \gt_{T^*H^*}} & \mathbb{R} \\ \bigoplus^3 A_{H^*}^* \ar[ur]_\pi } \label{eq:rinv0}
\end{array}
 \end{align}
for some linear section $\pi$ in $\wedge^3 A_{H^*}\to C^*$. 

{\em Step 1: $\overline{\pi}$ is tangent to the $\mathtt{s}_{H^*}$-fibers}. We shall show that 
$$
\overline{\pi}(\alpha,\cdot,\cdot)=0\quad  \text{for all}\quad\alpha \in \ker \mathtt{t}_{T^*H^*} =(\ker T \mathtt{s}_{H^*})^\circ.
$$ 
Let $p:H^* \rightarrow G$ be the vector bundle projection.  By \eqref{eq:R} we get that $\qq^r (\mathtt{s}_G^*(f))=0$ for any $f\in C^\infty(M)$. Thus using the definition of $\overline{\pi}$ we have
\begin{equation}\label{eq:aux1}
    \overline{\pi}(d p^*\mathtt{s}_G^*f,\cdot,\cdot)=[\overline{\pi},(\mathtt{s}_G^*f)^\uparrow](\cdot,\cdot)=\qq^r (\mathtt{s}_G^*f)^\uparrow(\cdot,\cdot)=0.
\end{equation}
Take $\theta \in \Gamma H^*$, $c\in \Gamma C$ and $h_i\in \Gamma H$ for $i=1,2$ and denote their linear functions by $\ell_{h_i}$. Linearity of $\overline{\pi}$ implies  that $\overline{\pi}(d\ell_{c^l},d\ell_{h_1},d\ell_{h_2})=\ell_{\Phi(c^l,h_1,h_2)}$ for some $\Phi(c^l,h_1,h_2)\in  \Gamma H$. Thus, we have that
\begin{align} 
&\qq^r( \theta )^\uparrow(c^l,h_1,h_2)= [\overline{\pi}, \theta^\uparrow](d\ell_{c^l},d\ell_{h_1},d\ell_{h_2})  =\langle  \theta ,\Phi(c^l,h_1,h_2) \rangle+ \label{eq:infder2} \\
&+\overline{\pi}(d p^* \langle  \theta ,c^l \rangle ,d\ell_{h_1},d\ell_{h_2})
 +\overline{\pi}(d p^* \langle \theta ,h_1 \rangle ,d\ell_{h_2},d\ell_{c^l})
+ \overline{\pi}(d p^* \langle \theta ,h_2 \rangle ,d\ell_{c^l},d\ell_{h_1}). \notag 
\end{align} 
Recall that $s_{H^*}^* \ell_c=\ell_{c^l}$ and that $H$ fits into the exact sequences
\begin{align} \xymatrix{0 \ar[r] & \mathtt{t}_G^*C \ar[r] &H \ar[r]^{(p,\mathtt{s}_H)}  & \mathtt{s}_G^*E \ar[r] &0}, \label{eq:rcore} \\
\xymatrix{0 \ar[r] & \mathtt{s}_G^*C \ar[r] &H \ar[r]^{(p,\mathtt{t}_H)}  & \mathtt{t}_G^*E \ar[r] &0}, \label{eq:lcore}
\end{align}  
 see  e.g.\cite{macdou2}.  Then we have two possibilities:
 \begin{enumerate}
     \item $\theta=\mathtt{s}_G^* \alpha$ for $\alpha \in \Gamma E^*$, using the dual of \eqref{eq:rcore}.  In this case  $\langle  \theta ,c^l \rangle=-\mathtt{s}_G^* \langle \alpha ,\mathtt{t}_H(c) \rangle $ and thus the term $\overline{\pi}(d p^* \langle \theta ,c^l \rangle ,d\ell_{h_1},d\ell_{h_2})$ vanishes for all $h_i$ by \eqref{eq:aux1}. Let us choose a splitting of the first sequence $\sigma: \mathtt{s}_G^*E \rightarrow H$. So we can assume that $h_i$ lies either in $\ker \mathtt{s}_H$ or in the image of $\sigma$; in the former case $\langle \mathtt{s}_H^* \alpha ,h_2 \rangle=0$ and in the latter case we have that $\langle \mathtt{s}_H^* \alpha ,h_2 \rangle$ is the pullback by $\mathtt{s}_G$ of a function on $M$ so the last two terms of \eqref{eq:infder2} vanish again by \eqref{eq:aux1}. As a consequence, plugging the identity $T[2]\gs_{H^*}\circ \qq^r=0$ in \eqref{eq:infder2} we get that $\Phi(c^l,h_1,h_2)$ lies in $\ker \mathtt{s}_H$. Hence $\overline{\pi}$ is tangent to the $\mathtt{s}_{H^*}$-fibers.
     \item $\theta$ lies in $\sigma'(\mathtt{t}_G^*E)^\circ$, for  $\sigma':\mathtt{t}_G^*E \rightarrow H$ a splitting of \eqref{eq:lcore}, and thus it is $\mathtt{s}_{H^*}$-related to $\theta_0\in \Gamma C^*$. In this case $\overline{\pi}(d p^* \langle  \theta ,c^l \rangle ,d\ell_{h_1},d\ell_{h_2})$ vanish again by \eqref{eq:aux1}. If $h_i$ is of the form $v^l$ for some $v\in \Gamma C$, then $\langle \theta, v^l \rangle =\mathtt{s}_G^* \langle \theta_0,v \rangle $ and so the terms $\overline{\pi}(d p^*\langle \theta, h_i \rangle , \cdot,\cdot) $ vanish. If $h_i$ lies in $\sigma'(\mathtt{t}_G^*E)$ then $\langle \theta, h_i \rangle =0$; in either case we get then that all the terms in the last row of \eqref{eq:infder2} vanish. 

     Now let us view the contraction $i_{c^l}:\Gamma (\wedge^\bullet H^* )\rightarrow \Gamma (\wedge^{\bullet-1} H^* )$ as a graded vector field on $H[1]$. We have that $[\qq^r,i_{c^l}]=0$ since right-invariant vector fields commute with left-invariant vector fields. So we get that 
\[ \qq^r(\theta)({c^l},\cdot,\cdot)=\qq^r \langle \theta, c^l \rangle(\cdot,\cdot) =(\qq^r \circ\mathtt{s}_G^*) \langle \theta_0,c \rangle (\cdot,\cdot)=0 \]
     and hence $\langle  \theta ,\Phi(c^l,h_1,h_2) \rangle$ vanishes as well. Since we can choose $\sigma'(\mathtt{t}_G^*E)$ to be complementary to $\ker \mathtt{s}_H$, it follows that $\Phi(c^l,h_1,h_2)=0$ and so we get $\overline{\pi}(d\ell_{c^l},d\ell_{h_1},d\ell_{h_2})=0$. 
 \end{enumerate}
 Finally, take $f\in C^\infty(G)$ arbitrary. Since $[\qq^r,i_{c^l}]=0$ we get that 
\[ \overline{\pi}(dp^*f,d\ell_{c^l},d\ell_{h})=(i_{c^l}\circ \qq^r)(f)(h)=(\qq^r\circ i_{c^l})(f)(h)=0 \] for all $h\in \Gamma H$. Therefore, $\overline{\pi}(d\ell_{c^l},\cdot,\cdot)=0$. 

{\em Step 2: $\overline{\pi}$ is right-invariant.} Since $\overline{\pi}$ is linear it is enough to check it on three linear sections and on one basic and two linear. The right invariance of $\qq^r$ action on functions of $H[1]$ is equivalent to
\begin{align} 
(\qq^r,0)( \mathtt{m}_{H[1]}^*\omega)=\mathtt{m}_{H[1]}^*  \qq^r(\omega)\qquad \omega\in C^\bullet(H[1])=\Gamma\wedge^\bullet H^*.\notag
\end{align}    
 Take $f\in C^\infty(G)$ and take $(h_i, k_i)\in \Gamma H^{(2)}$ for $i=1,2$. By definition of $\overline{\pi}$ we have that 
\begin{align} 
\mathtt{m}_{H^*}^* \overline{\pi}(dp^*f,d\ell_{\mathtt{m}_H(h_1,k_1)}, d\ell_{\mathtt{m}_H(h_2,k_2)})=&(\mathtt{m}_{H[1]}^*\qq^r(f))((h_1,k_1),(h_2,k_2)) \label{eq:aux2} \\ =&\big((\qq^r,0) \mathtt{m}_{H[1]}^*f\big)((h_1,k_1),(h_2,k_2))\notag\\
=&(\overline{\pi},0)(d \mathtt{m}_G^*f,d\ell_{(h_1,k_1)},d\ell_{(h_2,k_2)})\notag\\
=& (\overline{\pi},0)( \mathtt{m}_{H^*}^*(d p^*f) ,\mathtt{m}_{H^*}^*d\ell_{\mathtt{m}_H (h_1,k_1)},\mathtt{m}_{H^*}^*d\ell_{\mathtt{m}_H (h_2,k_2)})\notag
\end{align}
because $\mathtt{m}_{H^*}^*\ell_{\mathtt{m}_H(X_1,Y_1)}=\ell_{(X_1,Y_1)}.$

Take $\theta\in \Gamma H^*$ and $(h_i, k_i)\in \Gamma H^{(2)}$ for $i=1,2,3$. We have that
\begin{align} 
\mathtt{m}^*_{H^*} [\overline{\pi}, \theta^\uparrow](d\ell_{\mathtt{m}_H (h_1,k_1)},d\ell_{\mathtt{m}_H (h_2,k_2)},d\ell_{\mathtt{m}_H (h_3,k_3)})=&\mathtt{m}_{H^*}^* p^* Q 
\end{align}
where $Q\in C^\infty(G)$ is given by the formula
\begin{align*}
     Q=&\langle  \theta ,\Phi(\mathtt{m}_H (h_1,k_1),\mathtt{m}_H (h_2,k_2),\mathtt{m}_H (h_3,k_3)) \rangle
+\overline{\pi}(d p^* \langle  \theta ,\mathtt{m}_H (h_1,k_1) \rangle ,d\ell_{\mathtt{m}_H (h_2,k_2)},d\ell_{\mathtt{m}_H (h_3,k_3)})\\
&+ \overline{\pi}(d p^* \langle  \theta ,\mathtt{m}_H (h_3,k_3) \rangle ,dl_{\mathtt{m}_H (h_1,k_1)},d\ell_{\mathtt{m}_H (h_2,k_2)})
+ \overline{\pi}(d p^* \langle  \theta ,\mathtt{m}_H (h_2,k_2) \rangle ,d\ell_{\mathtt{m}_H (h_3,k_3)},d\ell_{\mathtt{m}_H (h_1,k_1)}) 
\end{align*}
with $\overline{\pi}(d\ell_{\mathtt{m}_H (h_1,k_1)},d\ell_{\mathtt{m}_H (h_2,k_2)},d\ell_{\mathtt{m}_H (h_3,k_3)})=\ell_{\Phi(\mathtt{m}_H (h_1,k_1),\mathtt{m}_H (h_2,k_2),\mathtt{m}_H (h_3,k_3)) }$. Thus, the multiplicativity of $a^r$ and the definition of $\overline{\pi}$ imply together that 
\begin{align}
    \mathtt{m}_{H^*}^* p^* Q =&\mathtt{m}^*_{H^*} [\overline{\pi}, \theta^\uparrow](d\ell_{\mathtt{m}_H (h_1,k_1)},d\ell_{\mathtt{m}_H (h_2,k_2)},d\ell_{\mathtt{m}_H (h_3,k_3)})\label{eq:aux3}\\
    =&(\mathtt{m}_{H[1]}^*\qq^r( \theta ))((h_1,k_1),(h_2,k_2),(h_3,k_3))\notag\\
    =&\big((\qq^r,0)\mathtt{m}_{H[1]}^* \theta \big)((h_1,k_1),(h_2,k_2),(h_3,k_3))\notag\\
    =& [(\overline{\pi},0), (\mathtt{m}_{H^*}^*\theta)^\uparrow]((d\ell_{h_1},d\ell_{k_1}),(dl_{ h_2},d\ell_{k_2}),(dl_{ h_3},d\ell_{k_3}))\notag \\
    =&(p,p)^*\langle \mathtt{m}_{H}^* \theta ,(\Phi,0)( (h_1,k_1), (h_2,k_2), (h_3,k_3))\rangle\notag\\
    &+(\overline{\pi},0)(d p^* \langle \mathtt{m}_{H}^* \theta , (h_1,k_1) \rangle ,dl_{ (h_2,k_2)},dl_{(h_3,k_3)}) \notag \\
& +(\overline{\pi},0)(d p^* \langle \mathtt{m}_{H}^* \theta , (h_3,k_3) \rangle ,dl_{ (h_1,k_1)},dl_{ (h_2,k_2)})\notag\\
& +(\overline{\pi},0)(d p^* \langle  \mathtt{m}_{H}^*\theta , (h_2,k_2) \rangle ,dl_{ (h_3,k_3)},dl_{ (h_1,k_1)}).\notag
\end{align}
So from the specific form of $Q$ and equation \eqref{eq:aux2} it follows that
\begin{align*} \mathtt{m}_G^* \langle  \theta ,\Phi(\mathtt{m}_H (h_1,k_1),\mathtt{m}_H (h_2,k_2),\mathtt{m}_H (h_3,k_3)) \rangle=\langle \mathtt{m}_{H}^* \theta ,(\Phi,0)( (h_1,k_1), (h_2,k_2), (h_3,k_3)) \rangle, \end{align*} 
hence we get 
\begin{align*}
    &\mathtt{m}_{H^*}^* \overline{\pi}(d\ell_{\mathtt{m}_H(h_1,k_1)},d\ell_{\mathtt{m}_H(h_2,k_2)},d\ell_{\mathtt{m}_H(h_3,k_3)})(  \mathtt{m}_{H}^*\theta)=\mathtt{m}_{H^*}^*\ell_{\Phi(\mathtt{m}_H (h_1,k_1),\mathtt{m}_H (h_2,k_2),\mathtt{m}_H (h_3,k_3))}(\mathtt{m}_{H}^* \theta)\\
    &\quad =\mathtt{m}_G^* \langle  \theta ,\Phi(\mathtt{m}_H (h_1,k_1),\mathtt{m}_H (h_2,k_2),\mathtt{m}_H (h_3,k_3)) \rangle\\
    &\quad =\langle \mathtt{m}_{H}^* \theta ,(\Phi,0)( (h_1,k_1), (h_2,k_2), (h_3,k_3)) \rangle\\
    &\quad =(\overline{\pi},0)(d\mathtt{m}_{H^*}^*\ell_{(h_1,k_1)},d\mathtt{m}_{H^*}^*\ell_{(h_2,k_2)},d\mathtt{m}_{H^*}^*\ell_{(h_3,k_3)})(  \mathtt{m}_{H}^*\theta)
\end{align*}
proving that $\overline{\pi}$ is right invariant. So it is of the form $\overline{\pi}=\pi^r$ for some $\pi\in \Gamma (\wedge^3 A_{H^*})$. 

{\em Step 3: $\pi$ is linear.} Let us consider the restriction map 
\[ r:\left(\bigoplus^3 T^*H^* \rightrightarrows \bigoplus^3 A_{H^*}^*;G \rightrightarrows M \right) \rightarrow \left(\bigoplus^3 H \rightrightarrows \bigoplus^3 E;G \rightrightarrows M \right) \]
defined by $\alpha \mapsto \alpha|_{\ker Tp}\in p^*H$ for all $\alpha \in T^*H^*$. We have that $r$ is the projection associated to a vector bundle structure and the fact that $\overline{\pi}$ is linear means that the contraction $\overline{\pi}:\bigoplus^3 T^*H^* \rightarrow \mathbb{R}$ is a vector bundle map with respect to this structure \cite[\S 6]{burcab} but then so is its restriction to the vector bundle $ \bigoplus^3 A_{H^*}^* \rightarrow \bigoplus^3 E$ and this means that $\pi$ is linear.

{\em The derivation $\qq^r$ defined by \eqref{eq:infder} is right-invariant.} Suppose that we have a linear section $\pi\in \Gamma (\wedge^3 A_{H^*}\to C^*)$. Then the 3-vector field $\overline{\pi}=\pi^r$ defined by diagram \eqref{eq:rinv0} is also linear. Since 
\[ \overline{\pi}(d p^*\mathtt{s}_G^*f,\cdot,\cdot)=[\pi^r,p^*f](\cdot,\cdot), \]
we have that $\qq^r (\mathtt{s}_G^*f)=0$. On the other hand, the equation $\overline{\pi}(d\ell_{c^l},\cdot,\cdot)=0$ applied to \eqref{eq:infder2} implies that $\qq^r\circ \mathtt{s}_H^* \alpha  =0 $ for all $\alpha \in \Gamma E^*$. As a consequence, $T[2]\gs_{H[1]}\circ\qq^r=0$. Since $\pi^r$ is right-invariant, we just need to use equations \eqref{eq:aux2} and \eqref{eq:aux3} starting from the equalities in the middle to their endpoints, as this follows the same approach that we used before, the details are left to the reader.
\end{proof}

\bibliographystyle{plain}
\bibliography{reff}
\end{document}